\documentclass[preprint,11pt]{elsarticle}
\usepackage[margin=1in]{geometry}
\usepackage{amssymb}
\usepackage{lineno}
\usepackage{hyperref}
\usepackage{subfigure}
\usepackage{float}
\usepackage{xcolor}
\usepackage{amsmath, amssymb, amsfonts, amsthm}
\usepackage{makecell}
\usepackage{multirow}
\usepackage{bm}
\usepackage{enumitem}
\usepackage[ruled,linesnumbered]{algorithm2e}
\usepackage{soul}

\setlist[itemize,1]{leftmargin=0.8in}

\DeclareMathOperator*{\argmax}{arg\,max}
\DeclareMathOperator*{\argmin}{arg\,min}
\newcommand{\bx}{{\mathbf{x}}}

\newtheorem{thm}{Theorem}[section]

\newtheorem{lem}[thm]{Lemma}
\theoremstyle{remark}
\newtheorem{rem}[thm]{Remark}

\newcommand{\ie}{{\it i.e.}}

\graphicspath{{./figs/}} 

\journal{}

\begin{document}

\begin{frontmatter}



\title{Iterative Thresholding Methods for Longest Minimal Length Partitions}


\author[a]{Shilong Hu}
\ead{shilonghu@link.cuhk.edu.cn}
\author[b]{Hao Liu}
\ead{haoliu@hkbu.edu.hk}

\author[a,c]{Dong Wang}
\ead{wangdong@cuhk.edu.cn}
\address[a]{School of Science and Engineering, The Chinese University of Hong Kong, Shenzhen, Guangdong 518172, China}
\address[b]{Department of Mathematics, Hong Kong Baptist University, Kowloon Tong, Hong Kong}
\address[c]{Shenzhen International Center for Industrial and Applied Mathematics, Shenzhen Research Institute of Big Data, Guangdong 518172, China}

\begin{abstract}
In this paper, we introduce two iterative methods for longest minimal length partition problem, which asks whether the disc (ball) is the set maximizing the total perimeter of the shortest partition that divides the total region into sub-regions with given volume proportions, under a volume constraint. The objective functional is approximated by a short-time heat flow using indicator functions of regions and Gaussian convolution. The problem is then represented as a constrained max-min optimization problem. Auction dynamics is used to find the shortest partition in a fixed region, and threshold dynamics is used to update the region. Numerical experiments in two-dimensional and three-dimensional cases are shown with different numbers of partitions, unequal volume proportions, and different initial shapes. The results of both methods are consistent with the conjecture that the disc in two dimensions and the ball in three dimensions are the solution of the longest minimal length partition problem. 
\end{abstract}


% \begin{highlights}
% \end{highlights}


\end{frontmatter}

%\linenumbers

\section{Introduction}
\label{sec:intro}
The fencing problem has a long history and is rooted in the division of land using fences. One of the earliest studies in this field can be traced back to \cite{auerbach1938probleme}, which investigated the scenario of a cylinder floating in water. This problem was further explored in subsequent studies such as \cite{radziszewski1956cordes,eggleston1961maximal,goodey2006area}, in which one considers dividing a convex set into two equal-area sets and studies the optimal arc length of the fence.

In the 1950s, P{\'o}lya raised a question in \cite{polya1958aufgabe}: Among all convex sets in $\mathbb{R}^2$ with a fixed area, which one has the longest shortest arc that bisects the area? This question remained unresolved throughout the last century.  According to \cite{croft1991unsolved}, Santal{\'o} made a conjecture that the disc is the answer. Until 2012, the conjecture was proven in \cite{Esposito2012}. A similar conjecture in high dimensions was proposed in \cite{wichiramala2007efficient} and \cite{morgan2010} after P{\'o}lya. It considers the case of dividing a region in $\mathbb{R}^d$ space into two sub-regions of arbitrary given volume proportions, and asks whether the least length of arc to divide a ball is greater than the least length of arc to divide any other convex body of the same volume. The new conjecture generalizes the bisection problem to arbitrary volume constraints. In the literature, some analytical results have been presented in \cite{berry2017convex} and \cite{wang2024note}, which offer partial proofs for this conjecture. As of now, the conjecture remains unanswered for the case of three-dimensional space and for large fraction areas in the two-dimensional case (except for bisection).

%In 1950s, P{\'o}lya raised a question in \cite{polya1958aufgabe}: in the class of planar convex sets having fixed area, which one maximizes the length of the shortest area-bisecting arc. This question stayed unsolved in the last century. According to \cite{croft1991unsolved}, Santal{\'o} made a conjecture that the disc is the answer. Until 2012, the authors in \cite{Esposito2012} proved the conjecture and finished the question. This conjecture was restated in \cite{wichiramala2007efficient} and \cite{morgan2010}, asking whether the least perimeter to enclose given volume inside an open ball in $\mathbb{R}^{d}$ is greater than inside any other convex body of the same volume. The new guess generalizes the bisection problem to arbitrary volume constraints. Several analytical results have been given in \cite{berry2017convex} and \cite{wang2024note}, which prove the conjecture partially. Now the conjecture remains unanswered for 3d case and large fraction areas in 2d case (except for bisection).

In \cite{bogosel2022longest}, the authors extend the conjecture even further. They investigate whether the disc in two dimensions and the ball in three dimensions still provide the longest minimal length partitions when the region is divided into more than two parts. The mathematical description of it is given below.

Given $\Omega\subset \mathbb{R}^{d}$ as an open, connected region with Lipschitz boundary and the number of partition $n>1$, $(\omega_{1},\dots,\omega_{n})$ is said to be a partition of $\Omega$ if $\bigcup_{i=1}^{n} \omega_{i}=\Omega$ and $\omega_{i} \cap \omega_{j}=\emptyset$ for any $1\leq i < j\leq n$. Given a vector $\mathbf{c}=(c_1,\dots,c_n)\in \mathbb{R}^{n}$ with $c_{i}>0$ and $\sum_{i=1}^{n}c_i=1$, the shortest partition of $\Omega$ with volume constraint $\mathbf{c}$ is defined by
\begin{align}
     SP(\Omega,\mathbf{c})=\argmin_{(\omega_i)_{i=1}^n}\{\sum_{i=1}^{n}\text{Per}_{\Omega}(\omega_i):(\omega_i)_{i=1}^n\text{ is a partition of }\Omega,\ |\omega_i|=c_i|\Omega|\},\nonumber
\end{align}
where $|\cdot|$ is the Lebesgue measure, and $\text{Per}_{\Omega}(\omega_i)$ is the relative perimeter of $\omega_i$ on $\Omega$.
The relative isoperimetric profile of the partition is defined as
\begin{align}
    PI(\Omega,\mathbf{c})=\min\{\sum_{i=1}^{n}\text{Per}_{\Omega}(\omega_i):(\omega_i)_{i=1}^n\text{ is a partition of }\Omega,\ |\omega_i|=c_i|\Omega|\}.\nonumber
\end{align}
Then the longest minimal length partitions problem is formulated by
\begin{align}
    \max_{|\Omega|=V} PI(\Omega,\mathbf{c}).\label{longest_minimal_length_partitions}
\end{align}
When $n=2$, problem \eqref{longest_minimal_length_partitions} degenerates into the original fencing problem. If we further set the dimension $d=2$ and $\mathbf{c}=(0.5, 0.5)$, the problem is exactly the same one as P{\'o}lya raised.

\begin{figure}[t!]
    \centering
    \subfigure[disc]{\includegraphics[width=0.24\textwidth, clip, trim = 4cm 2.5cm 3cm 1.5cm]{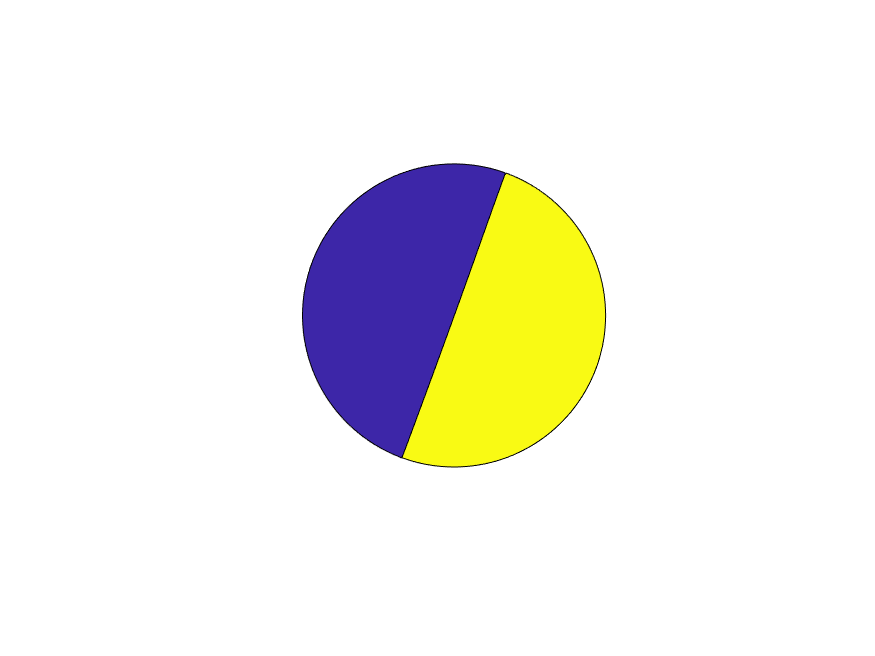}}
    \subfigure[square]{\includegraphics[width=0.24\textwidth, clip, trim = 4cm 2.5cm 3cm 1.5cm]{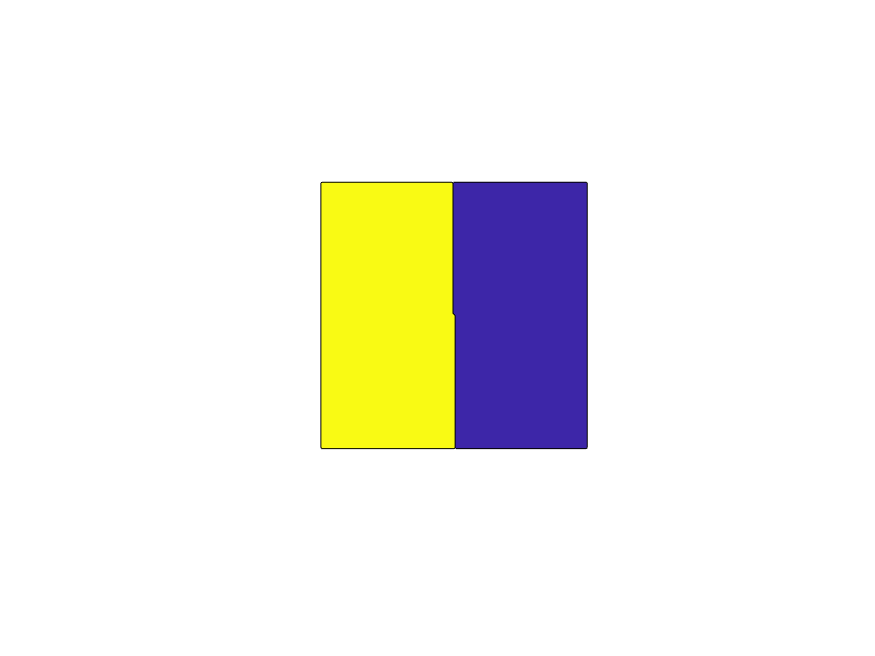}}
    \subfigure[triangle]{\includegraphics[width=0.24\textwidth, clip, trim = 2cm 1.5cm 4cm 2cm]{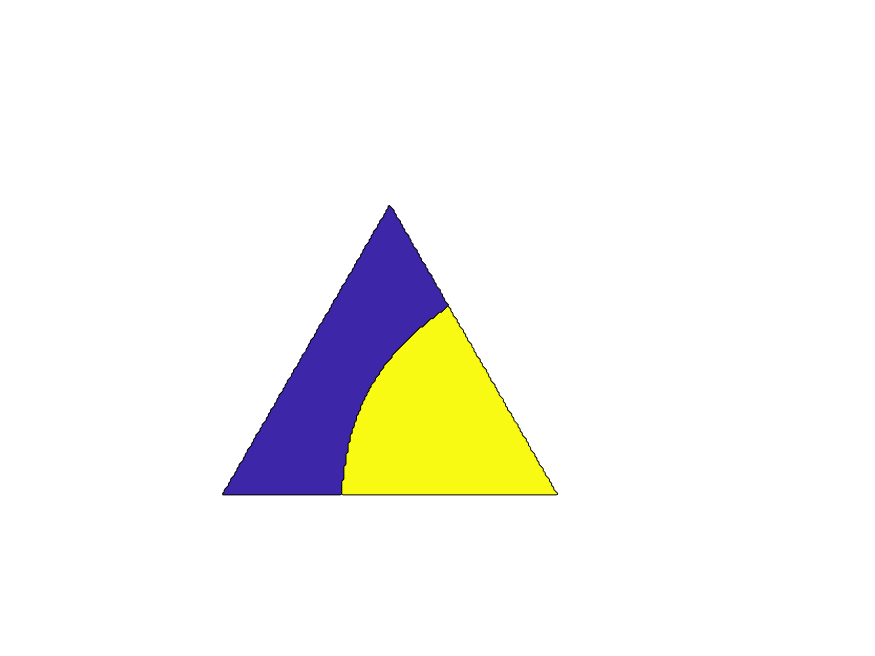}}
    \subfigure[flower]{\includegraphics[width=0.24\textwidth, clip, trim = 3cm 2cm 2cm 1.5cm]{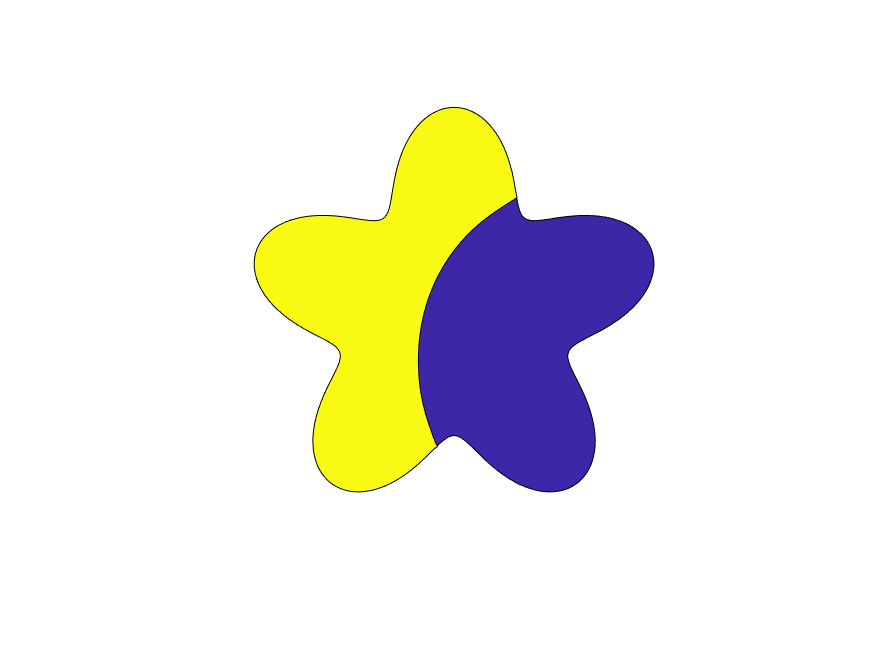}}
    \caption{The shortest partitions of different $\Omega$ in P{\'o}lya's question. The partitions are computed numerically by auction dynamics in Section \ref{sec:method}.}
    \label{fig:Intro}
\end{figure}

The relative perimeter here is defined by using the total variation. Suppose $\chi_{\omega}$ is the indicator function of $\omega$, then 
\begin{align*}
    \text{Per}_{\Omega}(\omega)=TV(\chi_{\omega},\Omega),
\end{align*}
where
\begin{align*}
    TV(u,\Omega)=\sup \left\{ \int_{\Omega}u \text{div} g : g\in C_{c}(\Omega,\mathbb{R}^d),||g||_{\infty}\leq 1\right\}
\end{align*}
is the total variation of $u$ on $\Omega$, and $C_{c}(\Omega,\mathbb{R}^d)$ means the space of $C^{\infty}$ vector-valued functions on $\Omega$ with compact support. Intuitively speaking, $\text{Per}_{\Omega}(\omega)=|\partial\omega \backslash \partial\Omega|$, and $\sum_{i=1}^{n}\text{Per}_{\Omega}(\omega_i)$ in $PI(\Omega,\mathbf{c})$ is $\frac{1}{2}(\sum_{i=1}^{n}|\partial\omega_i|-|\partial\Omega|)$. More details about the relative perimeter can be found in \cite{braides1998approximation}.

In \cite{bogosel2022longest}, the authors propose a phase-field method based on shape derivatives to numerically simulate the longest minimal length partitions. The method involves using a Modica-Mortola approximation to represent the perimeter, and the total region is parameterized using a radial function discretized with Fourier coefficients. The authors update the total region by iteratively modifying the Fourier coefficients based on shape derivatives. To maintain the volume constraint, they project the Fourier coefficients using a homothety at each iteration step. The authors claim that the results of their method provide numerical evidence for the conjecture that for any fixed $\mathbf{c}$, the disk is always the maximizer in two dimensions, and the ball is always the maximizer in three dimensions. 

In \cite{miranda2007short} the authors present a short-time heat flow approximation for the boundary integral energy of functions with bounded variation. This approximation utilizes indicator functions and Gaussian convolution and does not require taking derivatives of the functions. In \cite{merriman1992diffusion, MBO1993, merriman1994motion}, Merriman, Bence, and Osher (MBO) developed a threshold dynamics method for interface motion. The method has been proved to converge to mean curvature flow in \cite{evans1993convergence,barles1995simple}. Esedoglu and Otto further extended this method to the multi-phase case in \cite{esedoḡ2015threshold}. The method can be applied to general multiphase problems with arbitrary surface tensions. The MBO method and its idea have gained recognition for simplicity and stability, and have been widely applied in various areas such as motion by mean curvature for filaments \cite{ruuth2001diffusion}, volume-preserving interface motion \cite{ruuth2003simple,jacobs2018auction}, image processing \cite{esedog2006threshold,wang2017efficient,wang2022iterative,tai2007image,tai2005remark}, graph data processing \cite{merkurjev2013mbo,merkurjev2015global,boyd2018simplified}, foam bubbles \cite{wang2019dynamics}, grain boundary motion \cite{elsey2009diffusion} and wetting phenomena on solid surfaces \cite{xu2017efficient}. Recently, connections among the MBO method, operator-splitting methods and deep neural networks are pointed out in \cite{liu2023double,liu2023connections,tai2024pottsmgnet}.

Numerous algorithms and rigorous error analyses have been developed to refine and expand upon the original MBO method \cite{merriman1994motion,ruuth1998efficient_a,ishii2005optimal}. For instance, in \cite{jacobs2018auction}, the authors extend the MBO method to multiphase volume-constrained curvature motion using an auction dynamics scheme. This extension allows for solving problems related to minimal surface area with multiphase volume constraints.

In this paper, two new iterative methods are introduced for longest minimal length partitions. The objective functional is approximated by a short time heat flow. Then, the method of threshold dynamics is used to maximize $PI(\Omega,\mathbf{c})$ under volume constraints, and auction dynamics is used to find $SP(\Omega,\mathbf{c})$ for a fixed $\Omega$.

The rest of this paper is organized as follows. Section \ref{sec:appro} introduces approximations for the objective functional. These approximations serve as the foundation for the subsequent methods. Section \ref{sec:method} presents the derivation of the two-step methods based on the introduced approximations. In Section \ref{sec:exper}, we present numerical experiments in two and three dimensions to demonstrate the performance of the proposed methods. We propose in Section\ref{sec:monotone} an objective functional-monotone method. We then draw conclusions and discussions in Section \ref{sec:conclusion}.

\section{Approximations of the objective functional}
\label{sec:appro}

Let $\Omega\subset \mathbb{R}^d$ be an open and connected set to be partitioned. We define the objective functional (the length of the partition of $\Omega$) by 
\begin{align}
    E\left(\Omega,(\omega_i)_{i=1}^{n}\right)=\sum_{i=1}^{n}\text{Per}_{\Omega}(\omega_i). \label{exact energy}
\end{align}
Given a volume constraint $V>0$ and partition proportions $\mathbf{c}=(c_1,\dots,c_n)$, the relative isoperimetric profile is written as
\begin{align*}
    PI(\Omega,\mathbf{c})=\min_{|\omega_i|=c_i V}  E\left(\Omega,(\omega_i)_{i=1}^{n}\right).
\end{align*}
The longest minimal length problem can be formulated as the following max-min optimization problem
\begin{align}
    \label{original optimization}
    \max_{|\Omega|=V} PI(\Omega,\mathbf{c})=\max_{|\Omega|=V}\min_{|\omega_i|=c_i V}  E\left(\Omega,(\omega_i)_{i=1}^{n}\right).
\end{align}

To approximate the objective functional, we use indicator functions to implicitly represent the regions $\Omega$ and $(\omega_i)_{i=1}^{n}$. Define  $u_{\Omega}$ and $(u_{i})_{i=1}^{n}$ by
\begin{align}
    u_{\Omega}(\mathbf{x})=\begin{cases}
        1 & \mbox{ if } \mathbf{x}\in \Omega,\\
        0 & \mbox{ if } \mathbf{x}\notin \Omega,
    \end{cases}
    \quad 
    u_i(\mathbf{x})=\begin{cases}
        1 & \mbox{ if } \mathbf{x}\in \omega_i,\\
        0 & \mbox{ if } \mathbf{x}\notin \omega_i,
    \end{cases}
    \mbox{ for } i=1,...,n.
\end{align}
The volume constraints in (\ref{original optimization}) can be expressed as $\int_{\mathbb{R}^{d}}u_{\Omega}\mathrm{d}\mathbf{x}=V$ and $\int_{\mathbb{R}^{d}}u_{i}\mathrm{d}\mathbf{x}=c_{i}V$. Since $(\omega_i)_{i=1}^n$ is a partition of $\Omega$, we also have $\sum_{i=1}^n u_i =u_{\Omega}$.

To approximate the objective functional, instead of using a Modica-Mortola approximation as in \cite{bogosel2022longest}, we use indicator functions and Gaussian convolution. Note that $\text{Per}_{\Omega}(\omega_i)=\sum_{j=1}^{n}\text{Per}_{\omega_i\cup\omega_j}(\omega_i)$, and $\text{Per}_{\omega_i\cup \omega_i}(\omega_i)=\text{Per}_{\omega_i}(\omega_i)=0$. For any $i \neq j$, we have $\text{Per}_{\omega_i\cup\omega_j}(\omega_j)=\text{Per}_{\omega_i\cup\omega_j}(\omega_i)$, which can be approximated by 
\begin{align}
    \text{Per}_{\omega_i\cup\omega_j}(\omega_j)\approx \sqrt{\frac{\pi}{\tau}}\int_{\mathbb{R}^{d}} u_i \left( G_{\tau}\ast u_j\right) \mathrm{d}\mathbf{x}, 
    \label{appro:relative perimeter}
\end{align}
where $$G_{\tau}(\mathbf{x})=\frac{1}{(4\pi\tau)^{d/2}}\exp\left(-\frac{|\mathbf{x}|^2}{4\tau}\right)$$ is the Gaussian kernel and $\ast$ denotes convolution in $\mathbb{R}^{d}$. The convergence of this approximation for closed curves has been shown in \cite{miranda2007short}  as $\tau$ goes to $0^{+}$ and the multi-phase formula was proposed in \cite{esedoḡ2015threshold}.
By utilizing (\ref{appro:relative perimeter}), we approximate $E\left(\Omega,(\omega_i)_{i=1}^{n}\right)$ in \eqref{exact energy} by
\begin{align}
    E\left(\Omega,(\omega_i)_{i=1}^{n}\right)&=\sum_{i=1}^{n}\sum_{j\neq i}\text{Per}_{\omega_j\cup\omega_i}(\omega_i)\nonumber\\
    &\approx \sqrt{\frac{\pi}{\tau}}\sum_{i=1}^{n}\sum_{j\neq i}\int_{\mathbb{R}^{d}} u_i \left( G_{\tau}\ast u_j\right) \mathrm{d}\mathbf{x}\label{appro:intermediate 1}\\
    &=\sqrt{\frac{\pi}{\tau}}\sum_{i=1}^{n}\int_{\mathbb{R}^{d}} u_i \left( G_{\tau}\ast \left(\sum_{j\neq i}u_j\right)\right) \mathrm{d}\mathbf{x}\nonumber\\
    &=\sqrt{\frac{\pi}{\tau}}\sum_{i=1}^{n}\int_{\mathbb{R}^{d}} u_i \left( G_{\tau}\ast \left(u_{\Omega}-u_i\right)\right) \mathrm{d}\mathbf{x}\nonumber\\
    &=\sqrt{\frac{\pi}{\tau}}\sum_{i=1}^{n}\int_{\mathbb{R}^{d}} (G_{\tau/2}\ast u_i) \left( G_{\tau/2}\ast \left(u_{\Omega}-u_i\right)\right) \mathrm{d}\mathbf{x}. \label{appro:intermediate 2}
\end{align} 
Since $\sqrt{\frac{\pi}{\tau}}(G_{\tau/2}\ast u_i) \left( G_{\tau/2}\ast \left(u_{\Omega}-u_i\right)\right)$ concentrates on $\Omega$ as $\tau$ goes to $0^+$, \ie, 
\begin{align*}
    \lim_{\tau\rightarrow 0^+}\sqrt{\frac{\pi}{\tau}}\int_{\mathbb{R}^{d}\setminus\Omega} (G_{\tau/2}\ast u_i) \left( G_{\tau/2}\ast \left(u_{\Omega}-u_i\right)\right) \mathrm{d}\mathbf{x}=0,
\end{align*}
when $\tau$ is sufficiently small, the objective functional will not change too much if we replace the integrating domain in (\ref{appro:intermediate 2}) from $\mathbb{R}^d$ to $\Omega$. To achieve this, recall that $u_{\Omega}$ is the indicator function of $\Omega$, we multiply the integrand in (\ref{appro:intermediate 2}) by $u_{\Omega}$. Thus the objective functional in (\ref{appro:intermediate 2}) is further approximated by
\begin{align}
    \eqref{appro:intermediate 2}\approx& \sqrt{\frac{\pi}{\tau}}\sum_{i=1}^{n}\int_{\mathbb{R}^{d}} u_{\Omega}(G_{\tau/2}\ast u_i) \left( G_{\tau/2}\ast \left(u_{\Omega}-u_i\right)\right) \mathrm{d}\mathbf{x}\nonumber\\
    =&\sqrt{\frac{\pi}{\tau}}\int_{\mathbb{R}^{d}} u_{\Omega}\left(S_{\tau/2}-\sum_{i=1}^{n}(G_{\tau/2}\ast u_i)^2\right)\mathrm{d}\mathbf{x}\nonumber\\&-\sqrt{\frac{\pi}{\tau}}\int_{\mathbb{R}^{d}} u_{\Omega}S_{\tau/2} \left( G_{\tau/2}\ast (1-u_{\Omega})\right) \mathrm{d}\mathbf{x}, \label{appendix_1}
\end{align}
where $S_{\tau/2}$ denotes $S_{\tau/2}(\mathbf{x}):=G_{\tau/2}\ast \left(\sum_{i=1}^{n}u_i\right)(\mathbf{x})$.
By \cite{wang2021efficient}, in the sense of $\tau$ going to $0^+$, the last quadratic term can be approximated by
\begin{align*}
&\sqrt{\frac{\pi}{\tau}}\int_{\mathbb{R}^{d}} u_{\Omega}S_{\tau/2} \left( G_{\tau/2}\ast (1-u_{\Omega})\right) \mathrm{d}\mathbf{x} \\
\approx& \sqrt{\frac{\pi}{\tau}}\int_{\mathbb{R}^{d}} u_{\Omega}S_{\tau/2}^{\frac{1}{2}} \left( G_{\tau/2}\ast \left(S_{\tau/2}^{\frac{1}{2}}(1-u_{\Omega})\right)\right) \mathrm{d}\mathbf{x}.
\end{align*}

Denote the approximate objective functional by
\begin{align}
    \label{approximate energy 2}
    \Tilde{E}_{\tau}(u_{\Omega},(u_i)_{i=1}^n):=\sqrt{\frac{\pi}{\tau}}\int_{\mathbb{R}^{d}} u_{\Omega}\left(S_{\tau/2}-\sum_{i=1}^{n}(G_{\tau/2}\ast u_i)^2\right)\mathrm{d}\mathbf{x}\nonumber\\
    -\sqrt{\frac{\pi}{\tau}}\int_{\mathbb{R}^{d}} u_{\Omega}S_{\tau/2}^{\frac{1}{2}} \left( G_{\tau/2}\ast \left(S_{\tau/2}^{\frac{1}{2}}(1-u_{\Omega})\right)\right) \mathrm{d}\mathbf{x}.
\end{align}
The original problem \eqref{original optimization} is approximated by
\begin{align}
    \label{modified optimization 1}
    \max_{\substack{u_\Omega \in \mathcal{B} \\ \int_{\mathbb{R}^d} u_{\Omega}\mathrm{d}\mathbf{x}=V}}
    \min_{\substack{u_i \in \mathcal{B} \\ \int_{\mathbb{R}^d} u_{i}\mathrm{d}\mathbf{x}=c_i V \\ \sum u_i = u_{\Omega}}}
    \Tilde{E}_{\tau}(u_{\Omega},(u_i)_{i=1}^n),
\end{align}
where $\mathcal{B}:=\{u\in BV(\Omega,\mathbb{R})| u=\{0,1\}\}$, and $BV(\Omega,\mathbb{R})$ denotes the bounded-variation functional space.

Note that \eqref{appro:intermediate 1} can also be used for approximation, leading to another approximation of the objective functional
\begin{align}
    \label{approximate energy 1}
     \Hat{E}_{\tau}(u_{\Omega},(u_i)_{i=1}^n):=\sqrt{\frac{\pi}{\tau}}\sum_{i=1}^{n}\sum_{j\neq i}\int_{\mathbb{R}^{d}} u_i \left( G_{\tau}\ast u_j\right) \mathrm{d}\mathbf{x}.
\end{align}
The corresponding modified problem is 
\begin{align}
    \label{modified optimization 2}
    \max_{\substack{u_\Omega \in \mathcal{B} \\ \int_{\mathbb{R}^d} u_{\Omega}\mathrm{d}\mathbf{x}=V}}
    \min_{\substack{u_i \in \mathcal{B} \\ \int_{\mathbb{R}^d} u_{i}\mathrm{d}\mathbf{x}=c_i V \\ \sum u_i = u_{\Omega}}}
    \Hat{E}_{\tau}(u_{\Omega},(u_i)_{i=1}^n).
\end{align}
Here, $\Tilde{E}_{\tau}$ and  $\Hat{E}_{\tau}$ are equivalent to each other in the sense of $\tau$ converging to $0^+$, and hence are \eqref{modified optimization 1} and \eqref{modified optimization 2}.

\section{Derivation of the methods}
\label{sec:method}

Based on the two objective functional approximations in (\ref{approximate energy 2}) and (\ref{approximate energy 1}), in this section, two iterative methods are derived for longest minimal length partitions. Denote the indicator functions of the total region and the shortest partition in $k$-th iteration by $u_{\Omega}^{k}$ and $(u_{i}^{k})_{i=1}^n$ respectively. Each method contains two steps:
\begin{itemize}
    \item[Step 1:] Fix $(u_{i}^{k})_{i=1}^{n}$ and find $u_{\Omega}^{k+1}$ by maximizing $\widetilde{E}_{\tau}$.
    \item[Step 2:] Fix $u_{\Omega}^{k+1}$ and find $(u_{i}^{k+1})_{i=1}^{n}$ by minimizing $\widehat{E}_{\tau}$.
\end{itemize}
    
In Step 2, due to the good form of $\widehat{E}_{\tau}$, auction dynamics \cite{jacobs2018auction} can be directly employed in both methods. The $(u_{i}^{k})_{i=1}^{n}$ given by the two methods preserve $\sum_{i=1}^{n}u_{i}^{k}=u_{\Omega}^{k}$ to ensure that the collection of subsets is a partition of the total region. In Step 1, the update is based on the objective functional $\widetilde{E}_{\tau}$ and threshold dynamics is used. The constraint $\sum_{i=1}^{n}u_{i}=u_{\Omega}$ is relaxed in this part to update the total region. Otherwise, if $u_{\Omega}^{k+1}=\sum_{i=1}^{n}u_{i}^{k}=u_{\Omega}^{k}$, the total region would remain unchanged during iterations. To maintain the stability of the method, in the first method, the update of $u_{\Omega}^{k+1}$ is fully leaded by the behavior of $\Tilde{E}_{\tau}$, while $(u_i^{k+1})_{i=1}^n$ are updated by running auction dynamics several times to avoid local minimizer of $\Tilde{E}_{\tau}$. In the second method, a proximal term is introduced to control the evolution of $u_{\Omega}$ and only a one-time auction dynamics is used to updated $(u_i^{k+1})_{i=1}^n$, which enhance both the stability and efficiency. 

\subsection{The first method}

\subsubsection{Update of \texorpdfstring{$u_{\Omega}^{k+1}$}{}}

In this step, we find $u_{\Omega}^{k+1}$ based on $(u_{i}^{k})_{i=1}^n$.
In the update of the total region, the approximate objective functional $\Tilde{E}_{\tau}$ (defined in \eqref{approximate energy 2})  serves as a guidance. Fixing $(u_{i}^{k})_{i=1}^n$, we update $u_{\Omega}^{k+1}$ by
\begin{align}
    \label{update Omega: 1}
    u_{\Omega}^{k+1} = \argmax_{\substack{u_{\Omega} \in \mathcal{B} \\ \int_{\mathbb{R}^d} u_{\Omega}\mathrm{d}\mathbf{x}=V}} 
    \Tilde{E}_{\tau}(u_{\Omega},(u_{i}^{k})_{i=1}^n).
\end{align}

To solve \eqref{update Omega: 1}, we relax the constraint $u_{\Omega} \in \{0,1\}$ to $u_{\Omega}\in [0,1]$. Specifically, we define $\mathcal{K}:=\{u\in BV(\Omega,\mathbb{R})| u\in[0,1] \}$ and consider the following problem instead:
\begin{align}
    \label{update Omega: 1.2}
    u_{\Omega}^{k+1} = \argmax_{\substack{u_{\Omega} \in \mathcal{K} \\\int_{\mathbb{R}^d} u_{\Omega}\mathrm{d}\mathbf{x}=V}}
    \Tilde{E}_{\tau}(u_{\Omega},(u_{i}^{k})_{i=1}^n).
\end{align}
The following lemma shows that (\ref{update Omega: 1}) and (\ref{update Omega: 1.2}) have the same maximizer.
\begin{lem}\label{lem:equivalence}
   The problems \eqref{update Omega: 1} and \eqref{update Omega: 1.2} have the same maximizer, \ie, 
    \begin{align*}
        \argmax_{\substack{u_{\Omega} \in \mathcal{B} \\ \int_{\mathbb{R}^d} u_{\Omega}\mathrm{d}\mathbf{x}=V}} 
        \Tilde{E}_{\tau}(u_{\Omega},(u_{i}^{k})_{i=1}^n)=\argmax_{\substack{u_{\Omega} \in \mathcal{K} \\ \int_{\mathbb{R}^d} u_{\Omega}\mathrm{d}\mathbf{x}=V}}  \Tilde{E}_{\tau}(u_{\Omega},(u_{i}^{k})_{i=1}^n).
    \end{align*}
\end{lem}
\begin{proof}
    Note that 
    \begin{align*}
    \max_{\substack{u_{\Omega} \in \mathcal{B} \\ \int_{\mathbb{R}^d} u_{\Omega}\mathrm{d}\mathbf{x}=V }} \Tilde{E}_{\tau}(u_{\Omega},(u_{i}^{k})_{i=1}^n)\leq\max_{\substack{
    u_{\Omega} \in \mathcal{K} \\
    \int_{\mathbb{R}^d} u_{\Omega}\mathrm{d}\mathbf{x}=V 
    }} \Tilde{E}_{\tau}(u_{\Omega},(u_{i}^{k})_{i=1}^n),
    \end{align*}
    since $\mathcal{B}\subset \mathcal{K}$. To prove the equivalence, it suffices to show
    \begin{align*}
        \argmax_{\substack{u_{\Omega} \in \mathcal{K} \\ \int_{\mathbb{R}^d} u_{\Omega}\mathrm{d}\mathbf{x}=V }} \Tilde{E}_{\tau}(u_{\Omega},(u_{i}^{k})_{i=1}^n)\in\left\{u_{\Omega} \in \mathcal{B}:\int_{\mathbb{R}^d} u_{\Omega}\mathrm{d}\mathbf{x}=V\right\}.
    \end{align*}
    Assume this is not true and the maximizer is $u^{\star}$. Then there exists a set $A\subset\Omega$ with $|A|>0$ and $0<c<\frac{1}{2}$ such that
    \begin{align*}
        u^{\star}(\mathbf{x})\in(c,1-c), \text{ for any } \mathbf{x}\in\Omega.
    \end{align*}
    Divide $A$ into two subsets $A_1$ and $A_2$ such that $A_1\cup A_2=A$, $A_1\cap A_2 =\emptyset$ and $|A_1|=|A_2|=\frac{1}{2}|A|$. Define $u^t = u^{\star}+t\chi_{A_1}-t\chi_{A_2}$ where $\chi_{A_1}$ and $\chi_{A_2}$ are indicator functions of $A_1$ and $A_2$ respectively. Then $u^t\in\{u_{\Omega} \in \mathcal{K}:\int u_{\Omega}\mathrm{d}\mathbf{x}=V\}$ for any $|t|<\frac{c}{2}$. Observe that
    \begin{align*}
        &\frac{\mathrm{d}^2}{\mathrm{d}t^2}\Tilde{E}_{\tau}(u^{t},(u_{i}^{k})_{i=1}^n)\\=&2\sqrt{\frac{\pi}{\tau}}\int_{\mathbb{R}^{d}} (\chi_{A_1}-\chi_{A_2})(S_{\tau/2}^{k})^{\frac{1}{2}} \left(G_{\tau/2}\ast \left((\chi_{A_1}-\chi_{A_2})(S_{\tau/2}^{k})^{\frac{1}{2}}\right)\right) \mathrm{d}\mathbf{x}\\
        =&2\sqrt{\frac{\pi}{\tau}}\int_{\mathbb{R}^{d}}\left(G_{\tau/4}\ast \left((\chi_{A_1}-\chi_{A_2})(S_{\tau/2}^{k})^{\frac{1}{2}}\right)\right)^{2} \mathrm{d}\mathbf{x}\\
        >&0,
    \end{align*}
    where $S_{\tau/2}^{k}$ denotes $S_{\tau/2}^{k}(\mathbf{x})=G_{\tau/2}\ast \left(\sum_{i=1}^{n}u_{i}^{k}\right)(\mathbf{x})$.
    Particularly, when taking $t=0$, $u^0=u^{\star}$, $\frac{\mathrm{d}^2}{\mathrm{d}t^2}\Tilde{E}_{\tau}(u^{\star},(u_{i}^{k})_{i=1}^n)>0$, which contradicts with that $u^{\star}$ is the maximizer. Therefore, the maximizer is in $\mathcal{B}$ and this ends the proof.
\end{proof}

By Lemma \ref{lem:equivalence}, solving (\ref{update Omega: 1}) is converted to solving (\ref{update Omega: 1.2}), for which we follow \cite{esedoḡ2015threshold} to linearize $\Tilde{E}_{\tau}(u_{\Omega},(u_{i}^{k})_{i=1}^n)$ as
\begin{align*}
    L_{\tau}(u_{\Omega},u_{\Omega}^{k},(u_{i}^{k})_{i=1}^n)=\int_{\mathbb{R}^{d}}u_{\Omega}\phi^{k}\mathrm{d}\mathbf{x},
\end{align*}
where the dominant function is
\begin{align}
    \phi^{k}(\mathbf{x})=\sqrt{\frac{\pi}{\tau}}\left(S_{\tau/2}^{k}-\sum_{i=1}^{n}\left(G_{\tau/2}\ast u_{i}^{k}\right)^{2}+(S_{\tau/2}^{k})^{\frac{1}{2}}G_{\tau/2}\ast \left((S_{\tau/2}^{k})^{\frac{1}{2}}\left(2u_{\Omega}^{k}-1\right)\right)\right).\label{dominant function1}
\end{align}
By threshold dynamics, $u_{\Omega}^{k+1}$ is updated by maximizing the linearized objective functional under the volume constraint
\begin{align*}
    u_{\Omega}^{k+1}=\argmax_{\int_{\mathbb{R}^d} u_{\Omega}\mathrm{d}\mathbf{x}=V} L_{\tau}(u_{\Omega},u_{\Omega}^{k},(u_{i}^{k})_{i=1}^n).
\end{align*}
Let $\sigma>0$ such that
\begin{align}
    \int_{\mathbb{R}^{d}} \chi_{\sigma}(\mathbf{x})\mathrm{d}\mathbf{x}=V,\label{eq:volume preserving}
\end{align}
where $\chi_{\sigma}(\mathbf{x})$ is the indicator function of the set $\{\mathbf{x}\in\mathbb{R}^{d}:\phi^{k}(\mathbf{x})\geq\sigma\}$. It is easy to see that the maximizer $u_{\Omega}^{k+1}$ is determined point-wisely by
\begin{align}
    u_{\Omega}^{k+1}(\mathbf{x})=
    \begin{cases}
        1 & \mbox{ if } \phi^{k}(\mathbf{x})\geq\sigma, \\
        0 & \mbox{ if } \phi^{k}(\mathbf{x})<\sigma,
    \end{cases}
    \label{update total}
\end{align}
which is exactly $u_{\Omega}^{k+1}(\mathbf{x})=\chi_{\sigma}(\mathbf{x})$. 

\begin{rem}
    In \cite{bogosel2022longest}, during the update of Fourier coefficients of the radial function which represents the total region, the volume constraint is broken. Then the authors project the updated coefficients onto the volume constraint using a homothety. For the methods in this paper, volume preserving is kept in the way of \eqref{eq:volume preserving} and \eqref{update total}. For regular discretization, such a $\sigma$ is easy to find. Besides, since $\sigma$ is a threshold of the dominant function, it makes sure that the update maximizes the approximate objective functional in this time step under the volume constraint.
\end{rem}

\begin{figure}[t!]
    \centering
    \subfigure[initial]{\includegraphics[width=0.3\textwidth, clip, trim = 4cm 2.5cm 3cm 1.5cm]{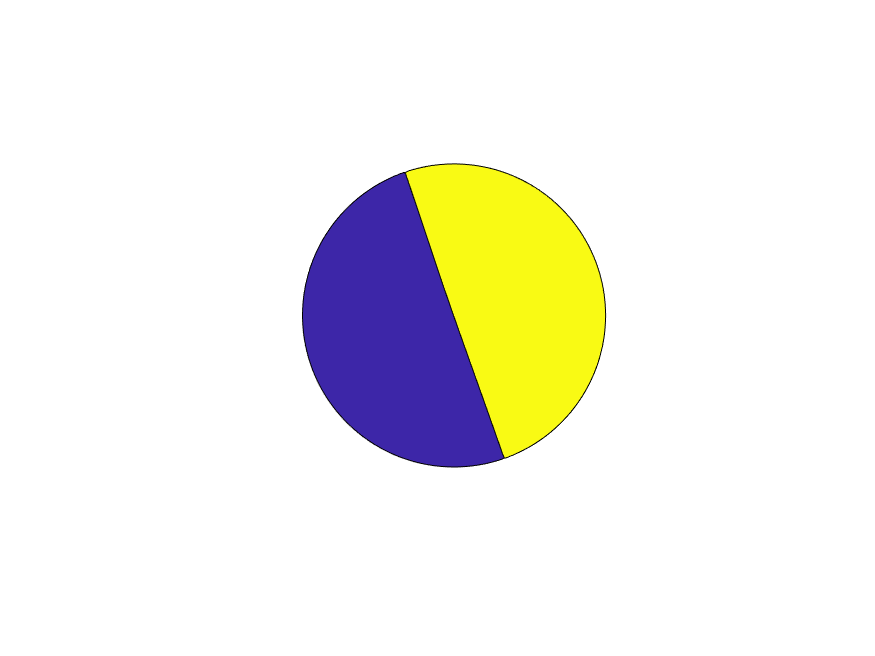}}
    \subfigure[after one iteration]{\includegraphics[width=0.3\textwidth, clip, trim = 4cm 2.5cm 3cm 1.5cm]{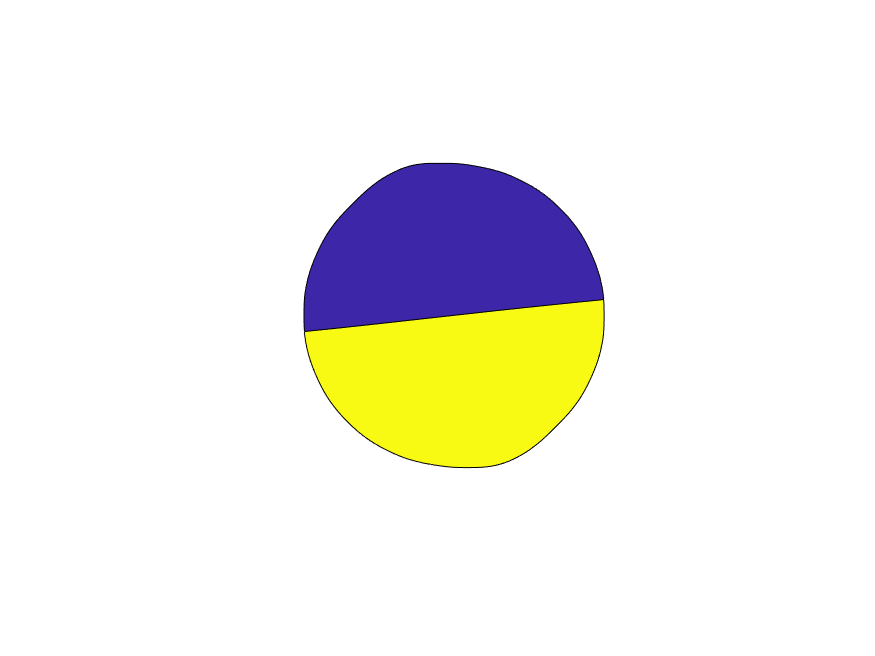}}
    \caption{The initial condition is a disc. After one iteration by \eqref{update total}, the total region changes instead of remaining round. }
    \label{fig:method}
\end{figure}

In fact, if we just update the total region by \eqref{update total} in each iteration, $u_{\Omega}^{k+1}$ will not converge to a stationary point, and the iterations will not stop. This is because the update of $u_{\Omega}^{k+1}$ depends on $(u_{i}^{k})_{i=1}^{n}$, which represents a shortest partition of $u_{\Omega}^{k}$ but may not be unique. For example, each diameter of a disc gives a shortest partition in the two-dimensional bisection case, and there are infinity many diameters. When we update the total region, $u_{\Omega}^{k+1}$ differs from $u_{\Omega}^{k}$ in the direction of increasing the objective functional of $(u_{i}^{k})_{i=1}^{n}$. However, this may decrease the objective functional of other shortest partitions and eventually decrease the objective function we want to maximize. As Figure \ref{fig:method} shows, if we take a disc as the initial total region, the update by \eqref{update total} will give a region which is not a disc after one iteration. This contradicts with the analytical result in \cite{Esposito2012} that the disc is the optimal solution of bisection in the two-dimensional case.

To address this issue, one approach is to identify all the shortest partitions of $u_{\Omega}^{k}$, and then update the total region in the direction of increasing the objective functional of all of them. However, this is impractical to implement as the number of shortest partitions can be infinite. An alternative approach is to reduce the update rate gradually when the change of objective functional is small. Then when $u_{\Omega}^{k}$ is close to the optimal solution, it will not change too much and go far away.

Instead of using all the points in the level set $\{\mathbf{x}\in\mathbb{R}^{d}:\phi^{k}(\mathbf{x})\geq\sigma\}$, we only use a part of it during the update of the total region. Denote $0\leq\beta\leq1$ the update step-length, and $u_{\Omega}^{process}(\mathbf{x})=\chi_{\sigma}(\mathbf{x})$ the indicator function obtained by \eqref{update total}. Denote $\Omega_{process}$ the region corresponding to $u_{\Omega}^{process}$, $A=\Omega_{process}-\Omega^k$ and $B=\Omega^k-\Omega_{process}$. Intuitively speaking, to increase the approximate objective functional in the direction of $(u_{i}^{k})_{i=1}^{n}$, the original algorithm tends to delete the points in $B$ and add the points in $A$ into the total region. Since $|\Omega_{process}|=|\Omega^k|$ by volume conservation, $|A| = |B|$ and $A\cap B=\emptyset$. In order to reduce the update rate, we find $\sigma_1$, $\sigma_2$ such that
\begin{align*}
    \int_{A}\chi_{\sigma_1}(\mathbf{x})\mathrm{d}\mathbf{x}&=\beta|A|,\\
    \int_{B}\chi_{\sigma_2}(\mathbf{x})\mathrm{d}\mathbf{x}&=(1-\beta)|B|,
\end{align*}
and define $\Tilde{A}=\{x\in A:\phi^{k}(\mathbf{x})\geq\sigma_1\}$ and $\Tilde{B}=\{x\in B:\phi^{k}(\mathbf{x})\leq\sigma_2\}$. Then $\Tilde{A}$ is a subset of $A$ with $|\Tilde{A}|=\beta |A|$, and for any $\mathbf{x}_{1}\in \Tilde{A}$, $\mathbf{x}_{2}\in (A-\Tilde{A})$, we have
\begin{align*}
    \phi^{k}(\mathbf{x}_1)\geq\sigma_1>\phi^{k}(\mathbf{x}_2),
\end{align*}
which means $\Tilde{A}$ is the part of $A$ with higher dominant function values. Similarly, $|\Tilde{B}|\subset |B|$, $|\Tilde{B}|=\beta |B|$ and $\Tilde{B}$ is the part of $B$ with lower dominant function values.
We update the total region by
\begin{align}
    u_{\Omega}^{k+1}(\mathbf{x})=u_{\Omega}^{k}(\mathbf{x})+\chi_{\Tilde{A}}-\chi_{\Tilde{B}},\label{update by part}
\end{align}
where $\chi_{\Tilde{A}}$ and $\chi_{\Tilde{B}}$ are the indicator functions of $\Tilde{A}$ and $\Tilde{B}$. In other words, we delete the points in $\Tilde{B}$ and add the points in $\Tilde{A}$ into the total region, instead of the whole $B$ and $A$. Note that $|\Tilde{A}|=|\Tilde{B}|$ and $\Tilde{A}\cap\Tilde{B}=\emptyset$. The update of the total region follows volume conservation. Here $\beta$ controls the degree of differences between $u_{\Omega}^{k}$ and $u_{\Omega}^{k+1}$. When $\beta=0$, the total region will not update during iterations; when $\beta=1$, the update is the same as \eqref{update total}. Now the update step-length can be controlled by setting different value of $\beta$.  

When $\beta$ is large, $u_{\Omega}^k$ may be difficult to converge due to the non-uniqueness of the solution and the two-step updating strategy. In fact, it oscillates around a disc. To obtain a stable and convergent scheme, we need to control the updating rate of $u_{\Omega}^k$, which is $\beta$ here. For a fixed $\beta$, after sufficiently many iterations, even though $u_{\Omega}^k$ oscillates, its objective functional stays at a certain scale. 
To evaluate the change of objective functional during iterations, we calculate the change of averaged objective functional over a  few ($M$ in Algorithm \ref{alg:first}) consecutive iterations. When the change is less than a tolerance $r_{tol}$, it is regarded to be small, and we reduce $\beta$ by multiplying a factor $\gamma$ ($0<\gamma<1$). Here the averaged objective functional of a few steps is used instead of the objective functional in one iteration. This can mitigate the impact of local shortest partitions on objective functional calculation.

\subsubsection{Update of \texorpdfstring{$(u_{i}^{k+1})_{i=1}^n$}{1}}

In the update of shortest partition, we fix $u_{\Omega}^{k+1}$ and use the approximate objective functional $\Hat{E}_{\tau}$ defined in \eqref{approximate energy 1}. Functions $(u_{i}^{k+1})_{i=1}^n$ are computed by solving  
\begin{align}
    (u_{i}^{k+1})_{i=1}^n = \argmin_{\substack{u_i \in \mathcal{B} \\
    \int_{\mathbb{R}^d} u_{i}d\mathbf{x}=c_i V\\
    \sum u_i = u_{\Omega}^{k+1}}}  \Hat{E}_{\tau}(u_{\Omega}^{k+1},(u_i)_{i=1}^n).
\end{align}

Observe that $\Hat{E}_{\tau}$ is proportional to the heat content energy in \cite{jacobs2018auction} 
\begin{align}
    HC_{\tau}=\frac{1}{\sqrt{\tau}}\sum_{i\neq j}\int u_i(\mathbf{x})(G_{\tau}\ast u_j)(\mathbf{x})\mathrm{d}\mathbf{x}.
\end{align} 
Therefore, given $u_{\Omega}^{k+1}$ with parameters of stopping criterion $\bm{\mu}$,  auction dynamics with volume constraints from \cite{jacobs2018auction} directly finds a shortest partition $(u_{i}^{k+1})_{i=1}^n$. Here $\bm{\mu}=(m,\epsilon_{min},\alpha,\epsilon_{0})$, $m$ is the number of maximum steps, $\epsilon_{min}$ is the auction error tolerance, $\alpha$ is the epsilon scaling factor and $\epsilon_{0}$ is the initial epsilon value. Denote the auction dynamics method by $ADM$. We compute $(u_{i}^{k+1})_{i=1}^n$ as
\begin{align}
     (u_{i}^{k+1})_{i=1}^n = ADM(u_{\Omega}^{k+1},V,\mathbf{c},\tau;\bm{\mu} ),\label{eq:adm}
\end{align}
where details of ADM is summarized in Algorithm \ref{alg:auction}, with the sub-routine membership auction summarized in Algorithm \ref{alg:membership}.

\begin{rem} In Algorithm \ref{alg:auction}, $\bm{\psi}^{l}$ is a tensor whose $i$-th element is $\psi_{i}^{l}(\mathbf{x})$. After discretizing the space, it becomes a matrix. The criterion for convergence is $u_{i,l+1}(\mathbf{x})=u_{i,l}(\mathbf{x})$ for all $i=1,\cdots,n$, \ie, when the partitions are no longer updated. In Algorithm \ref{alg:membership}, $i_{cs}(\mathbf{x},\mathbf{p})$ and $i_{next}(\mathbf{x},\mathbf{p},i^{\star})$ are defined by
    \begin{align*}
        i_{cs}(\mathbf{x},\mathbf{p})=\argmax_{1\leq i\leq n}\left[a_{i}(\mathbf{x})-p_{i}\right],\ 
        i_{next}(\mathbf{x},\mathbf{p},i^{\star})\in \argmax_{j\neq i^{\star}} \left[a_{j}(\mathbf{x})-p_{j}\right].
    \end{align*}
    Here $i_{cs}(\mathbf{x},\mathbf{p})$ is a set containing all indices satisfying the right-hand side,  $i_{next}(\mathbf{x},\mathbf{p},i^{\star})$ is just an arbitrary element of the set satisfying the right-hand side.
\end{rem}

\begin{algorithm}[t!]
\caption{Auction dynamics}\label{alg:auction}
\KwIn{Discrete total region $u_{\Omega}$, volume constraint $V$, partition proportions $\mathbf{c}$, time step $\tau$, parameters of stop conditions $\bm{\mu}=(m,\epsilon_{min},\alpha,\epsilon_{0})$. }
\KwOut{Shortest partitions $(u_i)_{i=1}^{n}$.}
\BlankLine
Set $l=0$; initialize $(u_{i,0})_{i=1}^{n}$ randomly with $\int u_{i,0} \mathrm{d}\mathbf{x}=c_{i}V$; set $\Bar{\epsilon}=\epsilon/n$; compute $G_{\tau}$;\\
\While{$l<m$ and not converged}{
Compute the convolutions: $\psi_{i}^{l}(\mathbf{x})=\sum_{i\neq j}(G_{\tau}\ast u_{j,l})(\mathbf{x})$;\\
Compute the assignment problem coefficients: $\mathbf{a}^{l} = 1-\bm{\psi}^{l}$;\\
Initialize the prices $\mathbf{p}=\mathbf{0}$, and $\epsilon=\epsilon_{0}$;\\
\While{$\epsilon\geq\Bar{\epsilon}$}{
Run Algorithm \ref{alg:membership}: $((u_i^{out})_{i=1}^{n},\mathbf{p}^{out})$= Membership auction$(\epsilon,V\mathbf{c},\mathbf{a}^{l},\mathbf{p},u_{\Omega})$;\\
Set $\mathbf{p}=\mathbf{p}^{out}$;\\
Divide $\epsilon$ by $\alpha$;\\
\If{$\epsilon<\Bar{\epsilon}$}{
Set $(u_{i,l+1})_{i=1}^{n}=(u_i^{out})_{i=1}^{n}$;
}}
Set $l=l+1$;
}
Return $(u_i)_{i=1}^{n}=(u_{i,l})_{i=1}^{n}$
\end{algorithm}

\begin{algorithm}[t!]
\caption{Membership auction }\label{alg:membership}
\KwIn{$\epsilon>0$, volumes $\mathbf{V}$, coefficients $\mathbf{a}$, initial prices $\mathbf{p}^{0}$ and indicator function $u_{\Omega}(\mathbf{x})$. }
\KwOut{Final partitions and prices $((u_i)_{i=1}^{n},\mathbf{p})$.}
\BlankLine
For every $i\in\{1,\cdots,n\}$ set $u_{i}(\mathbf{x})=0$ for all $\mathbf{x}$ such that $u_{\Omega}(\mathbf{x})=1$; set $\mathbf{p}=\mathbf{p}^{0}$\\
\While{some $\mathbf{x}$ is unassigned, i.e., $u_{\Omega}(\mathbf{x})=1$ but $\sum_{i}u_{i}(\mathbf{x})=0$,}{
\For{each unassigned $\mathbf{x}$}{
Compute $i_{cs}(\mathbf{x},\mathbf{p})$ and choose one element $i^{\star}\in i_{cs}(\mathbf{x},\mathbf{p})$;\\
Set $b(\mathbf{x})=p_{i^{\star}}+\epsilon+(a_{i^{\star}}(\mathbf{x})-p_{i^{\star}})-(a_{i_{next}(\mathbf{x},\mathbf{p},i^{\star})}(\mathbf{x})-p_{i_{next}(\mathbf{x},\mathbf{p},i^{\star})})$;\\
\eIf{$\int u_{i^{\star}}\mathrm{d}\mathbf{x}=V_{i^{\star}}$}{
Find $\mathbf{y}=\arg\min_{u_{i^{\star}}(\mathbf{z})=1}b(\mathbf{z})$;\\
Set $u_{i^{\star}}(\mathbf{y})=0$ and $u_{i^{\star}}(\mathbf{x})=1$;\\
Set $p_{i^{\star}}=\min_{u_{i^{\star}}(\mathbf{z})=1}b(\mathbf{z})$;
}{
Set $u_{i^{\star}}(\mathbf{x})=1$;\\
\If{$\int u_{i^{\star}}\mathrm{d}\mathbf{x}=V_{i^{\star}}$}{
Set $p_{i^{\star}}=\min_{u_{i^{\star}}(\mathbf{z})=1}b(\mathbf{z})$;
}}
}}
Return $((u_i)_{i=1}^{n},\mathbf{p})$
\end{algorithm}

The results of ADM critically depends on its initialization, which is usually set as random. In many cases, ADM does not always give global shortest partitions, but a local one. To make the partitions more likely to be a global shortest one, in the first method, ADM is repeated several times in each iteration. Every time a random initialization is given and a corresponding partition is created, we compute the objective functional of all partitions. The partitions with minimal objective functional are regarded as the shortest partitions. This increases the stability of the method.

The first two-step iterative method for the longest minimal length partitions problem is summarized in Algorithm \ref{alg:first}. The stopping criterion is set to be $\beta<\beta_{min}$ for some constant $\beta_{min}$, or the total regions between two iterations are exactly the same.

\begin{algorithm}[t!]
\caption{The first method for longest minimal length partitions }\label{alg:first}
\KwIn{Volume constraint $V$, partition proportions $\mathbf{c}=(c_1,\dots,c_n)$ satisfying $c_i>0$ and $\sum_{i=1}^n c_i=1$, time step $\tau>0$, auction dynamics parameters $\boldsymbol{\mu}$, repeat times of auction dynamics $p$, update step-length parameters $0\leq\beta_0\leq1$,$0<\gamma<1$, the number of steps in average objective functional $M\in\mathbb{N}$, $r_{tol}>0$, $0<\beta_{min}<1$.}
\KwOut{$u_{\Omega}$ maximizing the objective functional, shortest partitions $(u_i)_{i=1}^{n}$.}
\BlankLine
Set $k=0$; initialize $u_{\Omega}^{0}$ randomly with $\int u_{\Omega}^{0} \mathrm{d}\mathbf{x}=V$; $\beta=\beta_{0}$;\\
\For{$1\leq q\leq p$}{Compute $(u_{i,q}^0)_{i=1}^{n}=ADM(u_{\Omega}^0,V,\mathbf{c},\tau;\mu )$;}
Set $(u_{i}^{0})_{i=1}^{n}=\arg\min_{(u_{i,q}^0)_{i=1}^{n}}\Tilde{E}_{\tau}(u_{\Omega}^{0},(u_{i,q}^0)_{i=1}^{n})$;\\
\While{$\beta>\beta_{min}$ and not converge}{
Compute dominant function $ \phi^{k}$ by \eqref{dominant function1};\\
Update $u_{\Omega}^{k+1}$ by \eqref{update by part} using $\beta$;\\
\For{$1\leq q\leq p$}{Compute $(u_{i,q}^{k+1})_{i=1}^{n}=ADM(u_{\Omega}^{k+1},V,\mathbf{c},\tau;\mu )$;}
Set $(u_{i}^{k+1})_{i=1}^{n}=\arg\min_{(u_{i,q}^{k+1})_{i=1}^{n}}\Tilde{E}_{\tau}(u_{\Omega}^{k+1},(u_{i,q}^{k+1})_{i=1}^{n})$;\\
\If{$k>M$}{
Compute $E_{ave}^{k}=\frac{1}{M}\sum_{j=k-M+1}^{k}\Tilde{E}_{\tau}(u_{\Omega}^{j},(u_{i}^{j})_{i=1}^{n})$;\\
Compute $E_{ave}^{k+1}=\frac{1}{M}\sum_{j=k-M+2}^{k+1}\Tilde{E}_{\tau}(u_{\Omega}^{j},(u_{i}^{j})_{i=1}^{n})$;\\
\If{$|E_{ave}^{k+1}-E_{ave}^{k}|/|E_{ave}^{k+1}|<r_{tol}$}{$\beta=\gamma\beta$;\\}
}
$k=k+1$;
}
Return $u_{\Omega}=u_{\Omega}^{k}$, $(u_i)_{i=1}^{n}=(u_i^{k})_{i=1}^{n}$
\end{algorithm}

\begin{rem}
    In Algorithm \ref{alg:first} and Algorithm \ref{alg:second}, convergence means $u_{\Omega}^{k+1}=u_{\Omega}^{k}$.
\end{rem}

\begin{rem}
    In Algorithm \ref{alg:first}, although ADM is repeated with the same parameters, the random initialization makes the output $(u_{i,q}^{k+1})_{i=1}^{n}$ different each time.
\end{rem}

\subsection{The second method}
In the first method, to increase the stability, ADM is repeated several times in one iteration to find the shortest partitions. However, this increases the time complexity. To address this issue, the second method uses a more efficient way for stability. In the updating of $(u_{i}^{k})_{i=1}^{n}$, the second method follows a similar process as the first method but only employs ADM once. We omit the details here. In the updating of $u_{\Omega}^{k+1}$, a regularization term is introduced.
\subsubsection{Update of \texorpdfstring{$u_{\Omega}^{k+1}$ by $(u_{i}^{k})_{i=1}^n$}{}}
In the update of the total region, the approximate objective functional $\Tilde{E}_{\tau}$ in \eqref{approximate energy 2} is used. Given $u_{\Omega}^{k}$ and $(u_{i}^{k})_{i=1}^n$, $u_{\Omega}^{k+1}$ is attained by
\begin{align}\label{update Omega: 2}
    u_{\Omega}^{k+1} = \arg\max_{\substack{    u_{\Omega} \in \mathcal{B} \\
    \int u_{\Omega}=V}}  \Tilde{E}_{\tau}(u_{\Omega},(u_{i}^{k})_{i=1}^n)+\lambda\int_{\mathbb{R}^{d}} u_{\Omega} G_{\tau'}\ast\left( \sum_{i=1}^{n}u_{i}^{k}\right)\mathrm{d}\mathbf{x},
\end{align}
where $\lambda>0,\tau'>0$. In (\ref{update Omega: 2}), the second term is a regularization term to relax the constraint of $\sum_{i=1}^{n} u_i = u_{\Omega}$ and ensure that $u_{\Omega}^{k+1}$ does not deviate significantly from $u_{\Omega}^k$. In the following, we denote the regularized objective functional as 
\begin{align}
    \Tilde{E}_{\tau,\tau'}^{\lambda}(u_{\Omega},(u_{i}^{k})_{i=1}^n)=\Tilde{E}_{\tau}(u_{\Omega},(u_{i}^{k})_{i=1}^n)+\lambda\int_{\mathbb{R}^{d}} u_{\Omega}G_{\tau'}\ast\left( \sum_{i=1}^{n}u_{i}^{k}\right)\mathrm{d}\mathbf{x}.\nonumber
\end{align}

In fact, (\ref{update Omega: 2}) is a discretized scheme for an evolution equation of $u_{\Omega}$. Based on (\ref{approximate energy 2}), $u_{\Omega}^{k+1}$ is expected to maximize $\widetilde{E}_{\tau}(u_{\Omega},(u_i^k)_{i=1}^n)$. Consider a gradient ascent flow:
\begin{align}
    \frac{\partial u_{\Omega}}{\partial t}-\partial \widetilde{E}_{\tau}(u_{\Omega},(u_i^k)_{i=1}^n)=0.
    \label{eq.u_omega.update}
\end{align}
We update $u_{\Omega}^{k+1}$ by solving (\ref{eq.u_omega.update}) by a short time (one time step), and discretize (\ref{eq.u_omega.update}) with the time step $1/\lambda$ by a backward Euler scheme:
\begin{align}
    \frac{u_{\Omega}^{k+1}-u_{\Omega}^k}{1/\lambda}-\partial \widetilde{E}_{\tau}(u_{\Omega}^{k+1},(u_i^k)_{i=1}^n)=0.
\end{align}
Then $u_{\Omega}^{k+1}$ solves
\begin{align}
    u_{\Omega}^{k+1}=&\argmin_{u_{\Omega}} \left[\frac{\lambda}{2}\int_{\mathbb{R}} |u_{\Omega}-u_{\Omega}^k|^2d\bx -\widetilde{E}_{\tau}(u_{\Omega},(u_i^k)_{i=1}^n)\right] \nonumber\\
    =& \argmax_{u_{\Omega}} \left[\widetilde{E}_{\tau}(u_{\Omega},(u_i^k)_{i=1}^n)- \frac{\lambda}{2}\int_{\mathbb{R}} |u_{\Omega}-u_{\Omega}^k|^2d\bx\right]\nonumber\\
    =& \argmax_{u_{\Omega}} \left[\widetilde{E}_{\tau}(u_{\Omega},(u_i^k)_{i=1}^n)- \frac{\lambda}{2}\int_{\mathbb{R}} \left((u_{\Omega})^2+(u_{\Omega}^k)^2-2u_{\Omega}u_{\Omega}^k\right)d\bx\right].
    \label{eq.u_omega.update.1}
\end{align}
Using the constraint $u_{\Omega}\in \mathcal{B}, \int_{\mathbb{R}} u_{\Omega}d\bx=V$ and the fact $u_{\Omega}^k\in \mathcal{B}, \int_{\mathbb{R}} u_{\Omega}^kd\bx=V$, we further deduce
\begin{align}
    u_{\Omega}^{k+1}=& \argmax_{u_{\Omega}} \left[\widetilde{E}_{\tau}(u_{\Omega},(u_i^k)_{i=1}^n)- \frac{\lambda}{2}\int_{\mathbb{R}} \left(u_{\Omega}+u_{\Omega}^k-2u_{\Omega}u_{\Omega}^k\right)d\bx\right]\nonumber\\
    =& \argmax_{u_{\Omega}} \left[\widetilde{E}_{\tau}(u_{\Omega},(u_i^k)_{i=1}^n)- \frac{\lambda}{2}\int_{\mathbb{R}} \left(-2u_{\Omega}u_{\Omega}^k\right)d\bx\right]-\lambda V \nonumber\\
    =& \argmax_{u_{\Omega}} \left[\widetilde{E}_{\tau}(u_{\Omega},(u_i^k)_{i=1}^n)+\lambda\int_{\mathbb{R}} u_{\Omega}u_{\Omega}^kd\bx\right].
    \label{eq.u_omega.update.2}
\end{align}
Notice that $\sum_{i=1}^{n}u_{i}^{k}=u_{\Omega}^{k}$. Thus (\ref{update Omega: 2}) is an approximation of (\ref{eq.u_omega.update.2}), and such an approximation becomes exact as $\tau'\rightarrow 0^+$.

To solve \eqref{update Omega: 2}, similar to the first method, we relax the optimization to an equivalent problem: finding $u_{\Omega}^{k+1}$ such that 
\begin{align}
    u_{\Omega}^{k+1} = \arg\max_{\substack{
    u_{\Omega} \in \mathcal{K} \\
    \int u_{\Omega}=V }} \Tilde{E}_{\tau,\tau'}^{\lambda}(u_{\Omega},(u_{i}^{k})_{i=1}^n),
\end{align}
where $\mathcal{K}:=\{u\in BV(\Omega,\mathbb{R})| u\in[0,1] \}$. The proof of equivalence is similar to that of Lemma \ref{lem:equivalence}, and is omitted here.

The linearity of $\Tilde{E}_{\tau,\tau'}^{\lambda}(u_{\Omega},(u_{i}^{k})_{i=1}^n)$ is computed as
\begin{align*}
    L_{\tau,\tau'}^{\lambda}(u_{\Omega},u_{\Omega}^{k},(u_{i}^{k})_{i=1}^n)=\int_{\mathbb{R}^{d}}u_{\Omega}\phi^{k}\mathrm{d}\mathbf{x},
\end{align*}
where the dominant function is
\begin{align}
    \phi^{k}(\mathbf{x})=\sqrt{\frac{\pi}{\tau}}\left(S_{\tau/2}^{k}-\sum_{i=1}^{n}\left(G_{\tau/2}\ast u_{i}^{k}\right)^{2}+(S_{\tau/2}^{k})^{\frac{1}{2}}G_{\tau/2}\ast \left((S_{\tau/2}^{k})^{\frac{1}{2}}\left(2u_{\Omega}^{k}-1\right)\right)\right)+\lambda G_{\tau'}\ast\left(\sum_{i=1}^{n}u_{i}^{k}\right),\label{dominant function2}
\end{align}
and $S_{\tau/2}^{k}$ denotes $S_{\tau/2}^{k}(\mathbf{x})=G_{\tau/2}\ast \left(\sum_{i=1}^{n}u_{i}^{k}\right)(\mathbf{x})$. Then $u_{\Omega}^{k+1}$ can be computed by \eqref{update by part}, the same as in the first method.

For the stopping criterion, the second method also uses $\beta$ to control the update step-length, and reduces $\beta$ when the rate of change of the averaged objective functional in a few iterations is less than the tolerance $r_{tol}$. The iteration stops when $\beta<\beta_{min}$, or when the total region between two consecutive steps are the same. The algorithm of the second method is summarized in Algorithm \ref{alg:second}.

\begin{algorithm}[t!]
\caption{The second method for longest minimal length partitions}\label{alg:second}
\KwIn{Volume constraint $V$, partition proportions $\mathbf{c}=(c_1,\dots,c_n)$ satisfying $c_i>0$ and $\sum_{i=1}^n c_i=1$, time step $\tau>0$, regularization parameter $\lambda>0$, time step in regularization term $\tau'>0$ auction dynamics parameters $\boldsymbol{\mu}$, update step-length parameters $0\leq\beta_0\leq1$,$0<\gamma<1$, $M\in\mathbb{N}$, $r_{tol}>0$, $0<\beta_{min}<1$}
\KwOut{$u_{\Omega}$ maximizing the objective functional, relative shortest partitions $(u_i)_{i=1}^{n}$}
\BlankLine
Set $k=0$; initialize $u_{\Omega}^{0}$ randomly with $\int u_{\Omega}^{0} \mathrm{d}\mathbf{x}=V$;\\
Find $(u_i^0)_{i=1}^{n}=ADM(u_{\Omega}^0,V,\mathbf{c},\tau;\mu )$;\\
\While{$\beta>\beta_{min}$ and not converge}{
Compute dominant function $ \phi^{k}$ by \eqref{dominant function2} using $\tau$ and $\tau'$;\\
Update $u_{\Omega}^{k+1}$ by \eqref{update by part};\\
Compute $(u_i^{k+1})_{i=1}^{n}=ADM(u_{\Omega}^{k+1},V,\mathbf{c},\tau;\mu )$;\\
\If{$k>M$}{
Compute $E_{ave}^{k}=\frac{1}{M}\sum_{j=k-M+1}^{k}\Tilde{E}_{\tau}(u_{\Omega}^{j},(u_{i}^{j})_{i=1}^{n})$;\\
Compute $E_{ave}^{k+1}=\frac{1}{M}\sum_{j=k-M+2}^{k+1}\Tilde{E}_{\tau}(u_{\Omega}^{j},(u_{i}^{j})_{i=1}^{n})$;\\
\If{$|E_{ave}^{k+1}-E_{ave}^{k}|/|E_{ave}^{k+1}|<r_{tol}$}{$\beta=\gamma\beta$;\\}
}
$k=k+1$;
}
Return $u_{\Omega}=u_{\Omega}^{k}$, $(u_i)_{i=1}^{n}=(u_i^{k})_{i=1}^{n}$
\end{algorithm}

In the first method, the auction dynamics method is run several times in each iteration to find a partition, making the algorithm more stable and the partition more likely to be the global shortest one instead of a local one. However, the computational complexity of this method is high. In the second method, a regularization term is added in the approximate objective functional for stability. It doesn't need to run the auction dynamics several times at every iteration in order to maintain the stability of the algorithm.  Thus it is faster than the first method. Since the partition in each step may not be the global minimizer but a local one, the second method may need more iterations to converge. 

\section{Numerical experiments}
\label{sec:exper}
We demonstrate the effectiveness of the proposed methods by experiments in two and three dimensions. Our methods are implemented by MATLAB. Convolutions are implemented using Fast Fourier Transform. 

%To implement the method, matlab codes are used with CPU. The convolution of functions is calculated using the Fast Fourier Transform functions in matlab. 

\subsection{Numerical experiments in two dimensions}

To simulate the longest minimal length partitions in two dimensions, a $256\times256$ regular grid is formed in $[-\pi,\pi]^{2}$ with a space step $\Delta x=2\pi/256$. The time step is set to be $\tau=2\Delta x\approx 0.049$, and $\tau'=0.5\Delta x \approx 0.013$. The stopping parameters in the auction dynamics method are set to be $\bm{\mu}=(m,\epsilon_{min},\alpha,\epsilon_{0})=(1000,10^{-7},4,0.1)$. Other parameters are set as $\beta_0=1$ (the initial update rate for $u_{\Omega}$), $\gamma=\frac{1}{2}$ (the decay factor of $\beta$), $\beta_{min}=0,05$ (the lower bound of $\beta$), $M=5$ (the number of steps in averaged objective functional) and $r_{tol}=10^{-4}$. For the first method, auction dynamics repeat-time is $p=5$. For the second method, regularization parameter is $\lambda=10$.

\subsubsection{Two partitions}
\label{sec.numerical.2d.two}
We first conduct experiments for the two dimensional example with two partitions of equal proportions. The initial condition is set to be a five-petal flower with the indicator function $\rho^2<\pi^{2}(0.4+0.2\sin{5\theta})$, where $\rho$ and $\theta$ are the radius and angle in the polar coordinate respectively. The partition proportion is $\mathbf{c}=(\frac{1}{2},\frac{1}{2})$. Both methods derived in Section \ref{sec:method} are implemented. The first method stops after 34 iterations, and the second method stops after 63 iterations. The total regions and their shortest partitions of both methods in the initial, 5th, 10th and final iterations are plotted in Figure \ref{fig: 2partition}. For visualization of the results, the indicator functions of the region and partitions are smoothed using a Gaussian kernel convolution before being plotted in Figure \ref{fig: 2partition} and the figures in the following. %This can make the boundary of the region smoother, and does not change the shape of the region since the support of the mollified Gaussian kernel is chosen small.

\begin{figure}[t!]
	\centering
	\begin{tabular}{|c|c|c|c|c|}
		\hline 
		 & Initial & 5th iteration & 10th iteration & Final  \\
        \hline
        \makecell{First\\method}  &
		\raisebox{-.5\height}{\includegraphics[width=0.20\textwidth, clip, trim = 4cm 2.5cm 3cm 1.5cm]{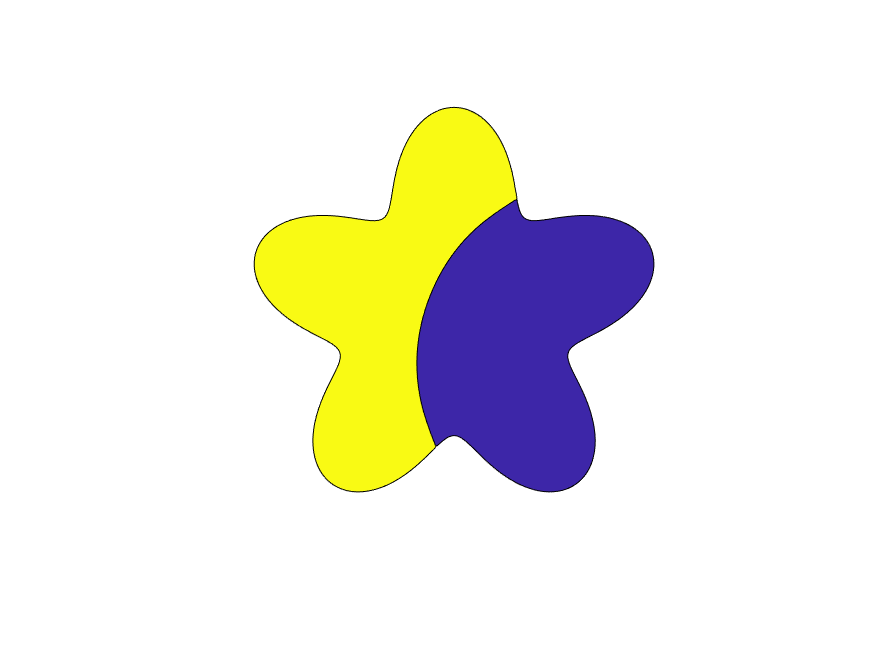}}&  
		\raisebox{-.5\height}{\includegraphics[width=0.20\textwidth, clip, trim = 4cm 2.5cm 3cm 1.5cm]{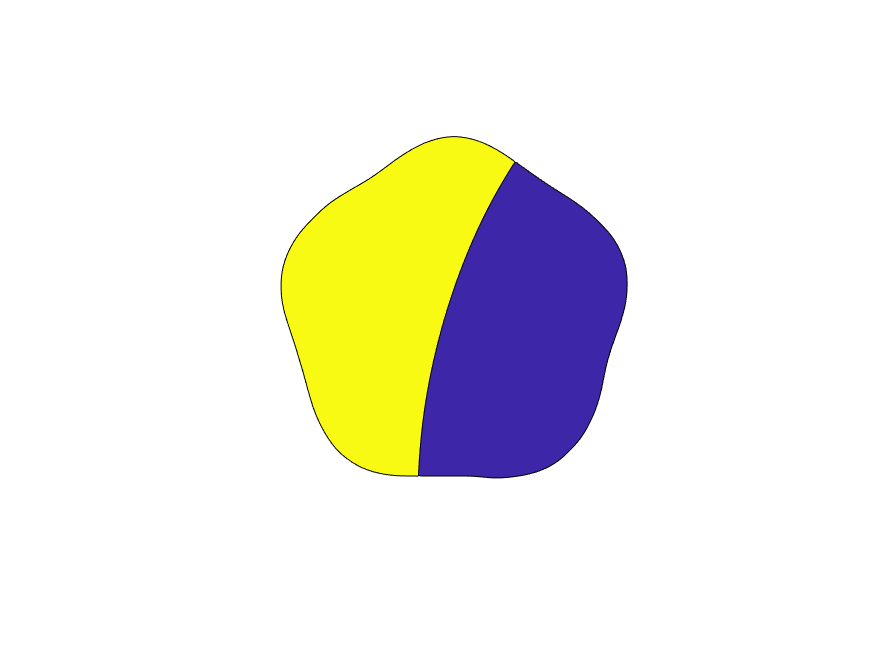}}&
		\raisebox{-.5\height}{\includegraphics[width=0.20\textwidth, clip, trim = 4cm 2.5cm 3cm 1.5cm]{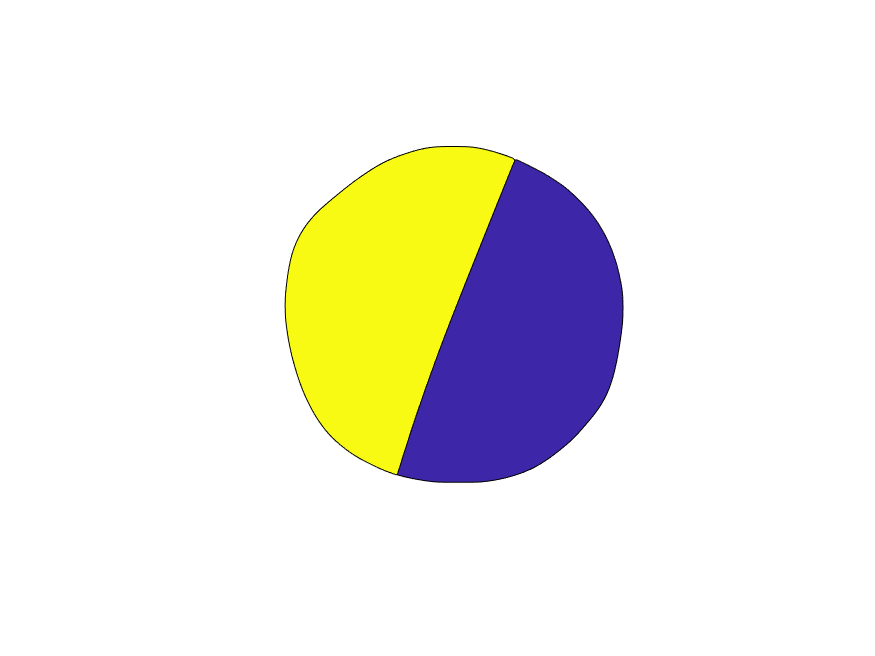}}&
		\raisebox{-.5\height}{\includegraphics[width=0.20\textwidth, clip, trim = 4cm 2.5cm 3cm 1.5cm]{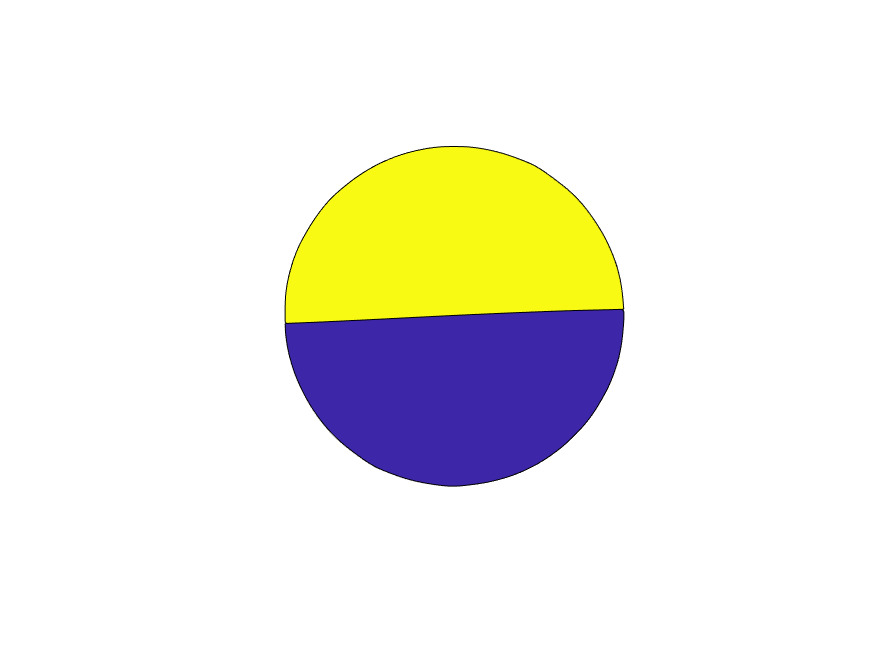}}
         \\ 
        \hline
        \makecell{Second\\method}  &
		\raisebox{-.5\height}{\includegraphics[width=0.20\textwidth, clip, trim = 4cm 2.5cm 3cm 1.5cm]{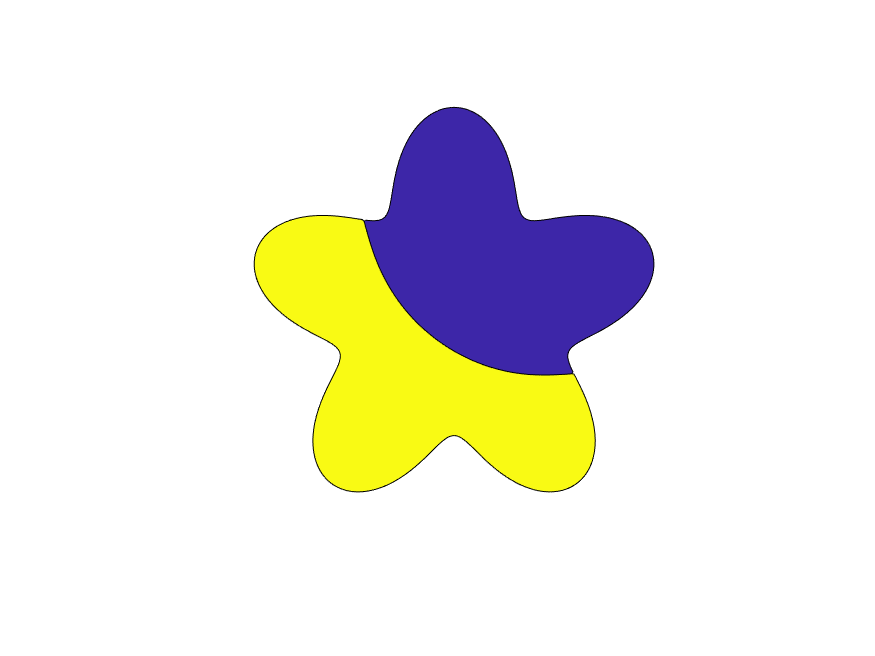}}&  
		\raisebox{-.5\height}{\includegraphics[width=0.20\textwidth, clip, trim = 4cm 2.5cm 3cm 1.5cm]{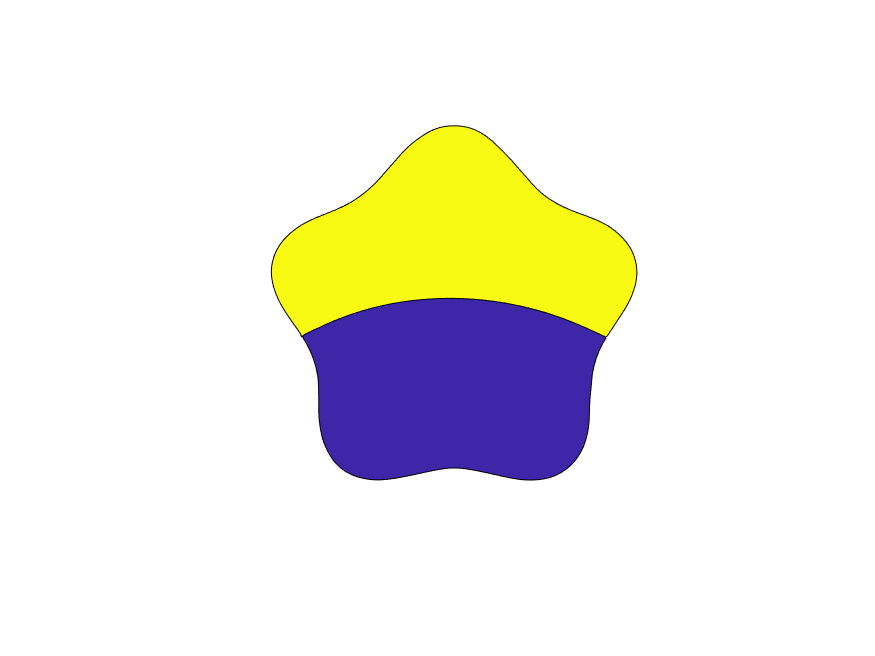}}&
		\raisebox{-.5\height}{\includegraphics[width=0.20\textwidth, clip, trim = 4cm 2.5cm 3cm 1.5cm]{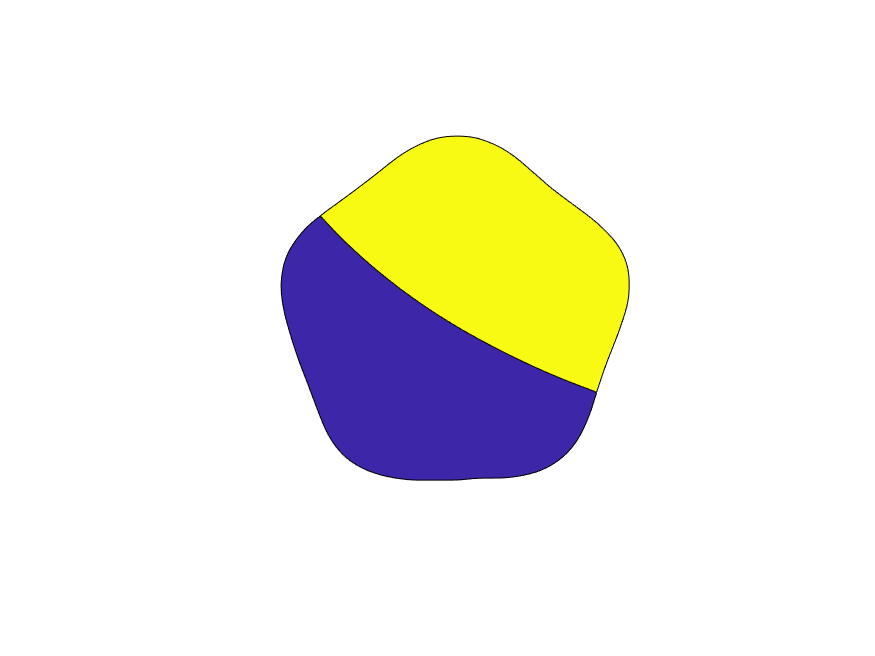}}&
		\raisebox{-.5\height}{\includegraphics[width=0.20\textwidth, clip, trim = 4cm 2.5cm 3cm 1.5cm]{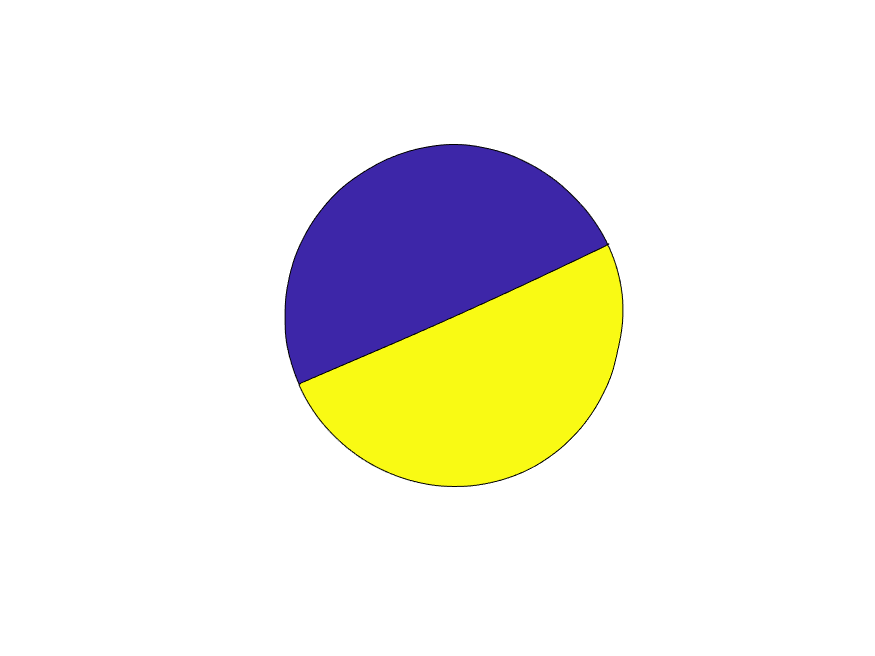}}\\
		\hline
	\end{tabular}
	\caption{Two uniform partitions in two dimensions. The initial condition is set to be a five-petal flower. The total regions and their shortest partitions of both methods in the initial, $5th$, $10th$,  and final iterations are plotted.}\label{fig: 2partition}
\end{figure}

The evolution of the approximate objective functionals $\Tilde{E}_{\tau}(u_{\Omega}^{k},(u_i^k)_{i=1}^{n})$ of both methods are plotted in Figure \ref{fig:energy 2partition}. As we can see, the final approximate objective functionals of two methods are both about $9.32$. In the $10$th iteration of the first method, the objective functional is near the final maximum objective functional, and the shape of total region shown in Figure \ref{fig: 2partition} is close to a disc. Compared with the  approximate objective functional of the second method, the approximate objective functional of first method has less fluctuation. This is because auction dynamics is repeated several times in each step of the first method, which increases the likelihood of obtaining a global shortest partition instead of a local minimizer. Thus the approximate objective functional in the first method should have a smaller gap with the exact objective functional compared to the second method.
\begin{figure}[t!]
    \centering
    \subfigure[First method]{ \includegraphics[width=0.49\textwidth, clip, trim = 1.5cm 0.8cm 0.5cm 0.5cm]{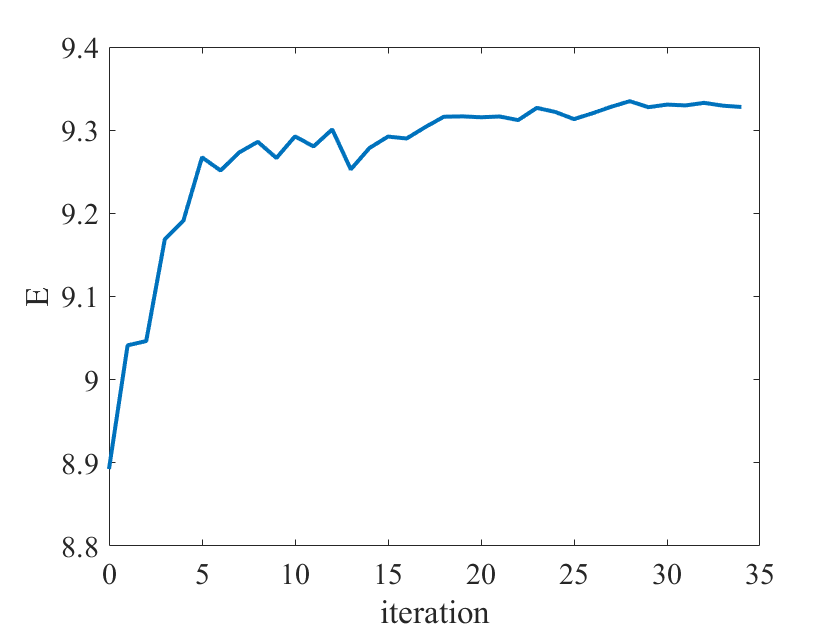}}
    \subfigure[Second method]{\includegraphics[width=0.49\textwidth, clip, trim = 1.0cm 0.8cm 0.5cm 0.5cm]{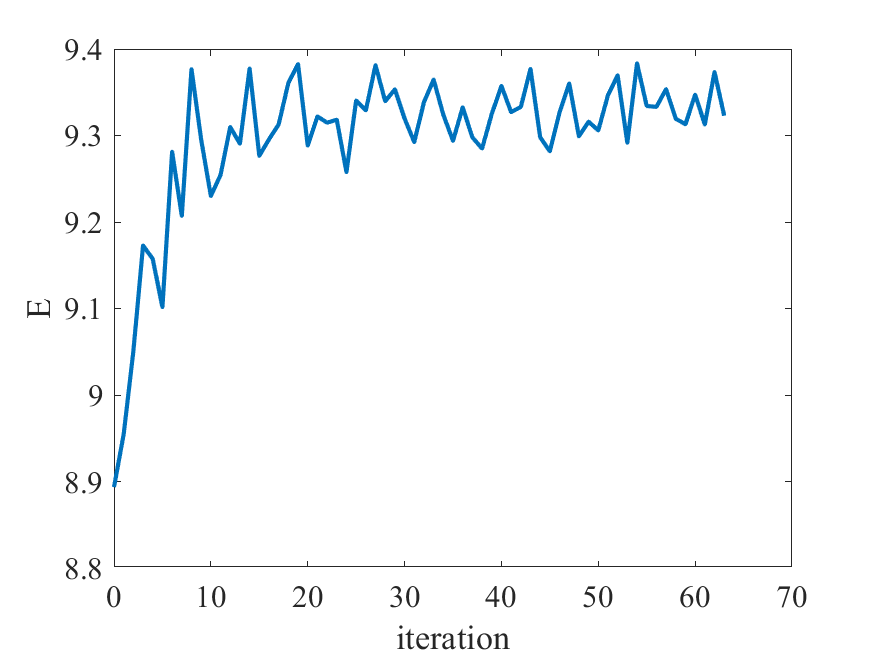}}
    \caption{Two uniform partitions in two dimensions. The approximate objective functionals $\Tilde{E}_{\tau}(u_{\Omega}^{k},(u_i^k)_{i=1}^{n})$ of two methods versus iteration times. The horizontal axis represents the number of iterations. The vertical axis represents the functional value.}
    \label{fig:energy 2partition}
\end{figure}

To further confirm that we get a disc as the optimal total region, the isoperimetric inequality is checked for the final iteration results in Figure \ref{fig: 2partition}. For any simple-connected region $\Omega$ in 2d, we have
\begin{align*}
    4\pi |\Omega|\leq |\partial \Omega|^{2},
\end{align*}
and the equality holds only when $\Omega$ is a disc. Using the indicator function of the initial condition, the area is computed to be $0.4\pi^3$. For the perimeter, we use the approximation
\begin{align*}
    |\partial\Omega|\approx \sqrt{\frac{\pi}{\tau}}\int_{\mathbb{R}^d}u_{\Omega}G_{\tau}\ast(1-u_{\Omega})\mathrm{d}\mathbf{x}.
\end{align*}
Then the values of $ 4\pi |\Omega|/ |\partial \Omega|^{2}$ for the results given by both methods are computed, which are 1.0056 and 1.0055. Here the final values are bigger than 1 due to the numerical error. Since they are close to $1$, the isoperimetric inequality corroborates that the optimal total region by each method is a disc.

As the experiments show, for the uniform two-partition case, numerical simulations imply that the disc is the maximizer for longest minimal length partitions, which is compatible with the analytical results proved in \cite{Esposito2012}. 

\subsubsection{Three partitions}

We then test the proposed methods to find the optimal solution for three partitions with equal and unequal proportions. The volume proportion parameter $\mathbf{c}$ is set to be $(\frac{1}{3},\frac{1}{3},\frac{1}{3})$, $(\frac{1}{4},\frac{1}{4},\frac{1}{2})$, $(\frac{1}{6},\frac{1}{6},\frac{2}{3})$, $(\frac{1}{6},\frac{1}{3},\frac{1}{2})$ and $(\frac{1}{10},\frac{1}{5},\frac{7}{10})$. The results of two methods are shown in Figure \ref{fig: unequal proportion}. As we can see, although the value of $\mathbf{c}$ varies, the optimal total regions given by both methods keep to be a disc, and the partition results of both methods look similar.

\begin{figure}[t!]
	\centering
	\begin{tabular}{|c|c|c|c|c|c|}
		\hline 
		  $\mathbf{c}$ & $(\frac{1}{3},\frac{1}{3},\frac{1}{3})$ & $(\frac{1}{4},\frac{1}{4},\frac{1}{2})$ & $(\frac{1}{6},\frac{1}{6},\frac{2}{3})$ & $(\frac{1}{6},\frac{1}{3},\frac{1}{2})$ & $(\frac{1}{10},\frac{1}{5},\frac{7}{10})$\\
        \hline
        \makecell{Initial\\condition} &
		\raisebox{-.5\height}{\includegraphics[width=0.15\textwidth, clip, trim = 4cm 2.5cm 3cm 1.5cm]{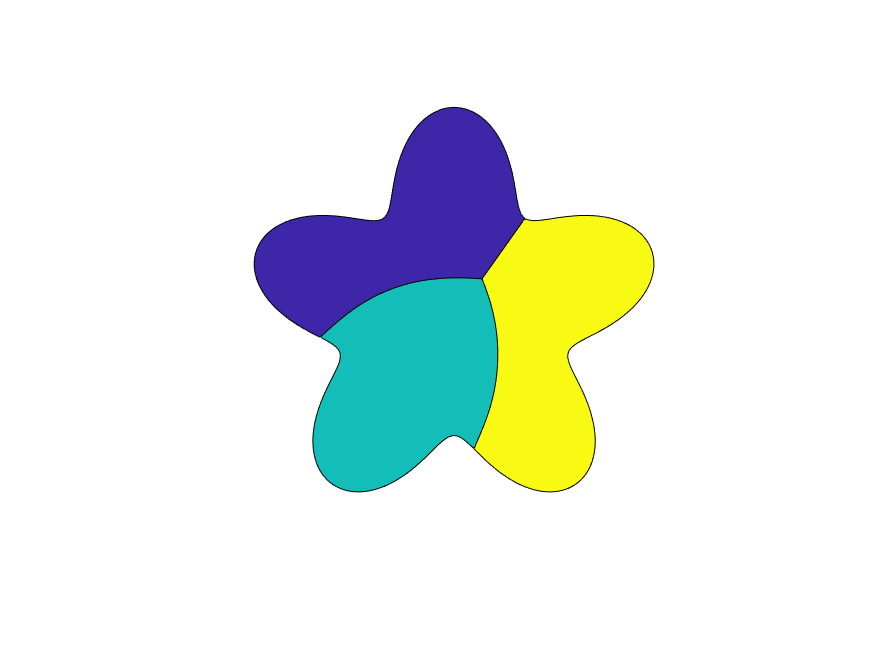}}&  
		\raisebox{-.5\height}{\includegraphics[width=0.15\textwidth, clip, trim = 4cm 2.5cm 3cm 1.5cm]{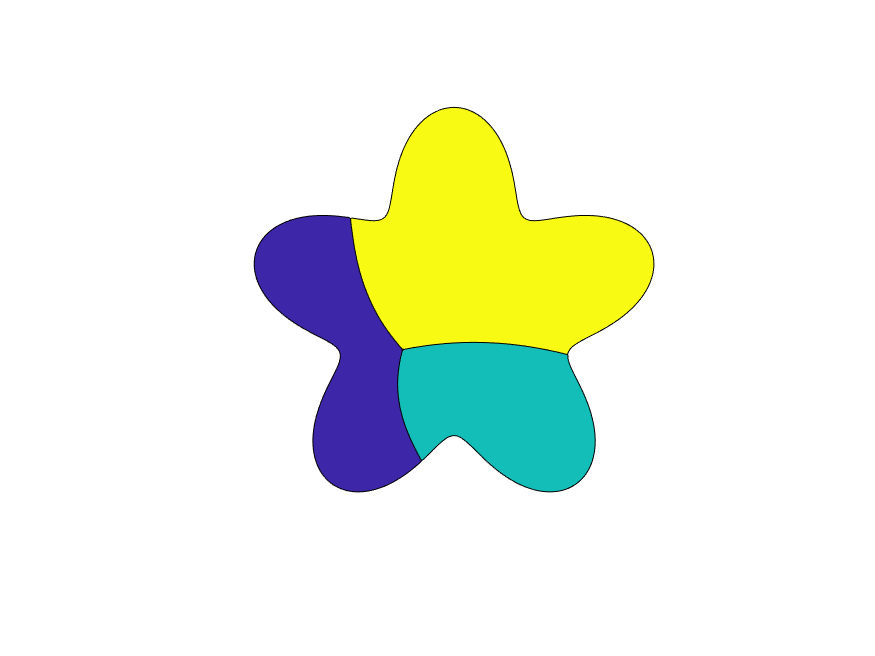}}&
		\raisebox{-.5\height}{\includegraphics[width=0.15\textwidth, clip, trim = 4cm 2.5cm 3cm 1.5cm]{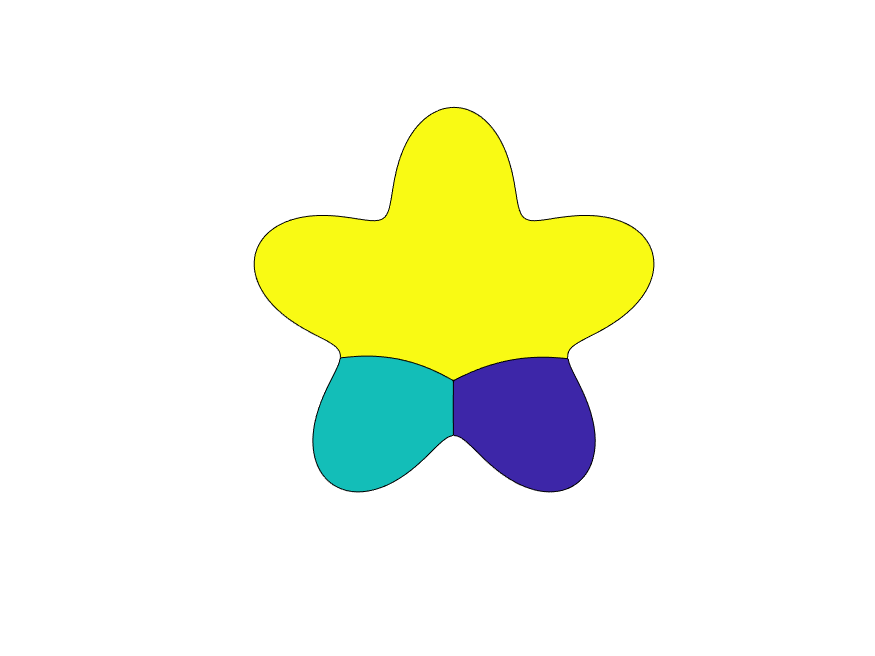}}&
        \raisebox{-.5\height}{\includegraphics[width=0.15\textwidth, clip, trim = 4cm 2.5cm 3cm 1.5cm]{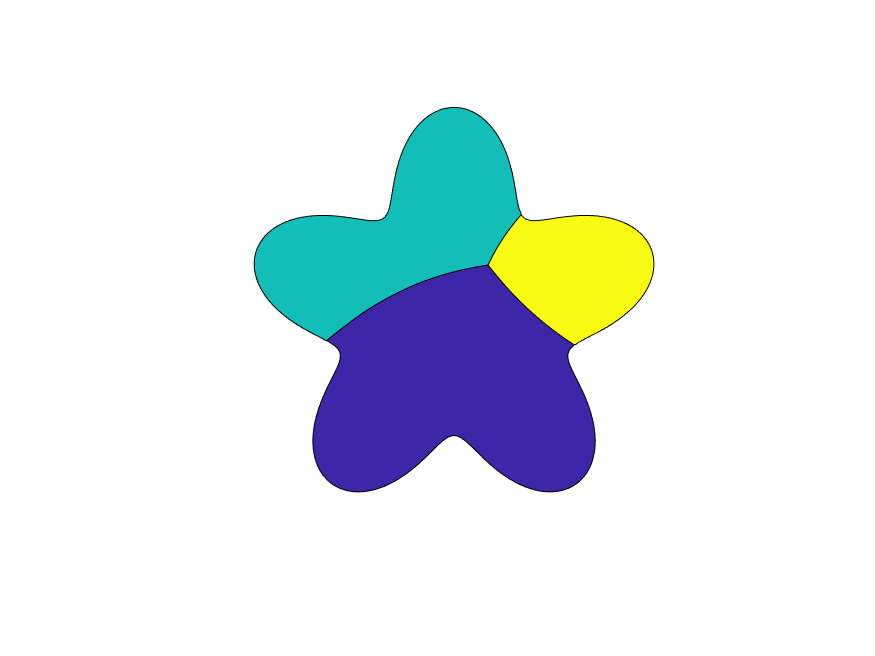}}&
		\raisebox{-.5\height}{\includegraphics[width=0.15\textwidth, clip, trim = 4cm 2.5cm 3cm 1.5cm]{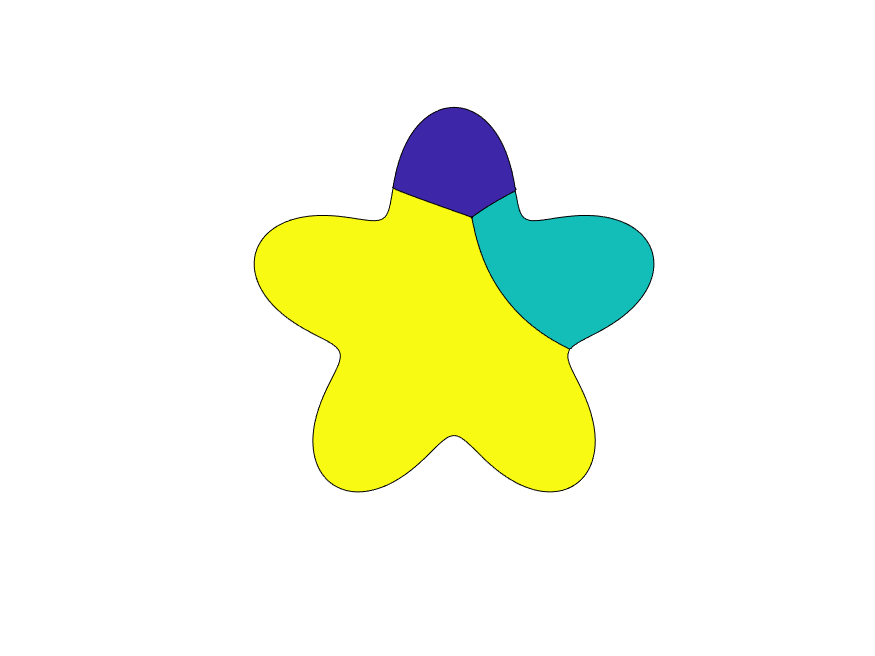}}\\ 
        \hline
        \makecell{First\\method}  &
		\raisebox{-.5\height}{\includegraphics[width=0.15\textwidth, clip, trim = 4cm 2.5cm 3cm 1.5cm]{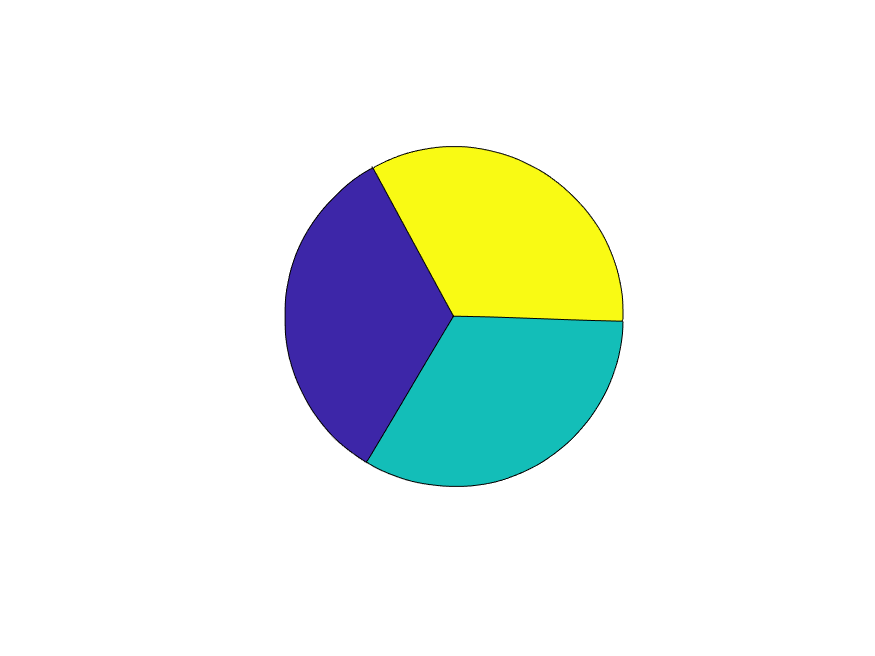}}&  
		\raisebox{-.5\height}{\includegraphics[width=0.15\textwidth, clip, trim = 4cm 2.5cm 3cm 1.5cm]{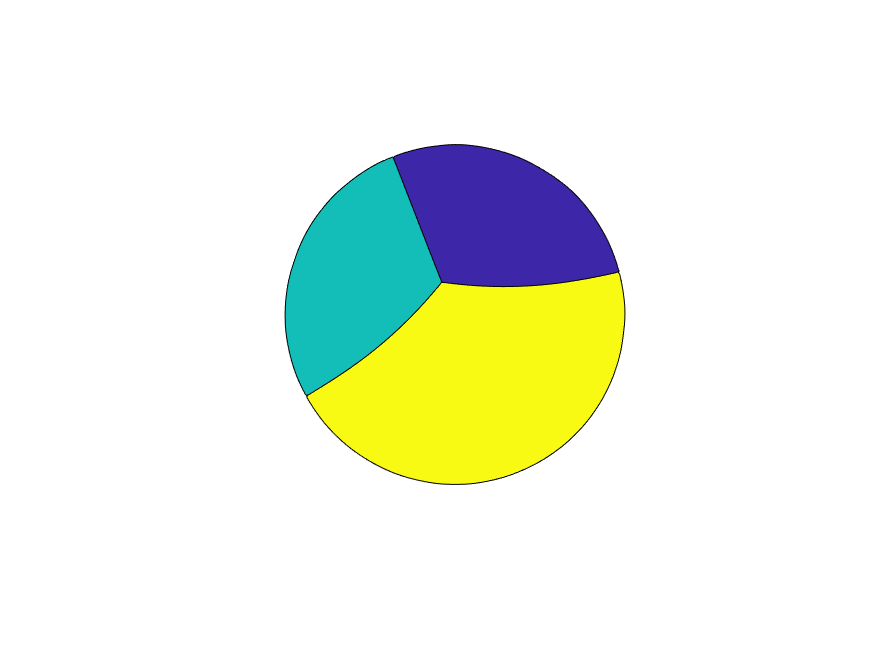}}&
		\raisebox{-.5\height}{\includegraphics[width=0.15\textwidth, clip, trim = 4cm 2.5cm 3cm 1.5cm]{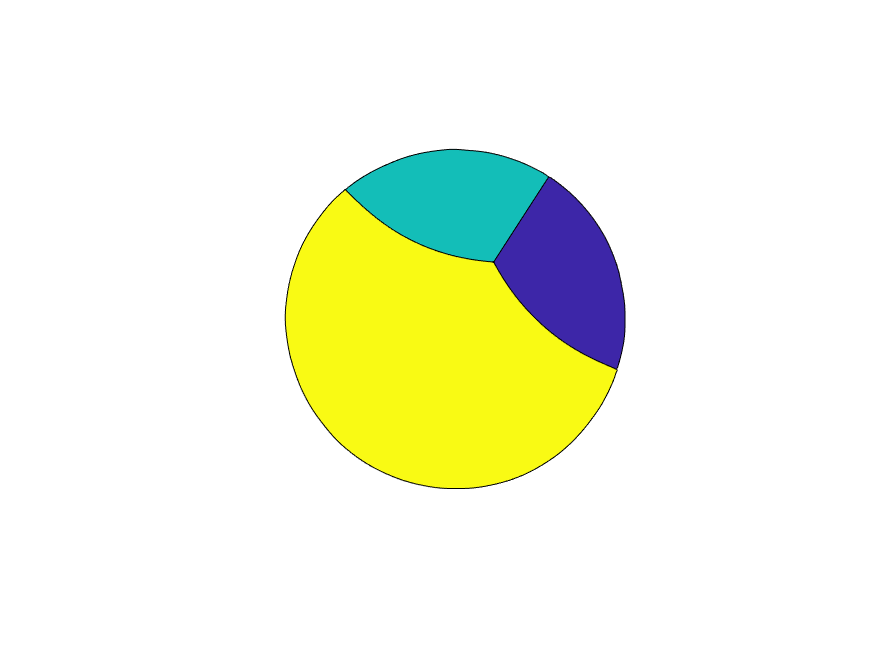}}&
		\raisebox{-.5\height}{\includegraphics[width=0.15\textwidth, clip, trim = 4cm 2.5cm 3cm 1.5cm]{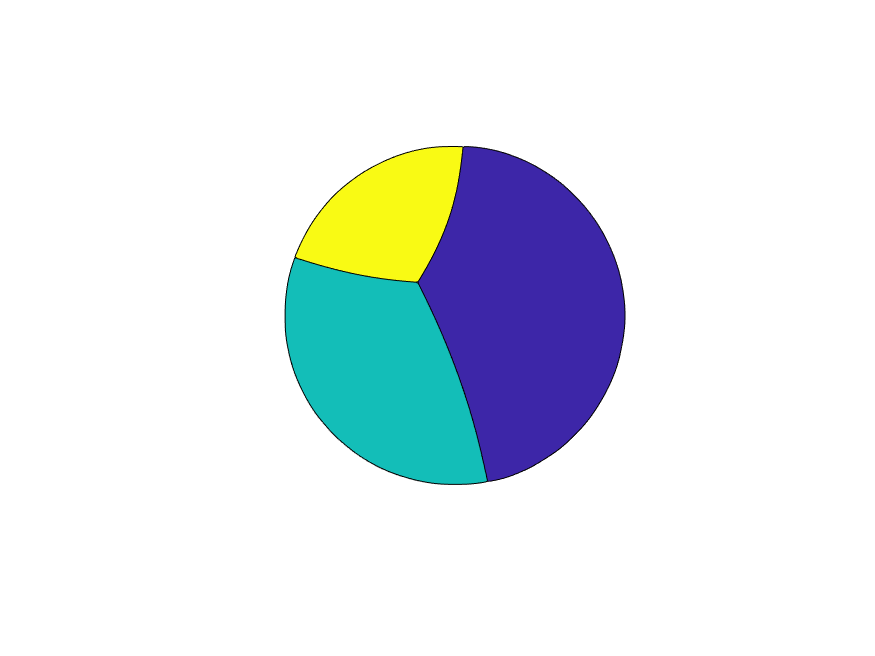}}&
		\raisebox{-.5\height}{\includegraphics[width=0.15\textwidth, clip, trim = 4cm 2.5cm 3cm 1.5cm]{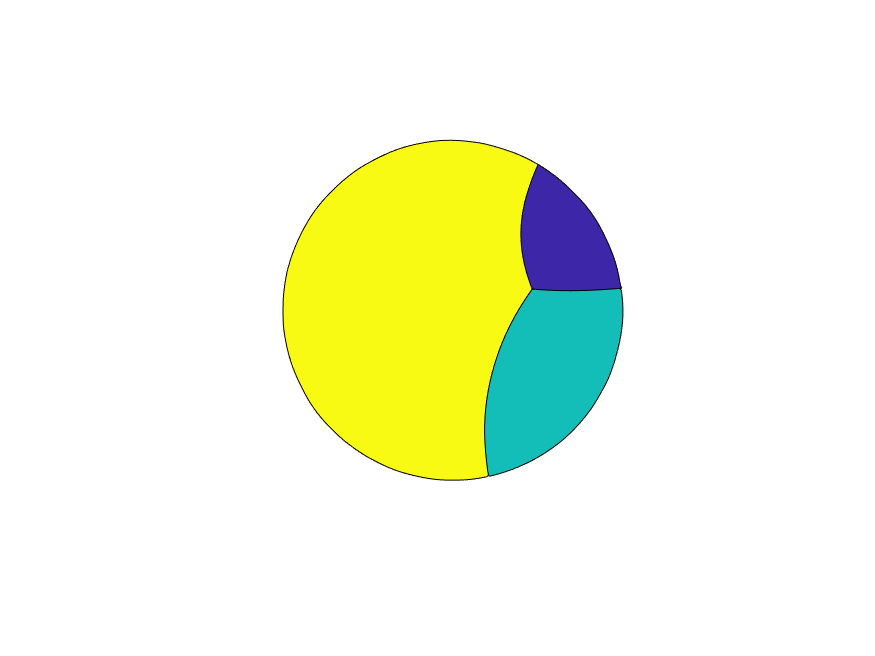}}\\
		\hline
        \makecell{Second\\method}  &
		\raisebox{-.5\height}{\includegraphics[width=0.15\textwidth, clip, trim = 4cm 2.5cm 3cm 1.5cm]{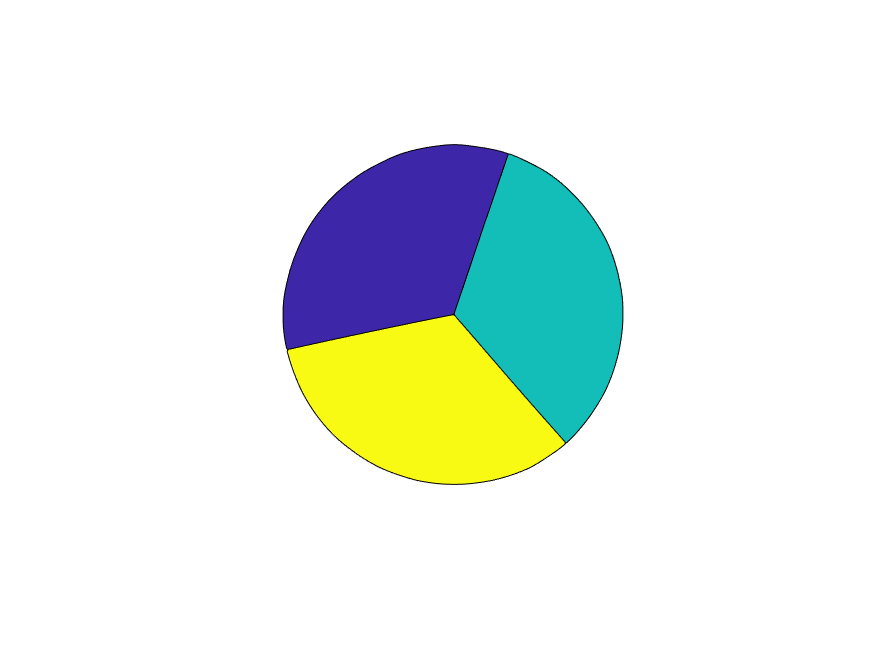}}&  
		\raisebox{-.5\height}{\includegraphics[width=0.15\textwidth, clip, trim = 4cm 2.5cm 3cm 1.5cm]{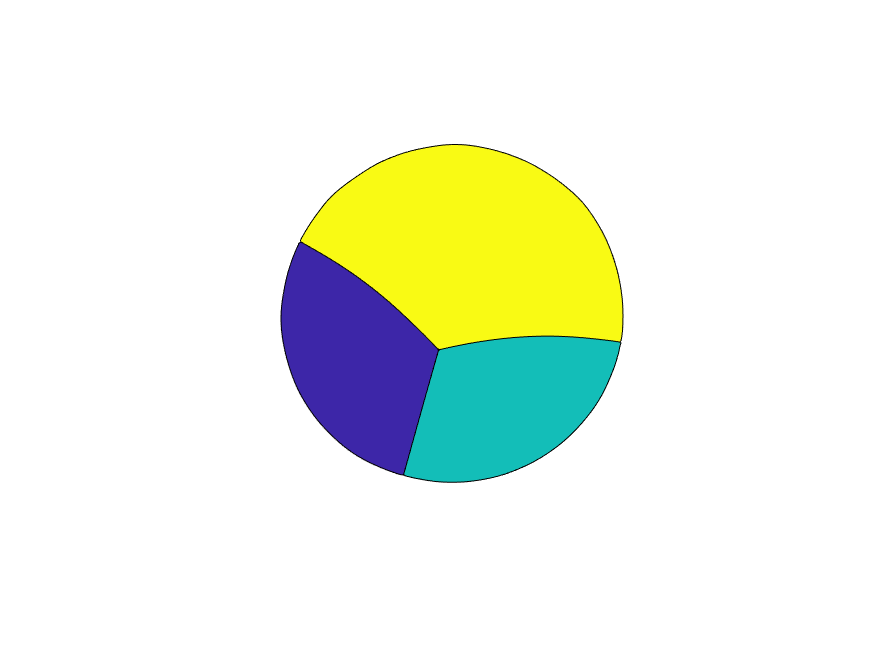}}&
		\raisebox{-.5\height}{\includegraphics[width=0.15\textwidth, clip, trim = 4cm 2.5cm 3cm 1.5cm]{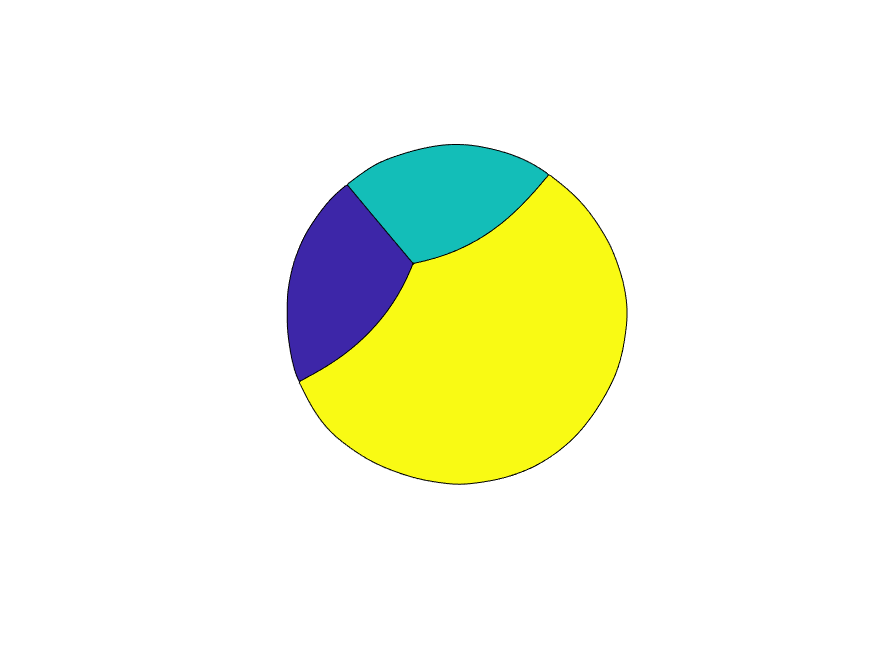}}&
		\raisebox{-.5\height}{\includegraphics[width=0.15\textwidth, clip, trim = 4cm 2.5cm 3cm 1.5cm]{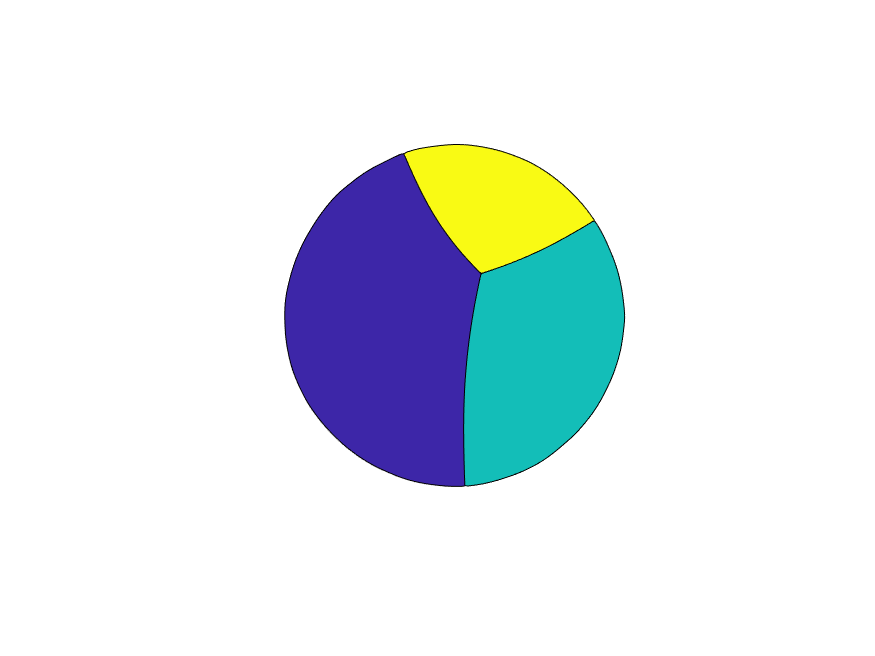}}&
		\raisebox{-.5\height}{\includegraphics[width=0.15\textwidth, clip, trim = 4cm 2.5cm 3cm 1.5cm]{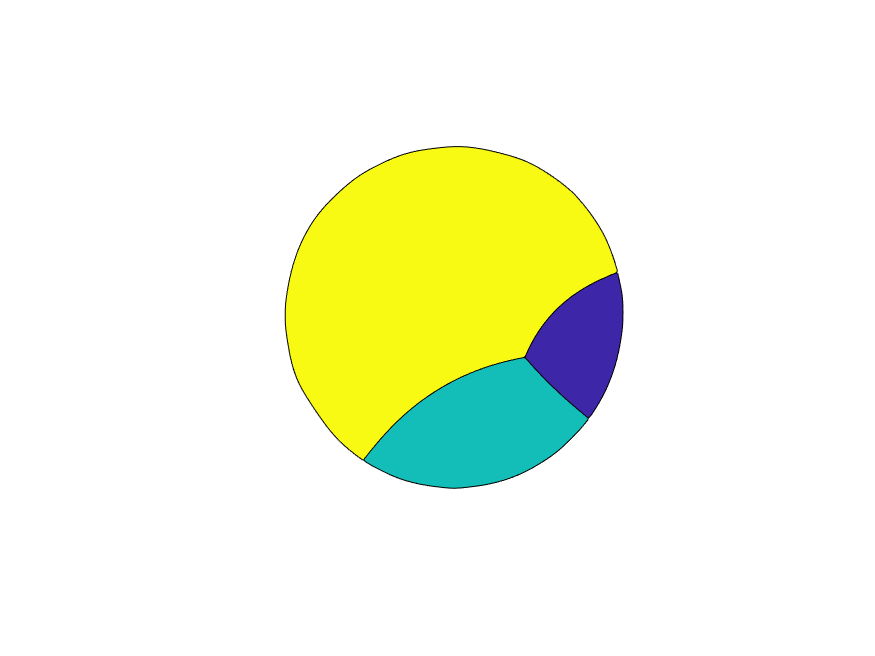}}\\
		\hline
	\end{tabular}
	\caption{Three partitions in two dimensions. The initial condition is set to be a five-petal flower. The total regions and their shortest partitions of both methods in the final iterations are plotted for different values of $\mathbf{c}$ of the three partition case.}\label{fig: unequal proportion}
\end{figure}

\subsubsection{Sensitivity to initial shape}

We test the sensitivity of the proposed methods to the initial shape. In this experiment, both methods are used to find the optimal solution for three partitions with equal proportions. The shape of initial condition is set to be a triangle, a rectangle, and randomly generated pentagons. The results are shown in Figure \ref{fig: different shape}. Both methods give the optimal total region to be a disc no matter how the shape of initial condition changes.

\begin{figure}[t!]
	\centering
	\begin{tabular}{|c|c|c|c|c|c|}
		\hline 
		   & Triangle & Rectangle & \multicolumn{3}{|c|}{Random pentagons}\\
        \hline
        \makecell{Initial\\condition} &
		\raisebox{-.5\height}{\includegraphics[width=0.15\textwidth, clip, trim = 4cm 2.5cm 3cm 1.5cm]{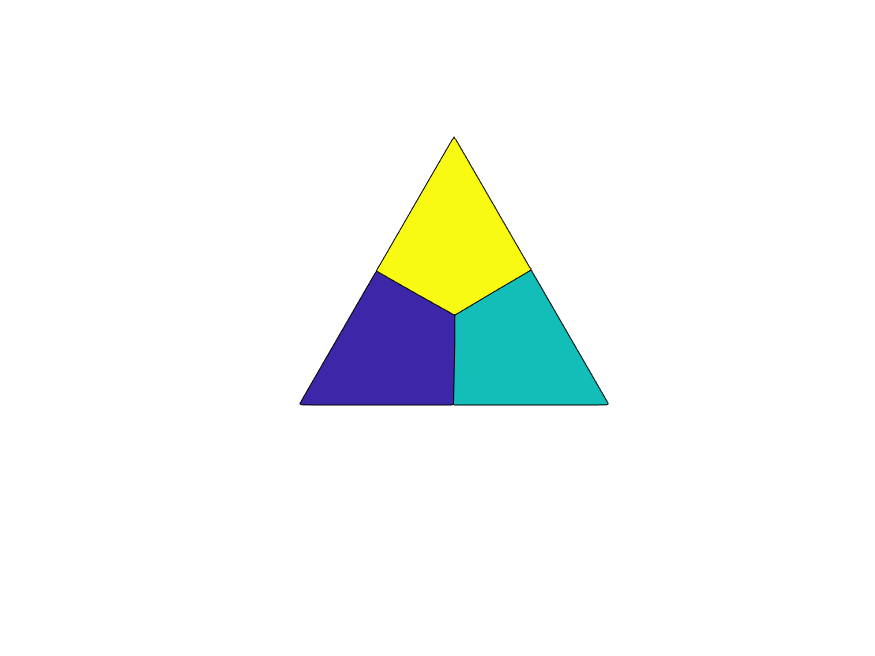}}&  
		\raisebox{-.5\height}{\includegraphics[width=0.15\textwidth, clip, trim = 3.5cm 2cm 2cm 1.5cm]{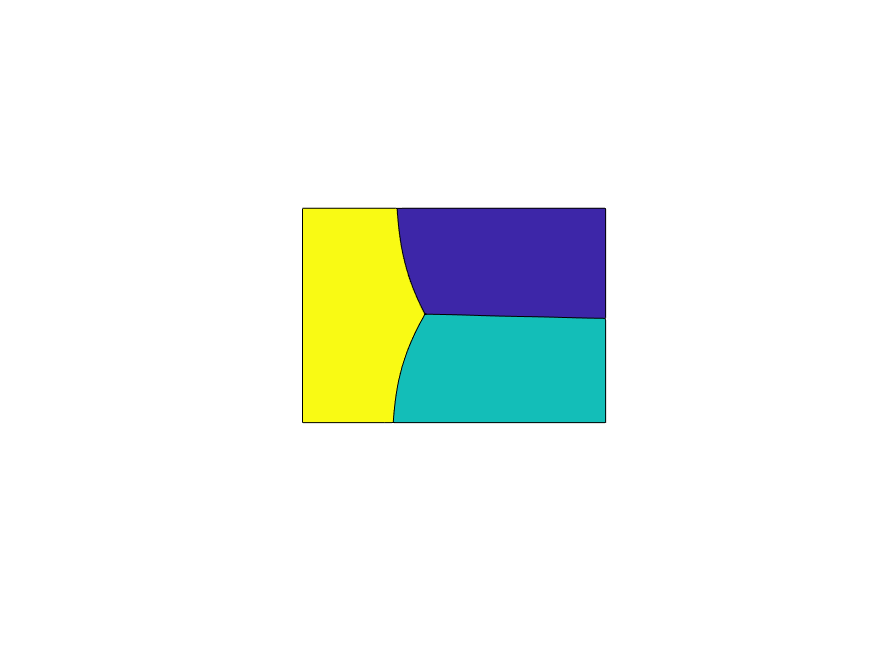}}&
		\raisebox{-.5\height}{\includegraphics[width=0.15\textwidth, clip, trim = 4.5cm 3.5cm 4cm 2cm]{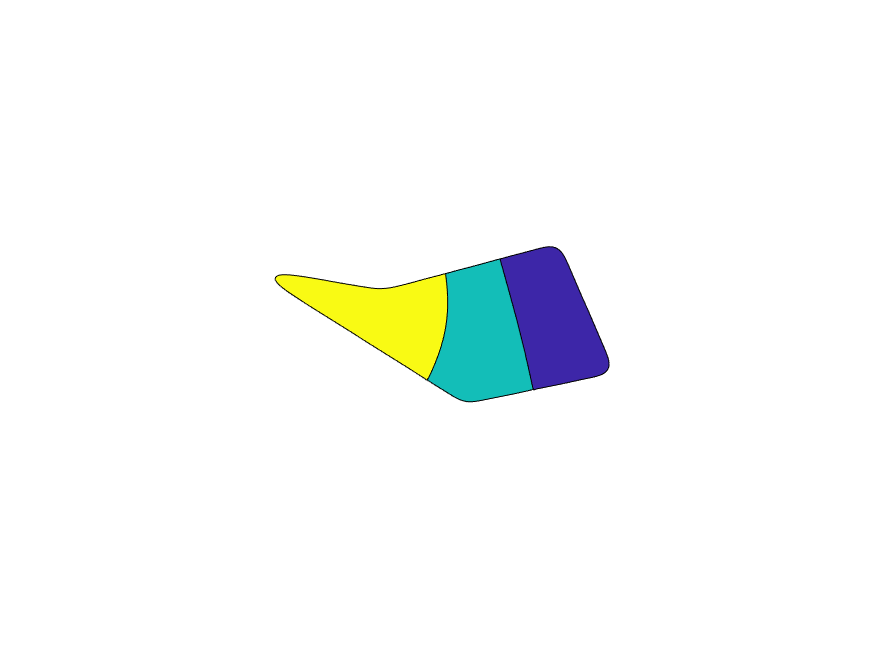}}&
        \raisebox{-.5\height}{\includegraphics[width=0.15\textwidth, clip, trim = 4.5cm 3.5cm 4cm 2cm]{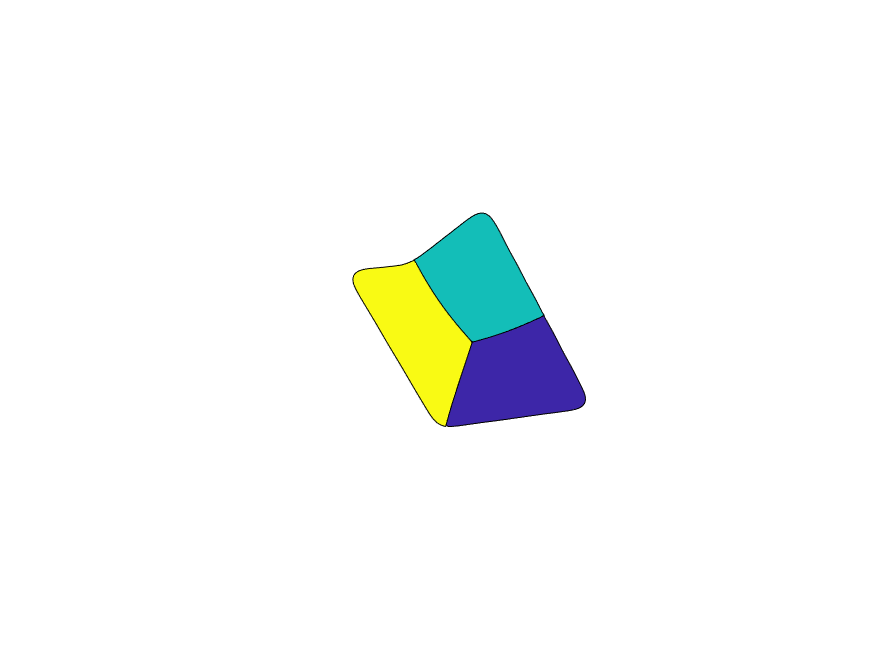}}&
		\raisebox{-.5\height}{\includegraphics[width=0.15\textwidth, clip, trim = 4.5cm 3.5cm 4cm 2cm]{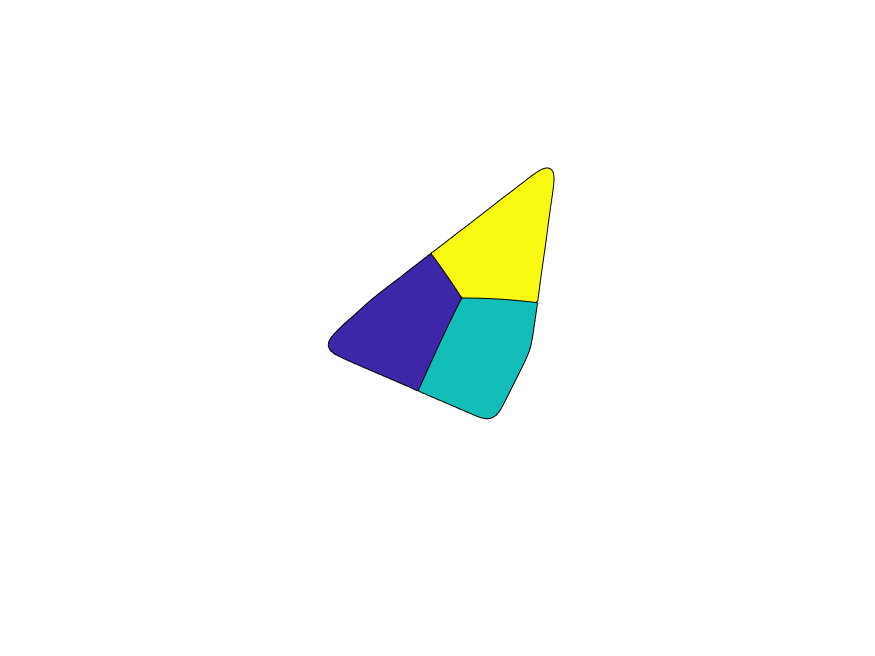}}\\ 
        \hline
        \makecell{First\\method}  &
		\raisebox{-.5\height}{\includegraphics[width=0.15\textwidth, clip, trim = 4cm 2.5cm 3cm 1.5cm]{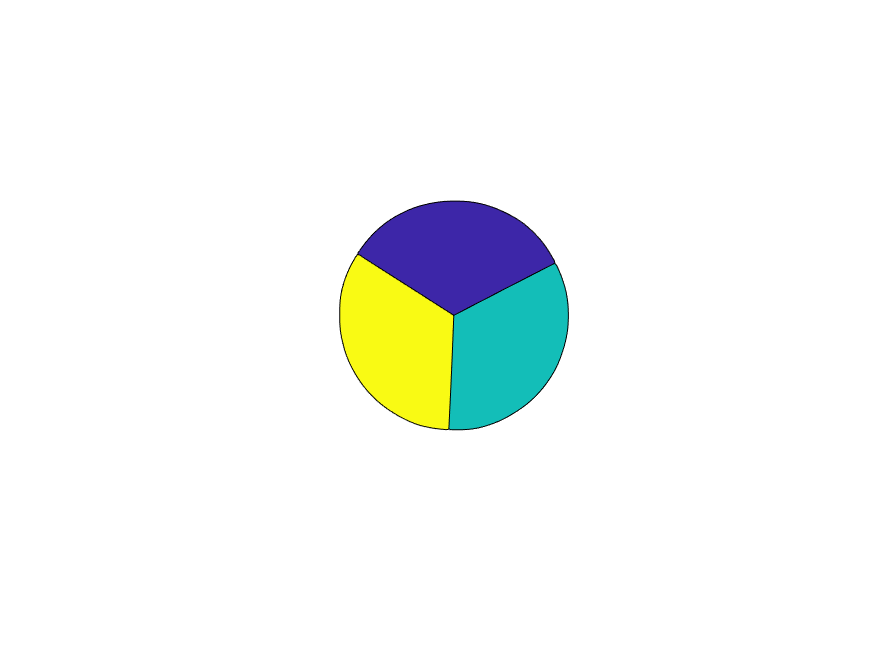}}&  
		\raisebox{-.5\height}{\includegraphics[width=0.15\textwidth, clip, trim = 3.5cm 2cm 2cm 1.5cm]{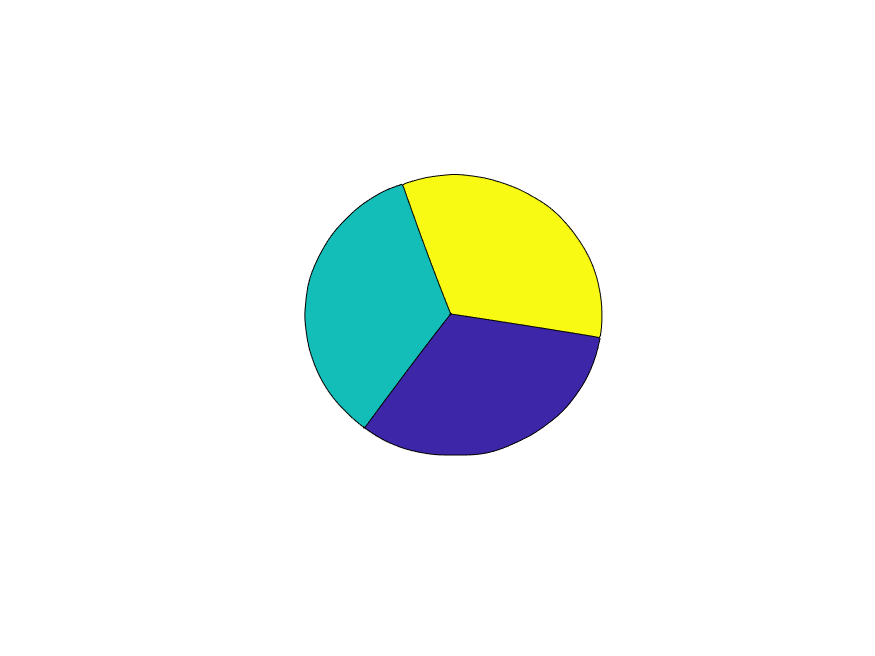}}&
		\raisebox{-.5\height}{\includegraphics[width=0.15\textwidth, clip, trim = 4.5cm 3cm 4cm 2.5cm]{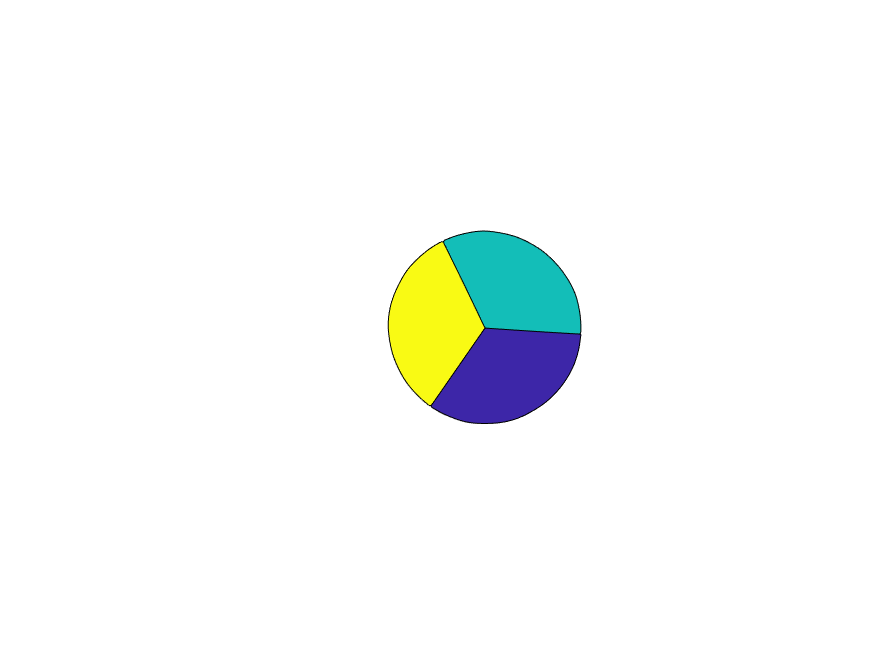}}&
		\raisebox{-.5\height}{\includegraphics[width=0.15\textwidth, clip, trim = 4.5cm 3cm 4cm 2.5cm]{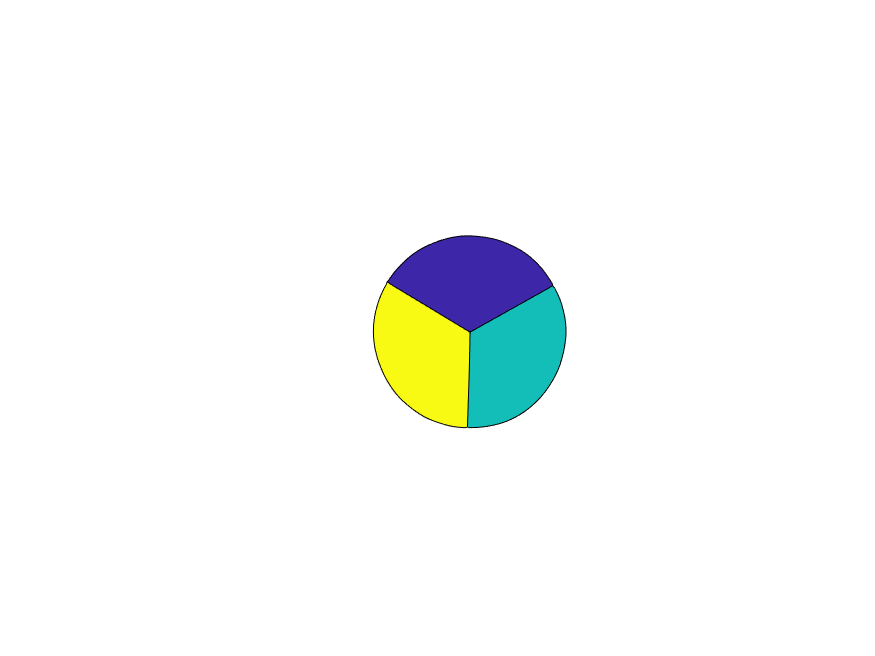}}&
		\raisebox{-.5\height}{\includegraphics[width=0.15\textwidth, clip, trim = 4.5cm 3.5cm 4cm 2cm]{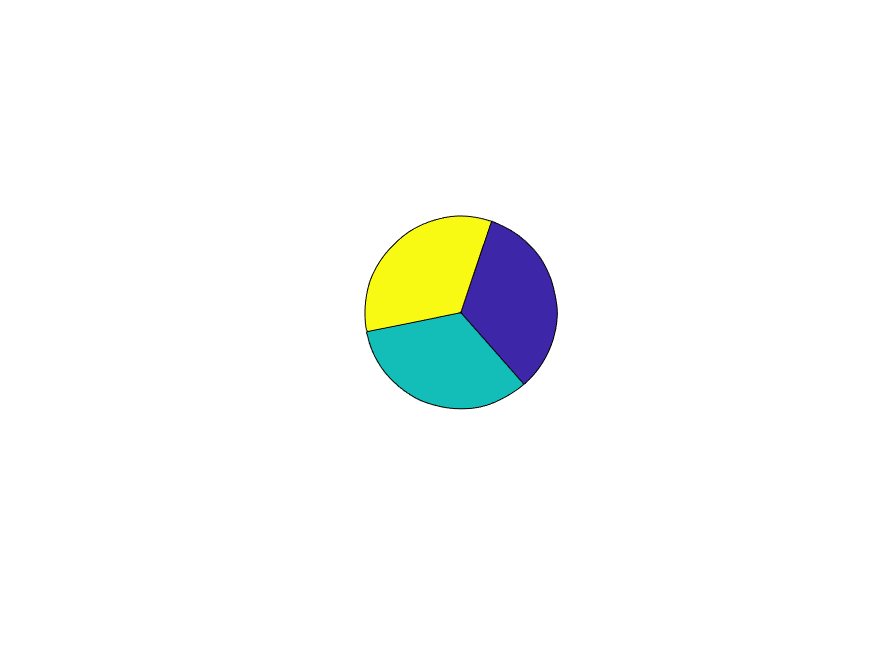}}\\
		\hline
        \makecell{Second\\method}  &
		\raisebox{-.5\height}{\includegraphics[width=0.15\textwidth, clip, trim = 4cm 2.5cm 3cm 1.5cm]{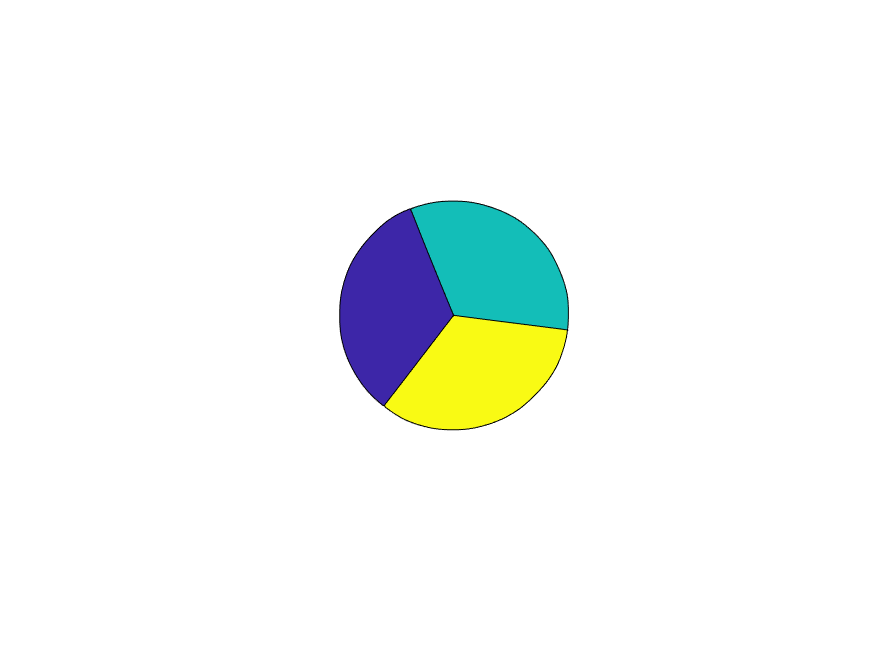}}&  
		\raisebox{-.5\height}{\includegraphics[width=0.15\textwidth, clip, trim = 3.5cm 2cm 2cm 1.5cm]{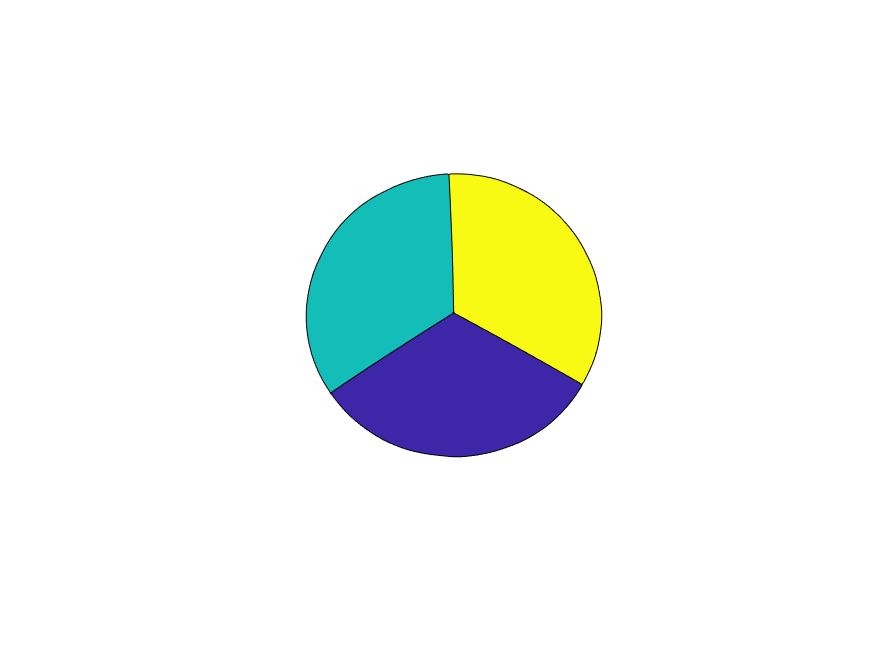}}&
		\raisebox{-.5\height}{\includegraphics[width=0.15\textwidth, clip, trim = 4.5cm 3cm 4cm 2.5cm]{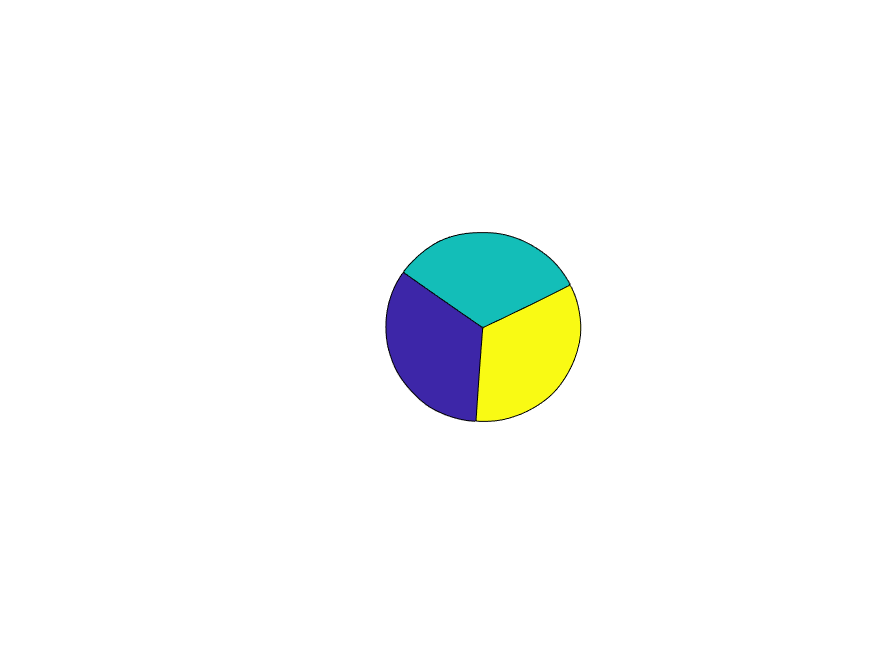}}&
		\raisebox{-.5\height}{\includegraphics[width=0.15\textwidth, clip, trim = 4.5cm 3cm 4cm 2.5cm]{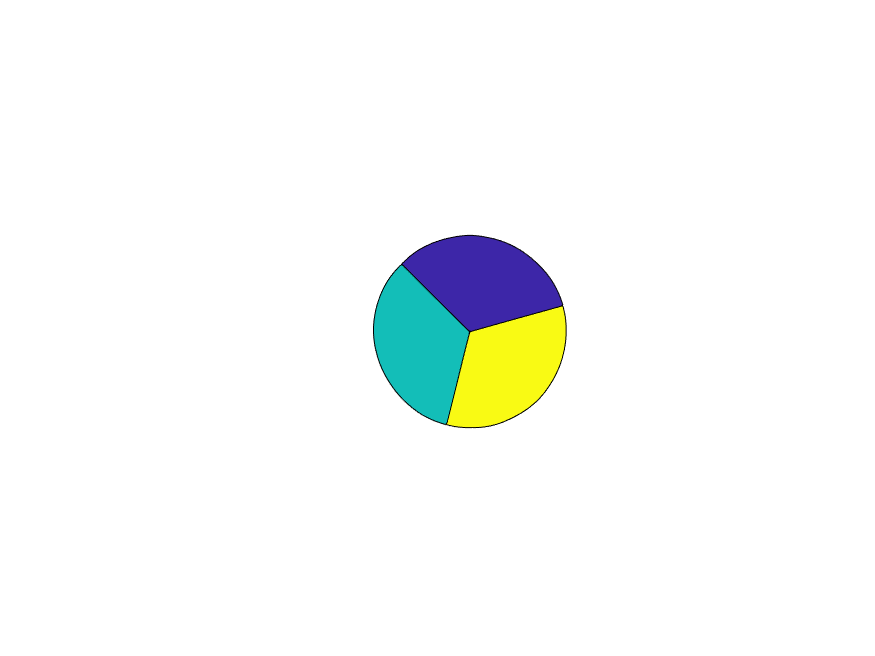}}&
		\raisebox{-.5\height}{\includegraphics[width=0.15\textwidth, clip, trim = 4.5cm 3.5cm 4cm 2cm]{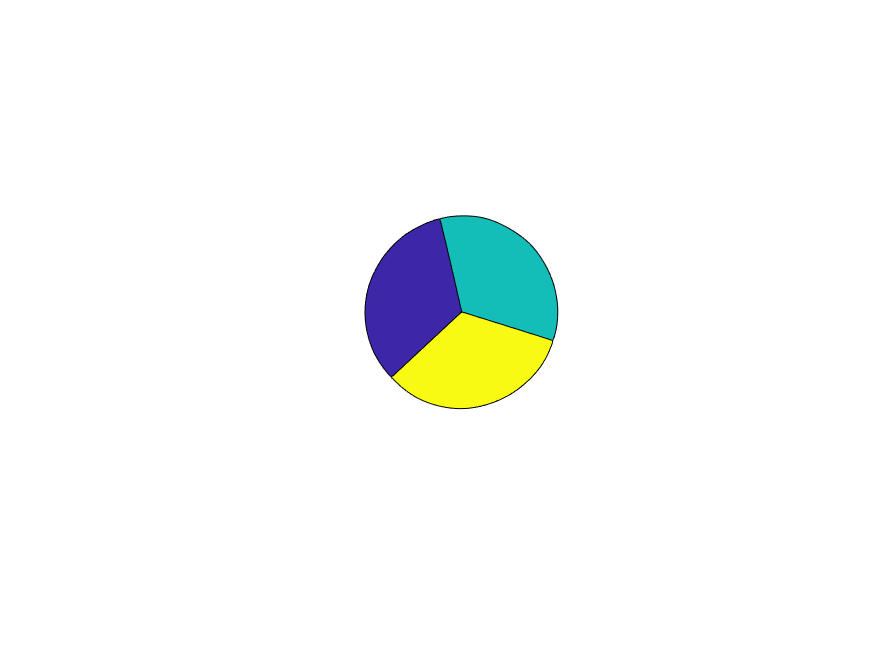}}\\
		\hline
	\end{tabular}
	\caption{Three partitions in two dimensions, sensitivity to initial shape. The proportion is set to be $\mathbf{c}=(\frac{1}{3},\frac{1}{3},\frac{1}{3})$. The total regions and their shortest partitions of both methods in the final iterations are plotted for different shape of initial conditions.}\label{fig: different shape}
\end{figure}

\subsubsection{Six and nine partitions}
In this set of experiments, the initial condition is set to be the five-petal flower, whose indicator function is given in Section \ref{sec.numerical.2d.two}. The two methods are tested to solve the six-partition and nine-partition problems. For each problem, the volume proportions of every part in the partitions are set to be equal. The total regions and their shortest partitions of both methods in the initial, 5th, 10th, 20th and final iterations are plotted in Figure \ref{fig: multi-partition}.

\begin{figure}[t!]
	\centering
	\begin{tabular}{|c|c|c|c|c|c|}
		\hline 
		 & Initial & 5th iteration & 10th iteration & Final  \\
        \hline
        \makecell{First\\method\\six\\ partitions} &
		\raisebox{-.5\height}{\includegraphics[width=0.15\textwidth, clip, trim = 4cm 2.5cm 3cm 1.5cm]{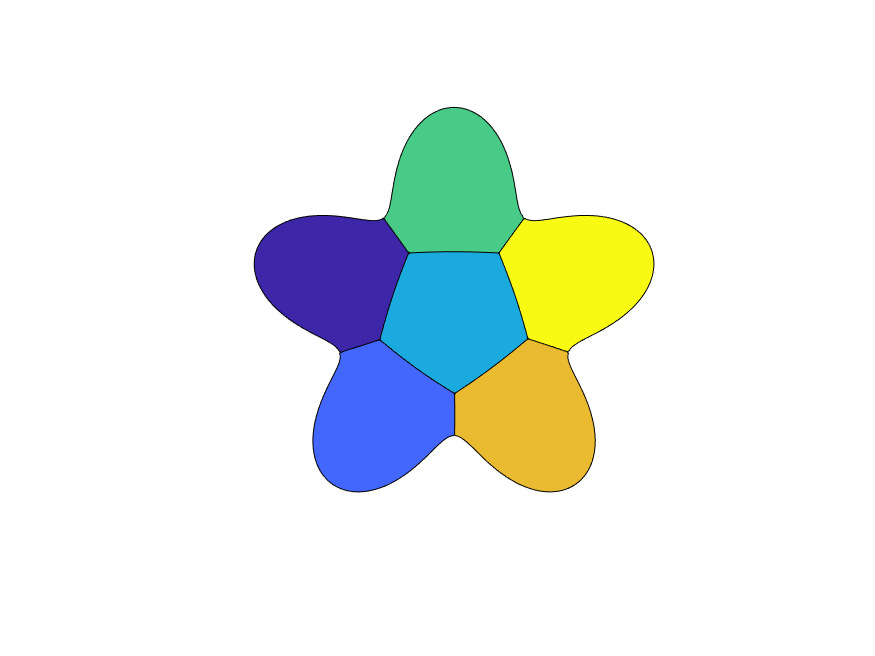}}&  
		\raisebox{-.5\height}{\includegraphics[width=0.15\textwidth, clip, trim = 4cm 2.5cm 3cm 1.5cm]{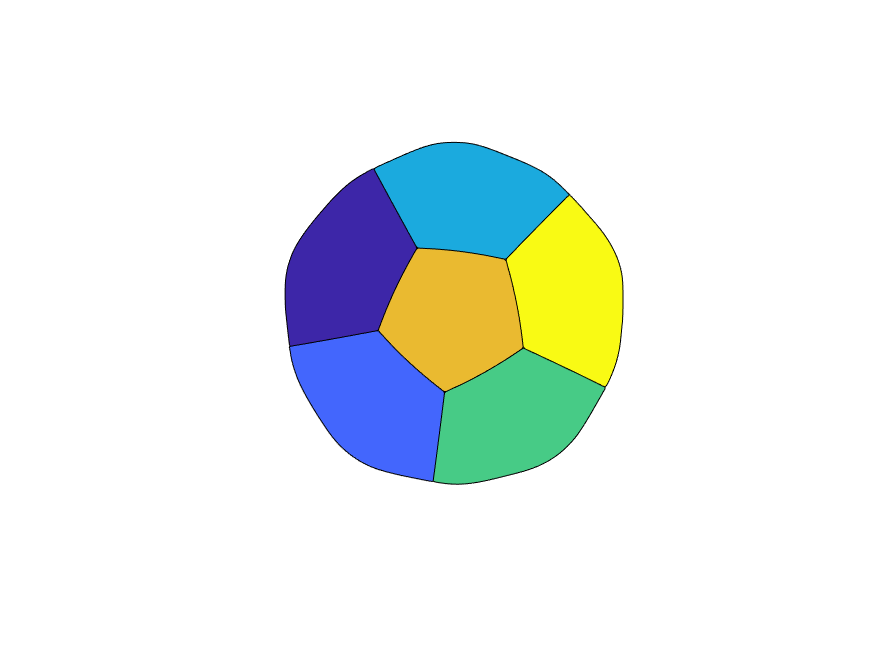}}&
		\raisebox{-.5\height}{\includegraphics[width=0.15\textwidth, clip, trim = 4cm 2.5cm 3cm 1.5cm]{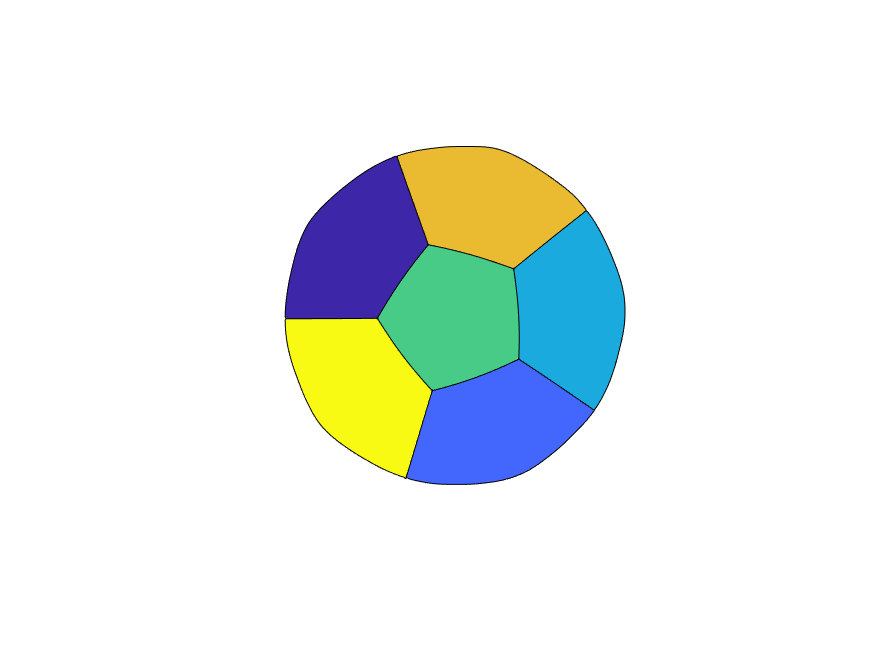}}&
		\raisebox{-.5\height}{\includegraphics[width=0.15\textwidth, clip, trim = 4cm 2.5cm 3cm 1.5cm]{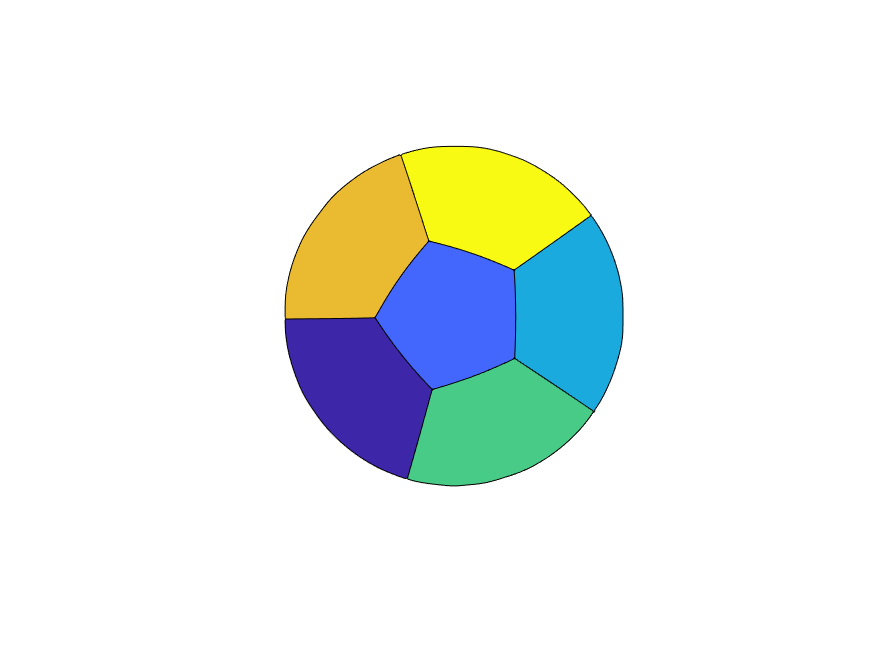}}\\ 
        \hline
        \makecell{Second\\method\\six\\ partitions}  &
		\raisebox{-.5\height}{\includegraphics[width=0.15\textwidth, clip, trim = 4cm 2.5cm 3cm 1.5cm]{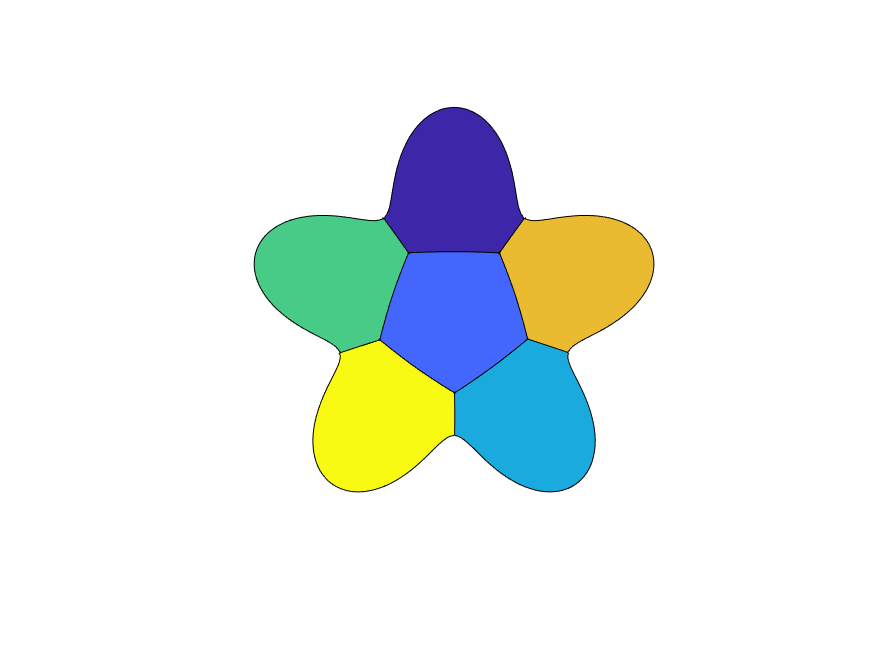}}&  
		\raisebox{-.5\height}{\includegraphics[width=0.15\textwidth, clip, trim = 4cm 2.5cm 3cm 1.5cm]{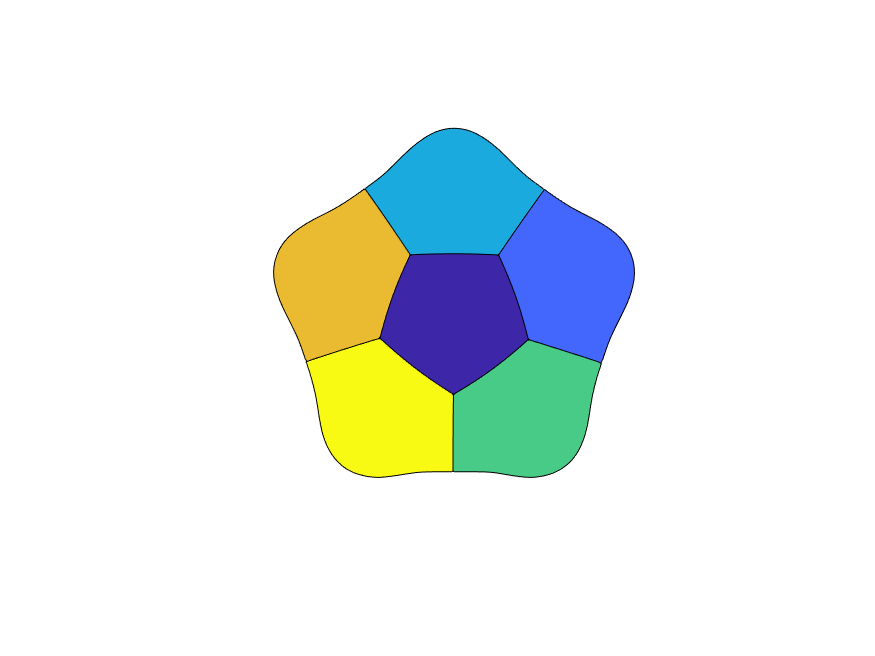}}&
		\raisebox{-.5\height}{\includegraphics[width=0.15\textwidth, clip, trim = 4cm 2.5cm 3cm 1.5cm]{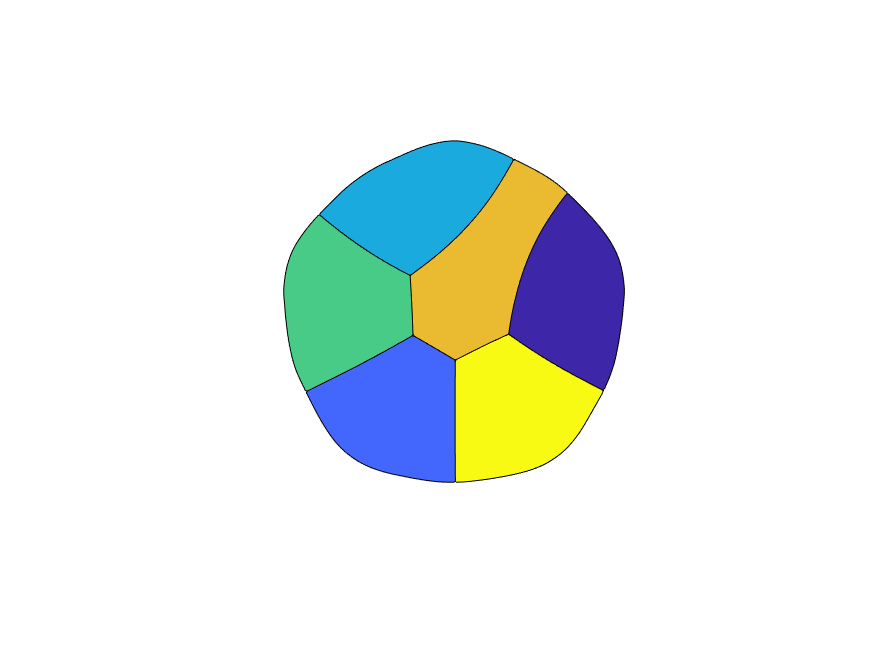}}&
		\raisebox{-.5\height}{\includegraphics[width=0.15\textwidth, clip, trim = 4cm 2.5cm 3cm 1.5cm]{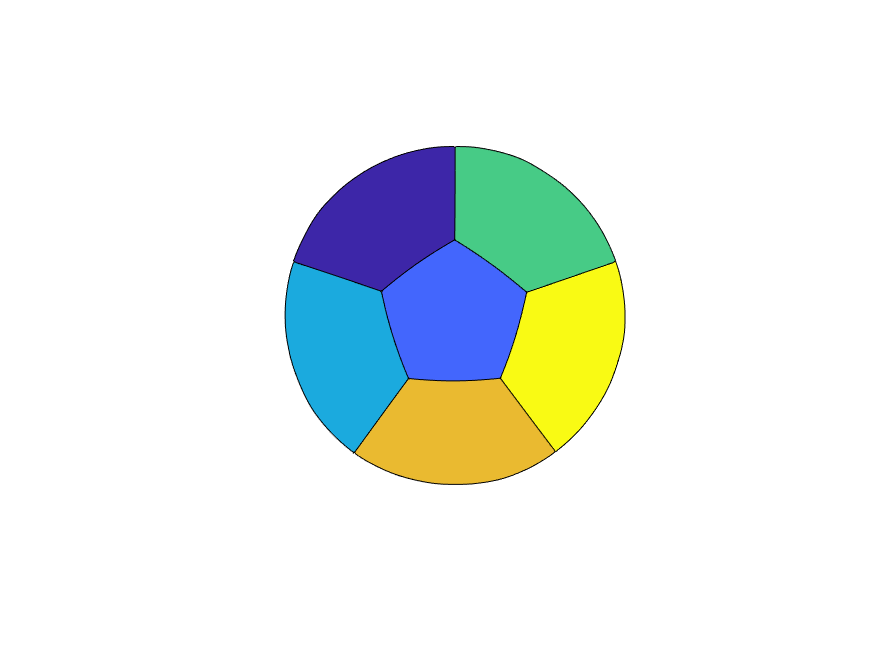}}\\
		\hline
        \makecell{First\\method\\nine\\ partitions}  &
		\raisebox{-.5\height}{\includegraphics[width=0.15\textwidth, clip, trim = 4cm 2.5cm 3cm 1.5cm]{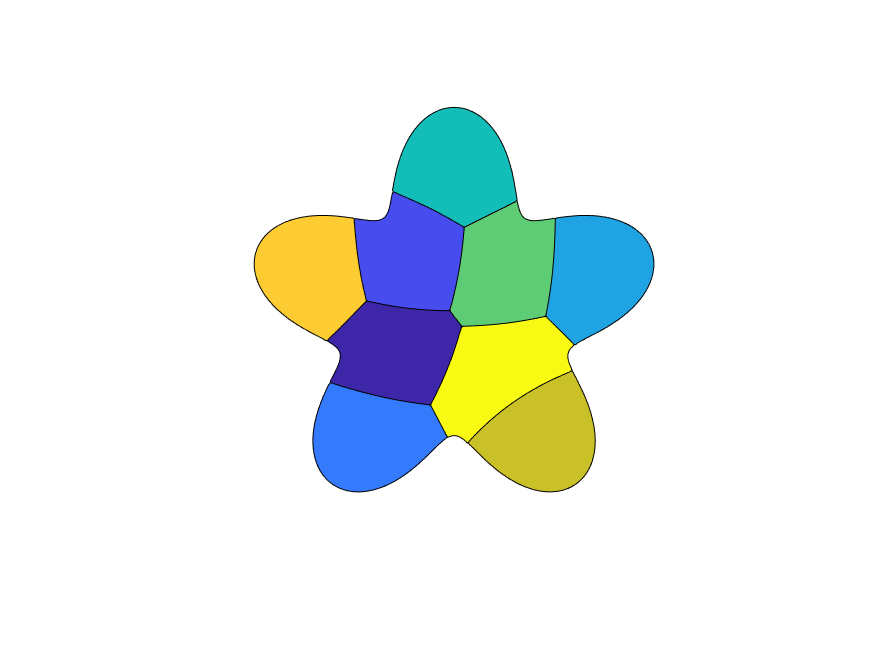}}&  
		\raisebox{-.5\height}{\includegraphics[width=0.15\textwidth, clip, trim = 4cm 2.5cm 3cm 1.5cm]{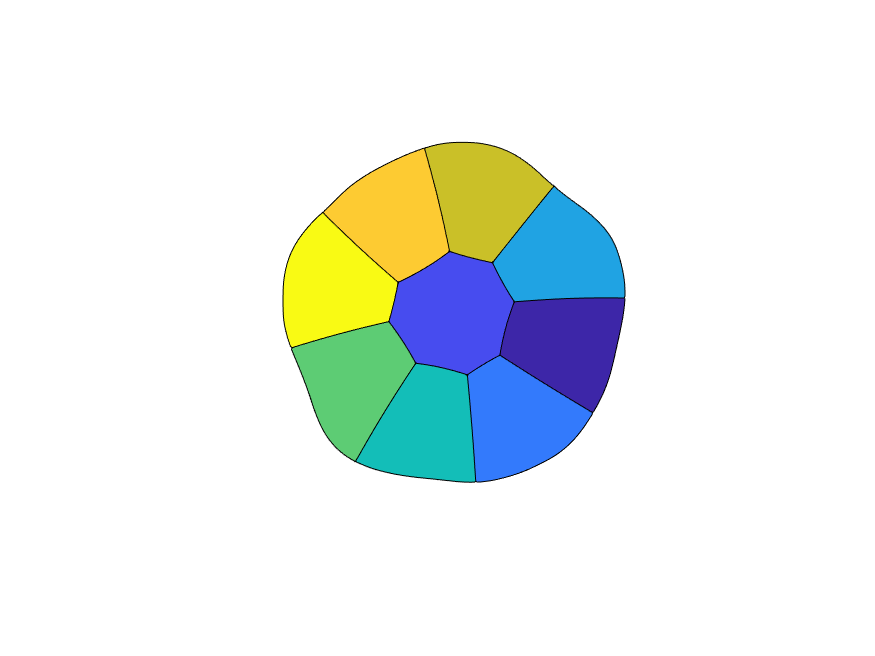}}&
		\raisebox{-.5\height}{\includegraphics[width=0.15\textwidth, clip, trim = 4cm 2.5cm 3cm 1.5cm]{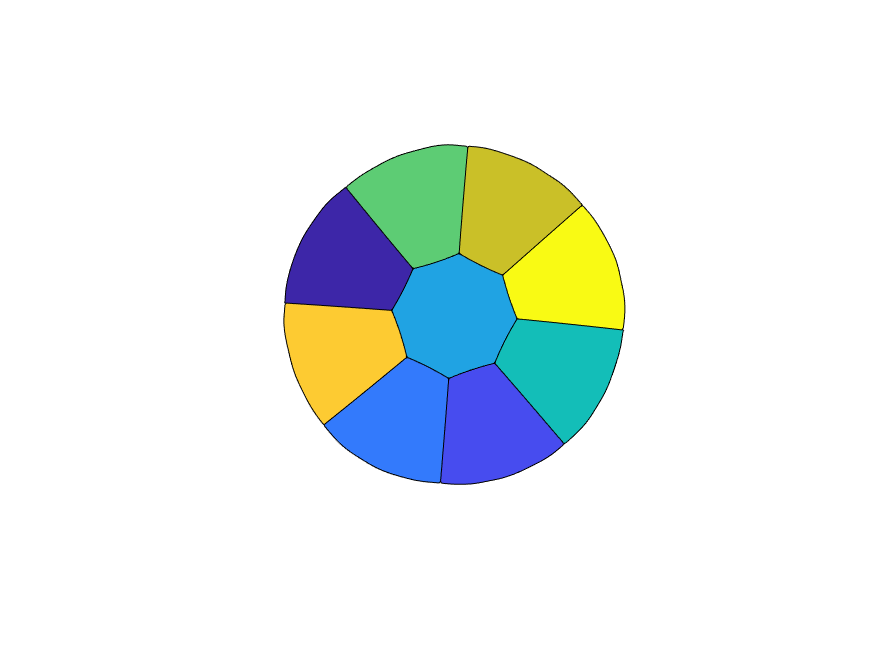}}&
		\raisebox{-.5\height}{\includegraphics[width=0.15\textwidth, clip, trim = 4cm 2.5cm 3cm 1.5cm]{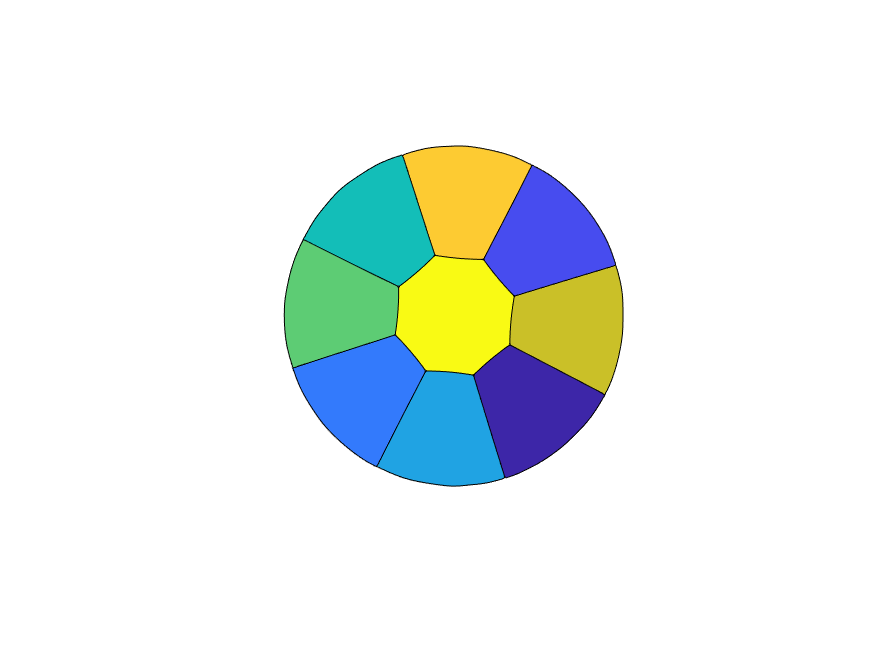}}\\
		\hline
        \makecell{Second\\method\\nine\\ partitions}  &
		\raisebox{-.5\height}{\includegraphics[width=0.15\textwidth, clip, trim = 4cm 2.5cm 3cm 1.5cm]{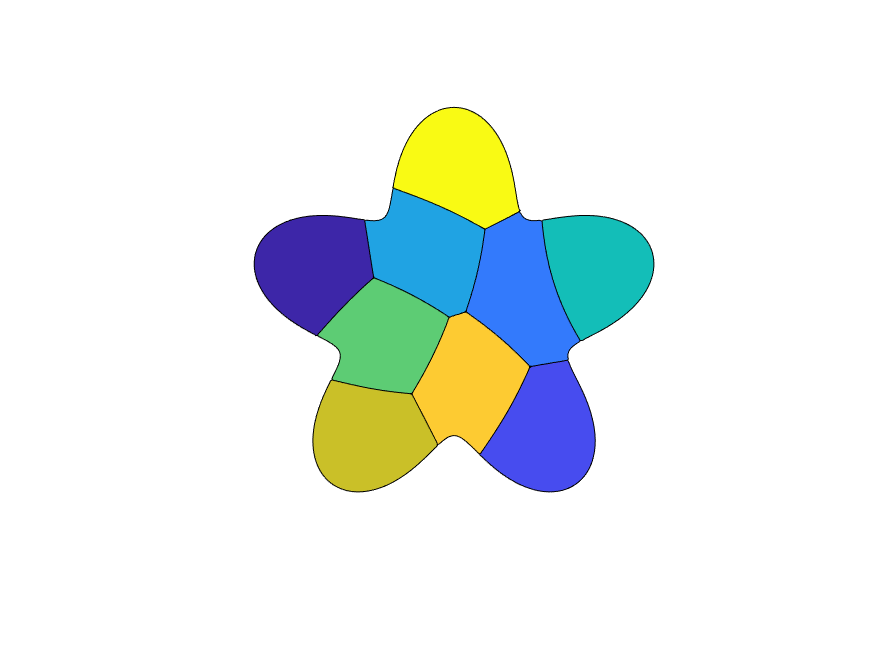}}&  
		\raisebox{-.5\height}{\includegraphics[width=0.15\textwidth, clip, trim = 4cm 2.5cm 3cm 1.5cm]{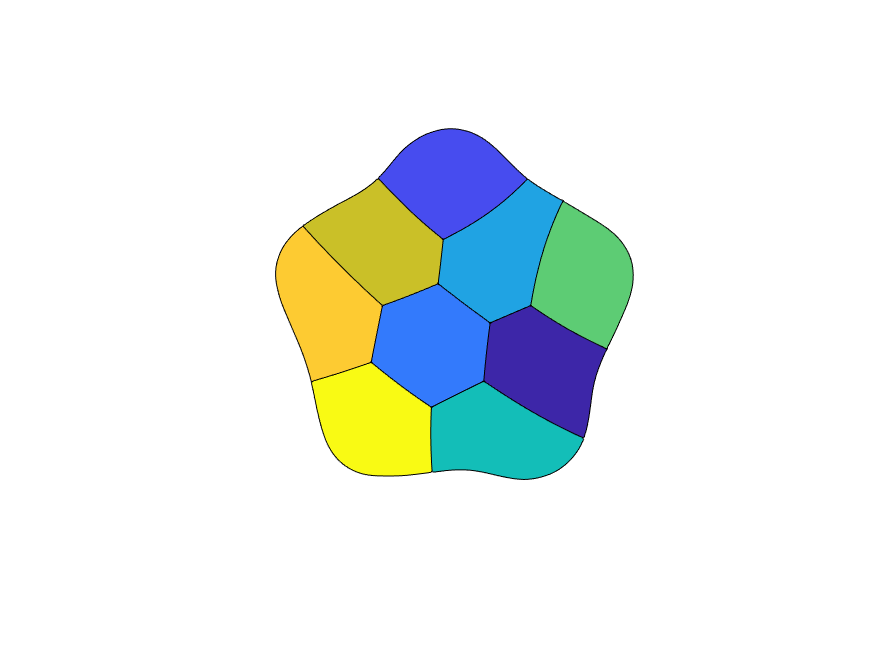}}&
		\raisebox{-.5\height}{\includegraphics[width=0.15\textwidth, clip, trim = 4cm 2.5cm 3cm 1.5cm]{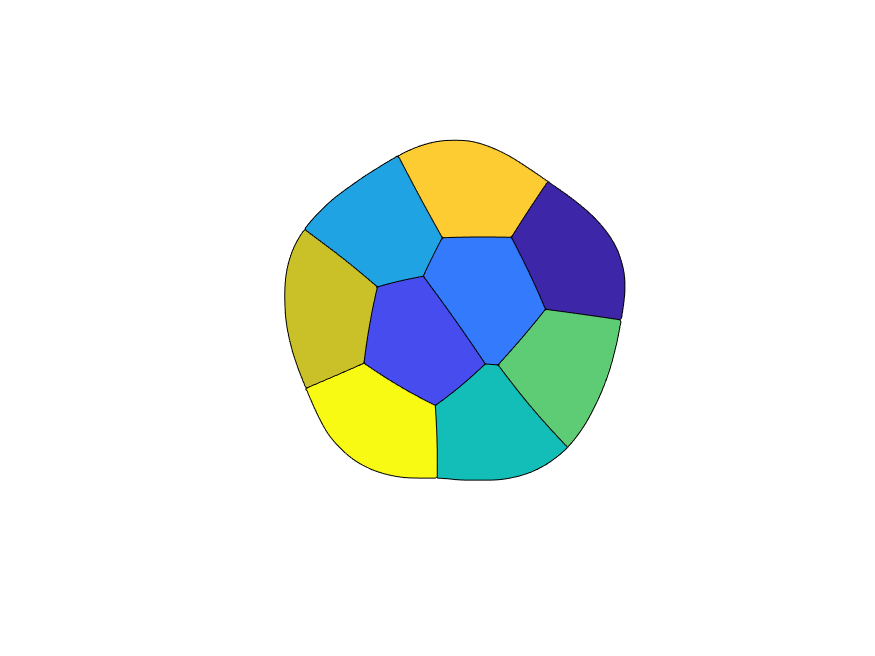}}&
		\raisebox{-.5\height}{\includegraphics[width=0.15\textwidth, clip, trim = 4cm 2.5cm 3cm 1.5cm]{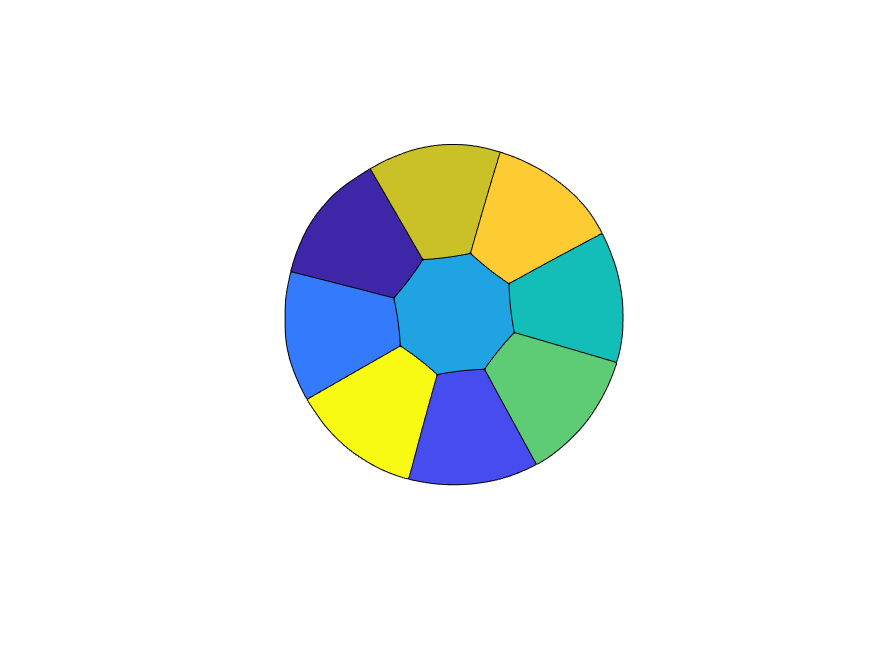}}\\
		\hline
	\end{tabular}
	\caption{Six and nine uniform partitions in two dimensions. The initial condition is set to be a five-petal flower. The total regions and their shortest partitions of both methods in the initial, $5th$, $10th$, and final iterations are plotted. }\label{fig: multi-partition}
\end{figure}

The evolution of the approximate objective functional $\Tilde{E}_{\tau}(u_{\Omega}^{k},(u_i^k)_{i=1}^{n})$ for both problems are plotted in Figure \ref{fig:energy multi-partition}. We observe some irregular protrusions in the objective functional evolution. This is because the auction dynamics may give a local shortest partition instead of a global one in some iterations. Then the approximate objective functional is computed to be larger than the real objective functional, which leads to a protrusion. Since the first method repeats the auction dynamics several times, it is more likely to get a global shortest partition than the second method. Thus the first method has less protrusions on the objective functional evolution. This implies that the first method is more stable than the second one. However, since it repeats auction dynamics many times, the first method requres more time than the second one in general. 
We show the CPU times of both methods with different numbers of partitions in Table \ref{tab:cpu time 2d}. As mentioned above, even though the first method requires fewer iterations, since it repeats auction dynamics several times, its overall CPU time is higher than that of the second method.

\begin{figure}[t!]
    \centering
    \subfigure[First method with six partitions]{\includegraphics[width=0.49\textwidth, clip, trim = 1cm 0.8cm 0.5cm 0.5cm]{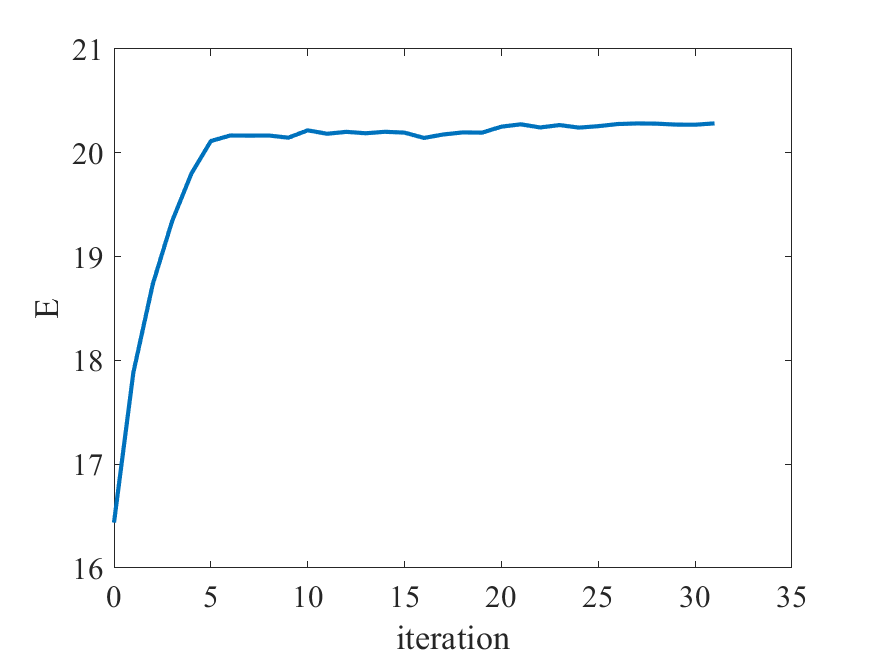}}
    \subfigure[First method with nine partitions]{\includegraphics[width=0.49\textwidth, clip, trim = 1cm 0.8cm 0.5cm 0.5cm]{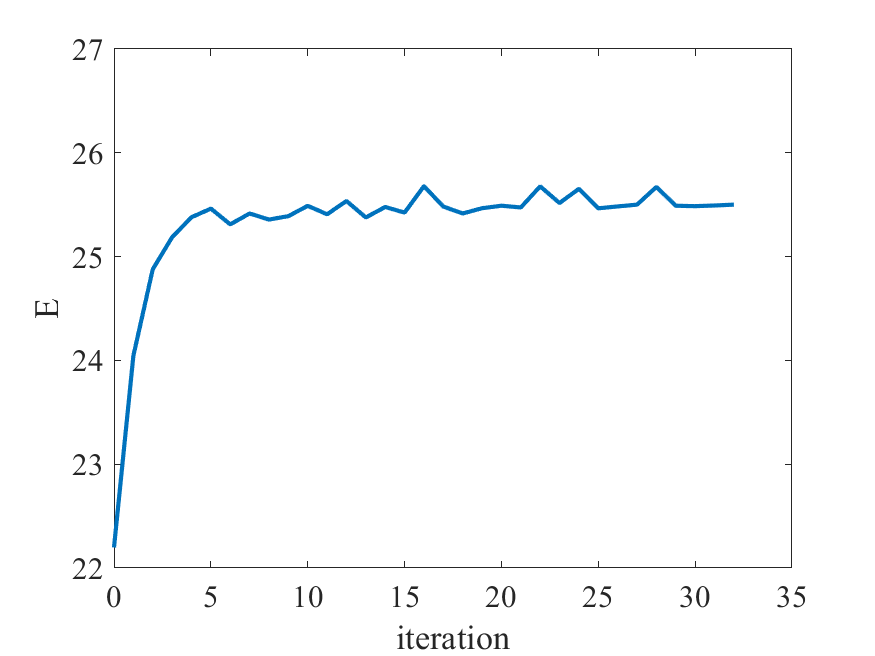}}
    \subfigure[Second method with six partitions]{\includegraphics[width=0.49\textwidth, clip, trim = 1cm 0.8cm 0.5cm 0.5cm]{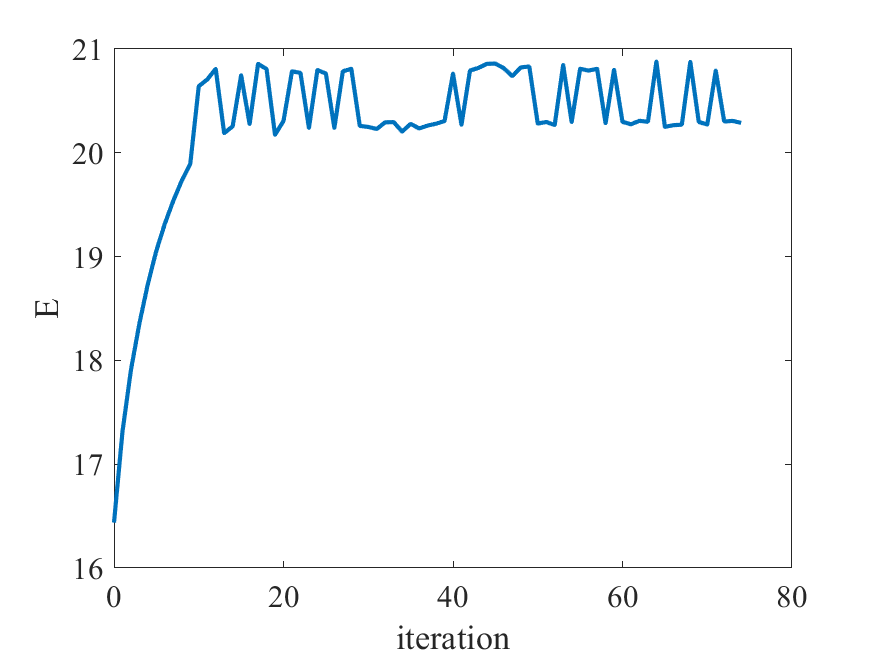}}
    \subfigure[Second method with nine partitions]{\includegraphics[width=0.49\textwidth, clip, trim = 1cm 0.8cm 0.5cm 0.5cm]{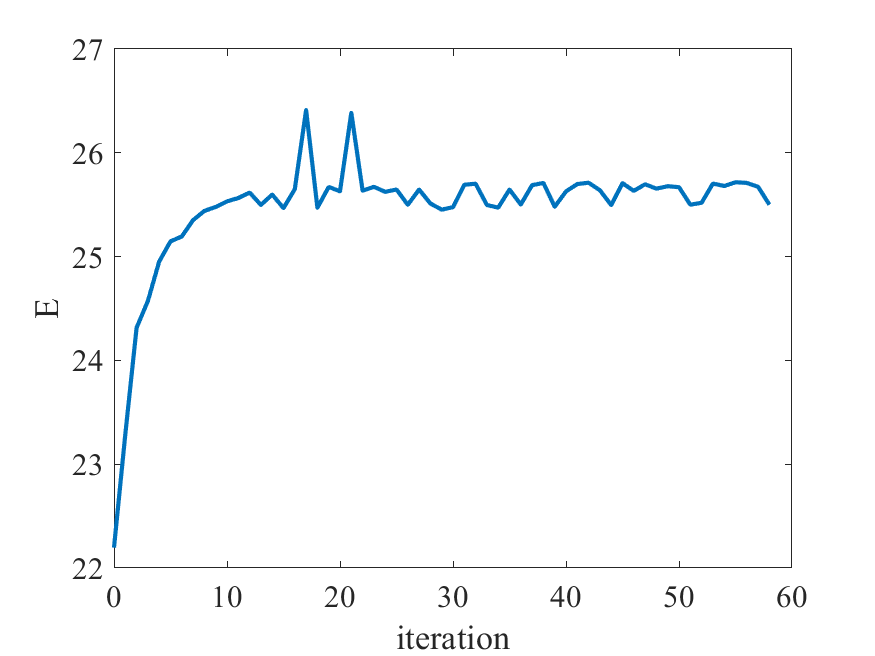}}
    \caption{Six and nine uniform partitions in two dimensions. The approximate objective functionals $\Tilde{E}_{\tau}(u_{\Omega}^{k},(u_i^k)_{i=1}^{n})$ of two methods versus iteration times. The horizontal axis represents the number of iterations. The vertical axis represents the functional value.}
    \label{fig:energy multi-partition}
\end{figure}

\begin{table}[t!]
    \centering
    \begin{tabular}{|c|c|c|c|c|}
    \hline
    \multirow{2}{*}{No. of partitions}& \multicolumn{2}{c|}{No. of Iter.} & \multicolumn{2}{c|}{CPU time(s)}\\
    \cline{2-5}
    & First method & Second method& First method & Second method\\
    \hline
      2 &34&63&  281.49 & 112.72 \\
      6 &31&74& 995.67 & 427.05 \\
      9 &32&59& 1611.79 & 650.43 \\
      \hline
    \end{tabular}
    \caption{Uniform partitions in two dimensions. Number of iterations and CPU times of both methods with different numbers of partition.} 
    \label{tab:cpu time 2d}
\end{table}

%In all cases, the second method is faster than the second one, which is consistent with the analysis of stability and efficiency above.

\subsection{Numerical experiments in three dimensions}
To simulate the problem in three-dimensional cases, an $128\times128\times128$ regular grid is formed in $[-\pi,\pi]^3$. The tolerance of objective functional rate is set to be $r_{tol}=0.0005$. For the first method, auction dynamics repeats $p=3$ times in each iteration. Other parameters are the same as in two-dimensional cases. The initial total region is set to be a cube. We test both methods for three-partition, six-partition and nine-partition cases. The volume proportions of every part in the partitions are set to be equal in each case.
 The results of both methods are shown in Figure \ref{fig: multi-partition 3d}, together with an expanded view and a translucent expanded view.
 
\begin{figure}[t!]
	\centering
	\begin{tabular}{|c|c|c|c|c|c|c|}
		\hline 
		 & \multicolumn{3}{|c|}{First method} & \multicolumn{3}{|c|}{Second method} \\
        \hline
        \makecell{Partition \\number} & 3 & 6 & 9 & 3 & 6 & 9\\
        \hline
        \makecell{Final \\iteration} &
		\raisebox{-.5\height}{\includegraphics[width=0.12\textwidth, clip, trim = 6cm 4cm 4cm 3.5cm]{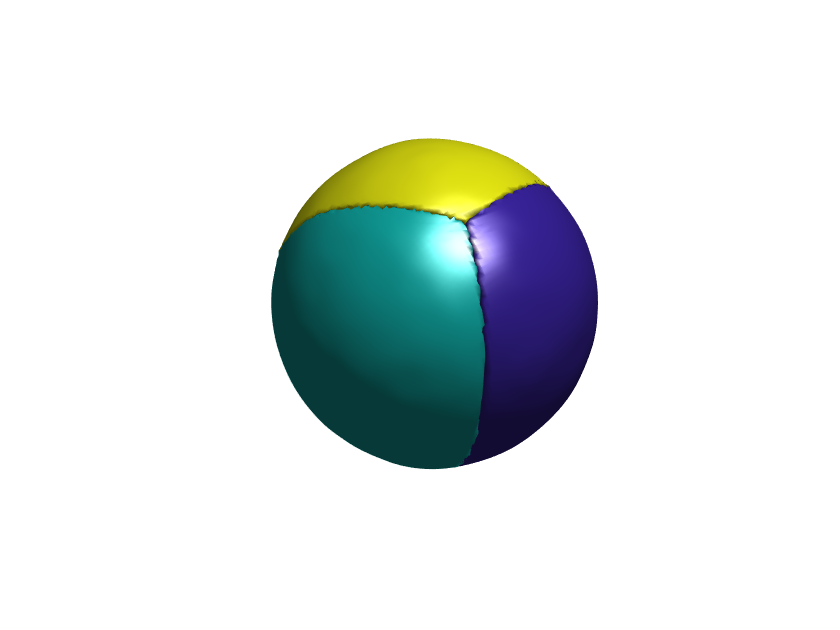}}&  
		\raisebox{-.5\height}{\includegraphics[width=0.12\textwidth, clip, trim = 6cm 4cm 4cm 3.5cm]{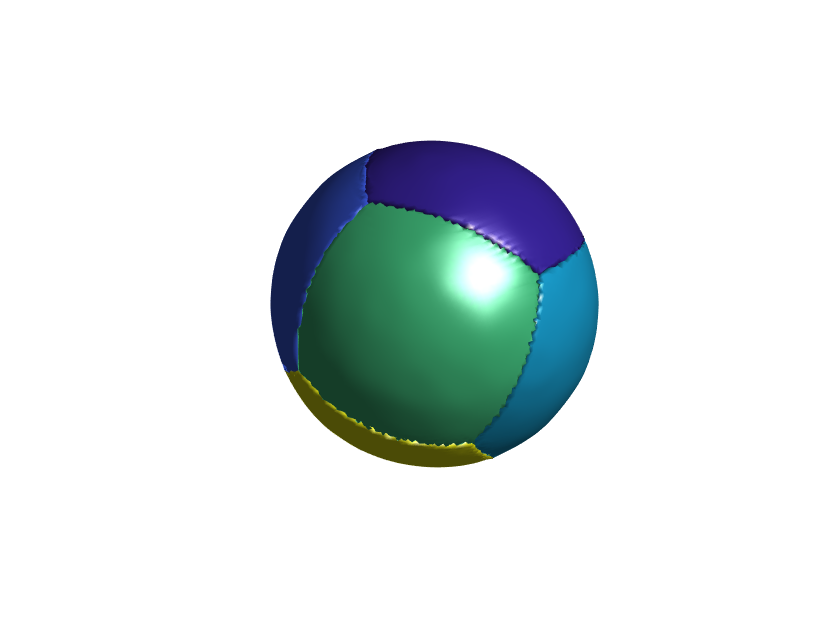}}&
		\raisebox{-.5\height}{\includegraphics[width=0.12\textwidth, clip, trim = 6cm 4cm 4cm 3.5cm]{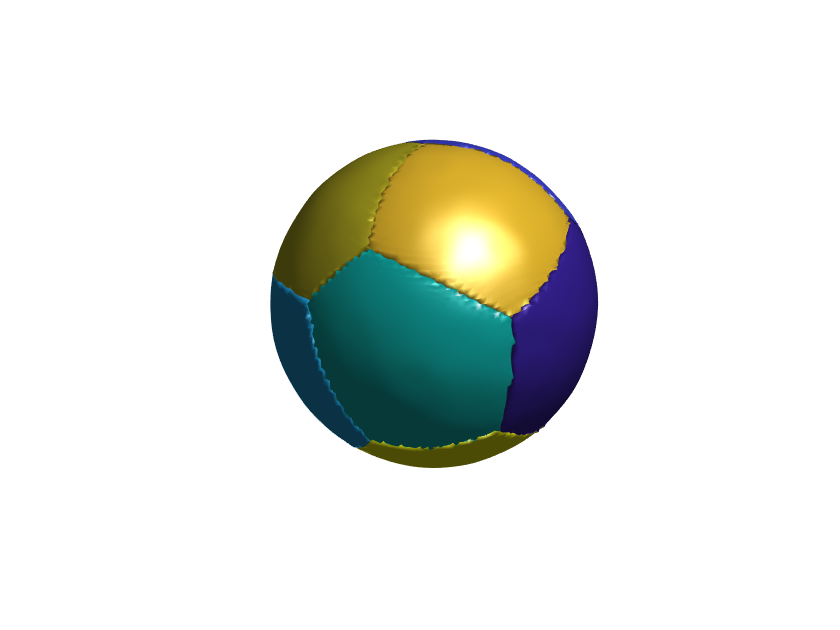}}&
		\raisebox{-.5\height}{\includegraphics[width=0.12\textwidth, clip, trim = 6cm 4cm 4cm 3.5cm]{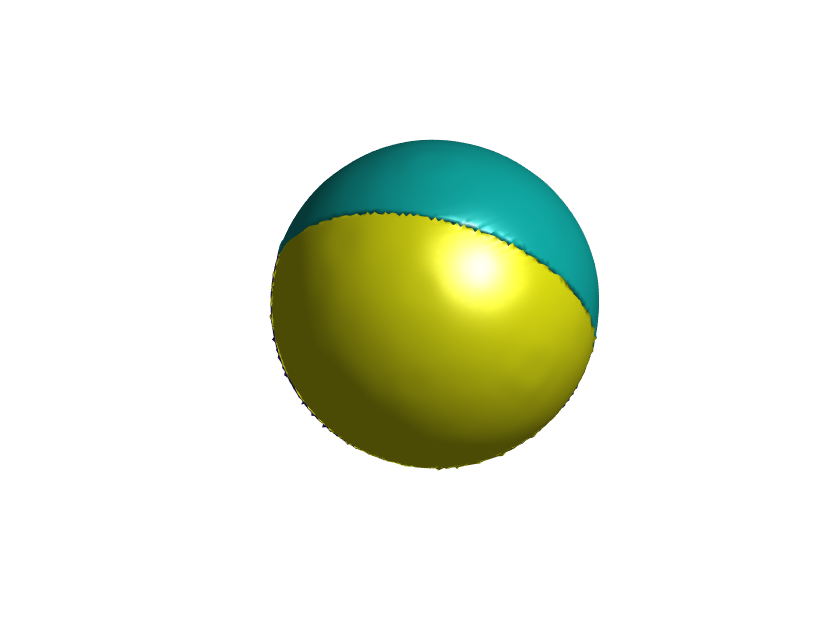}}&
		\raisebox{-.5\height}{\includegraphics[width=0.12\textwidth, clip, trim = 6cm 4cm 4cm 3.5cm]{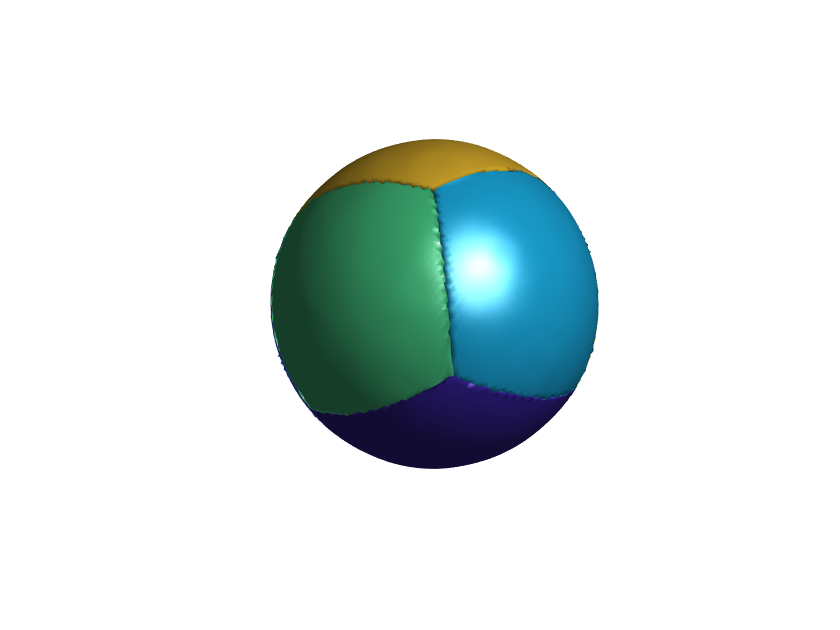}}&
        \raisebox{-.5\height}{\includegraphics[width=0.12\textwidth, clip, trim = 6cm 4cm 4cm 3.5cm]{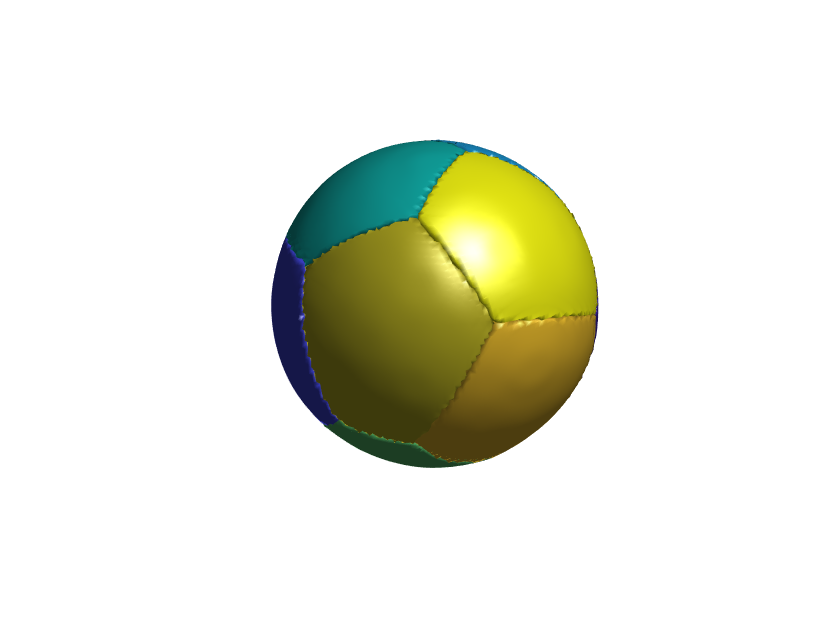}}
         \\ 
        \hline
        \makecell{An \\ expanded \\view} &
		\raisebox{-.5\height}{\includegraphics[width=0.12\textwidth, clip, trim = 4cm 2.5cm 3cm 2cm]{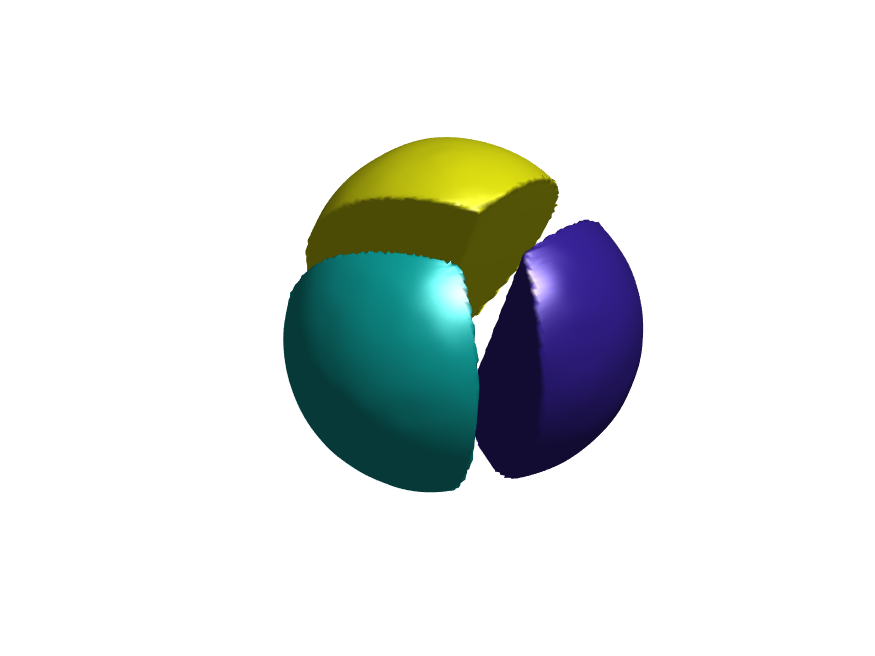}}&  
		\raisebox{-.5\height}{\includegraphics[width=0.12\textwidth, clip, trim = 4cm 2.5cm 3cm 2cm]{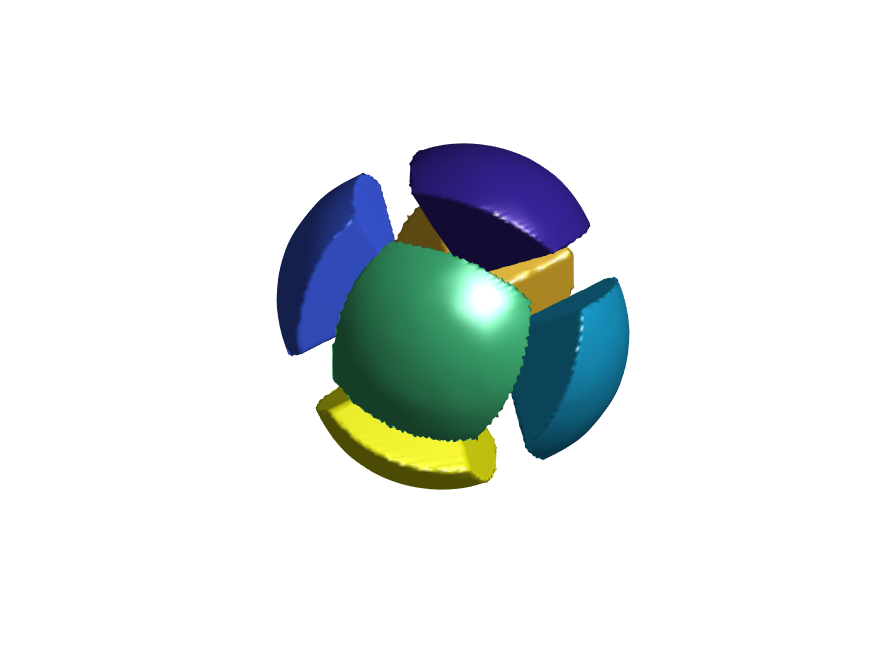}}&
		\raisebox{-.5\height}{\includegraphics[width=0.12\textwidth, clip, trim = 4cm 2.5cm 3cm 2cm]{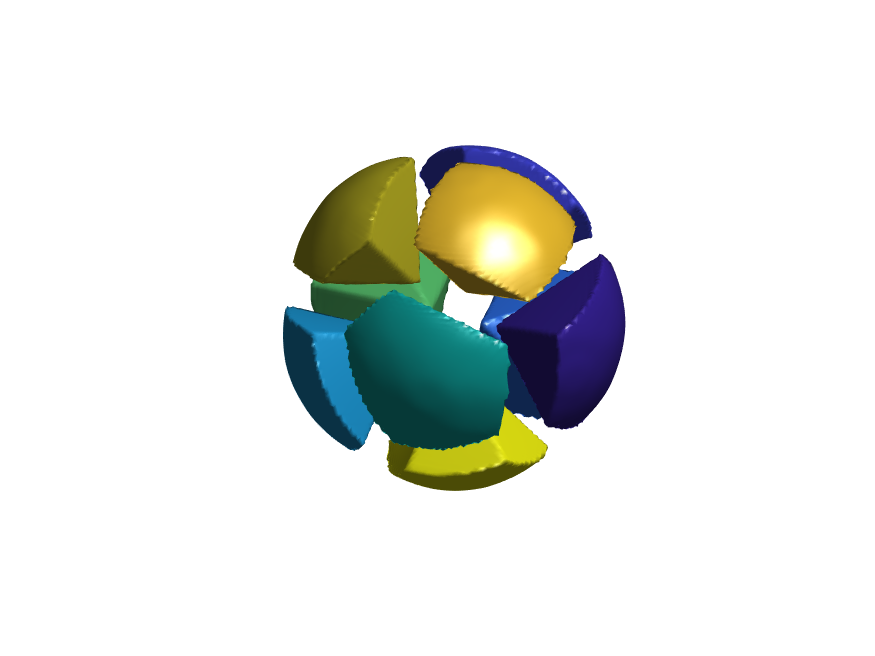}}&
		\raisebox{-.5\height}{\includegraphics[width=0.12\textwidth, clip, trim = 4cm 2.5cm 3cm 2cm]{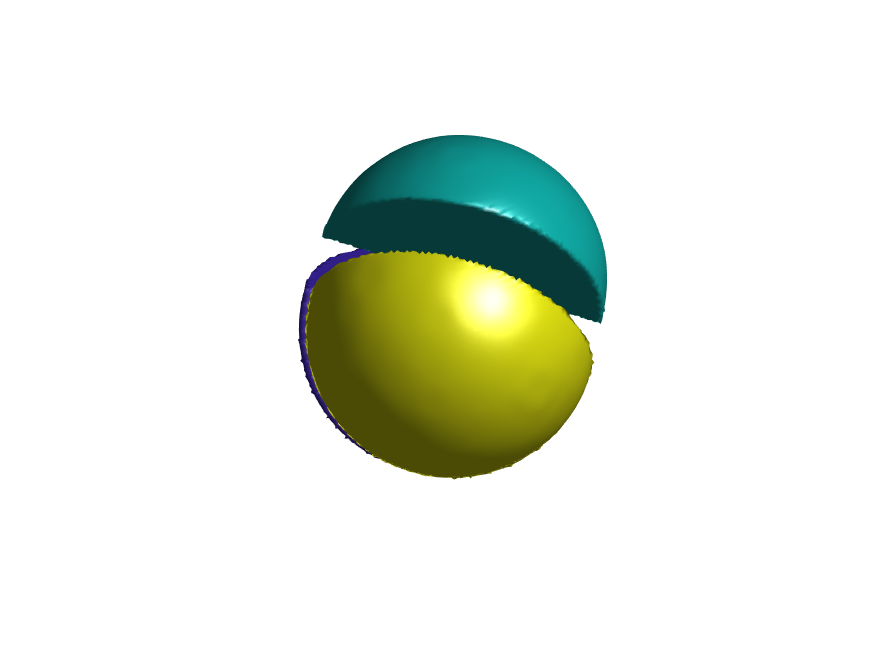}}&
		\raisebox{-.5\height}{\includegraphics[width=0.12\textwidth, clip, trim = 4cm 2.5cm 3cm 2cm]{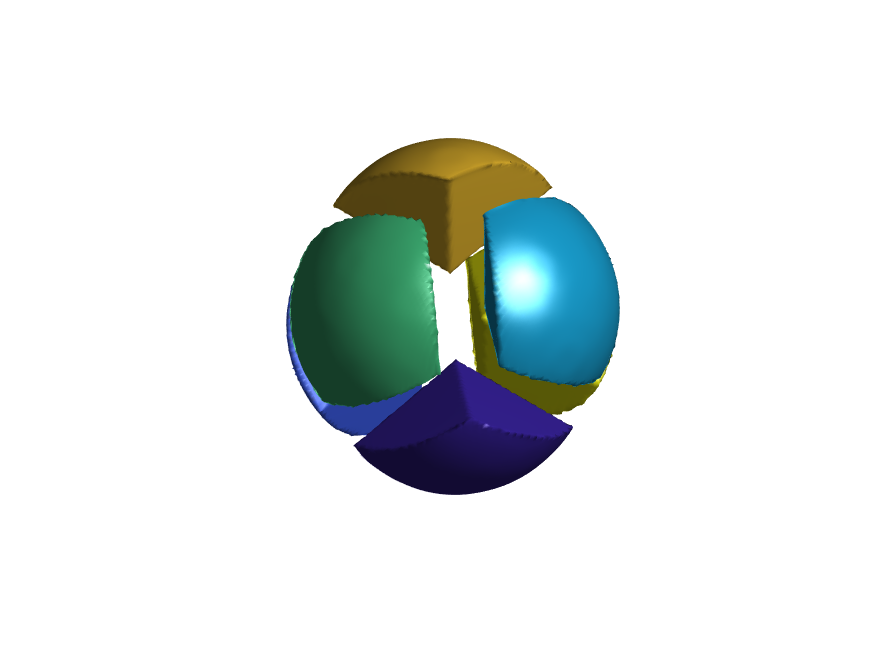}}&
        \raisebox{-.5\height}{\includegraphics[width=0.12\textwidth, clip, trim = 4cm 2.5cm 3cm 2cm]{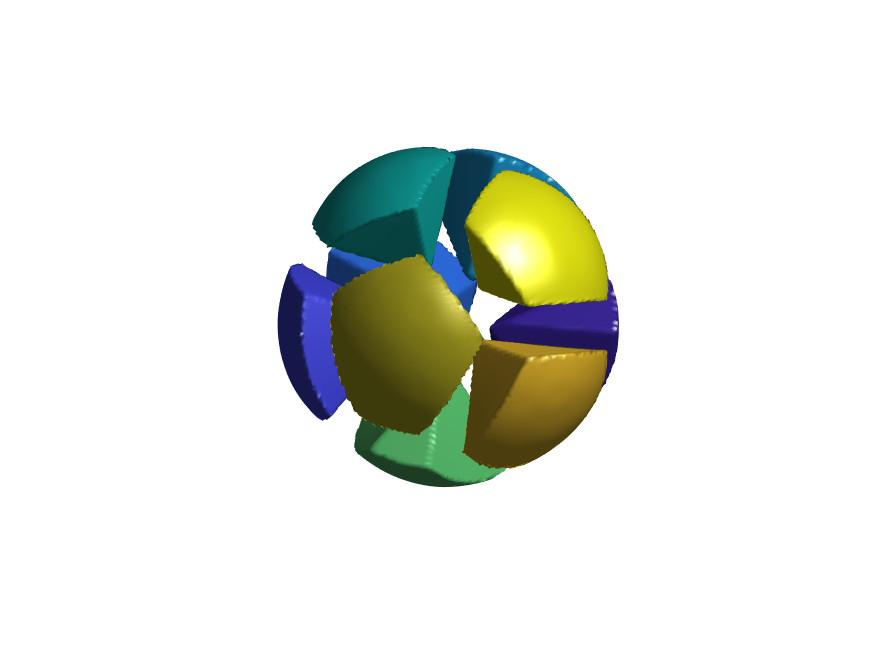}}
         \\ 
        \hline
        \makecell{A \\ translucent \\ expanded \\ view} &
		\raisebox{-.5\height}{\includegraphics[width=0.12\textwidth, clip, trim = 4cm 2.5cm 3cm 2cm]{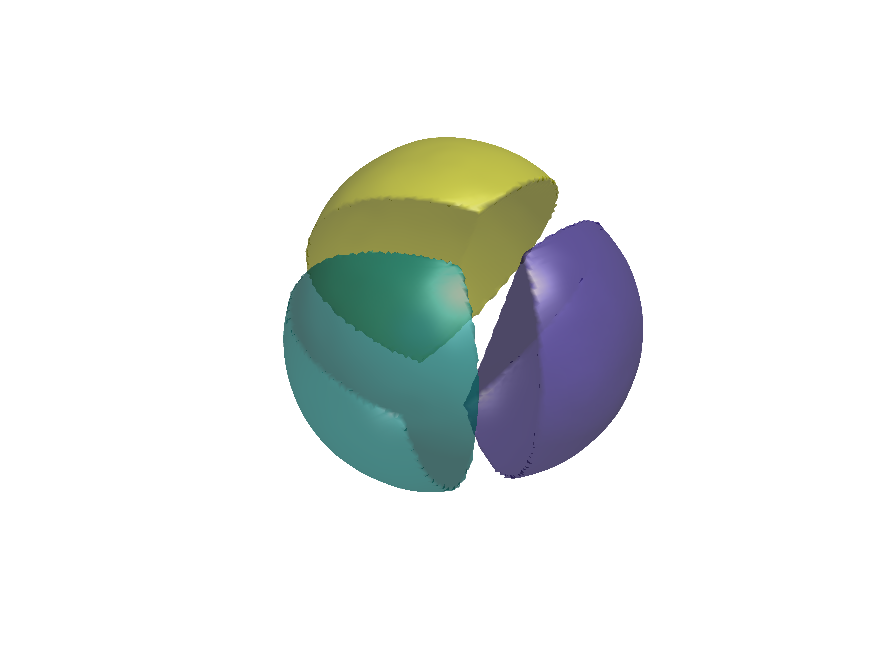}}&  
		\raisebox{-.5\height}{\includegraphics[width=0.12\textwidth, clip, trim = 4cm 2.5cm 3cm 2cm]{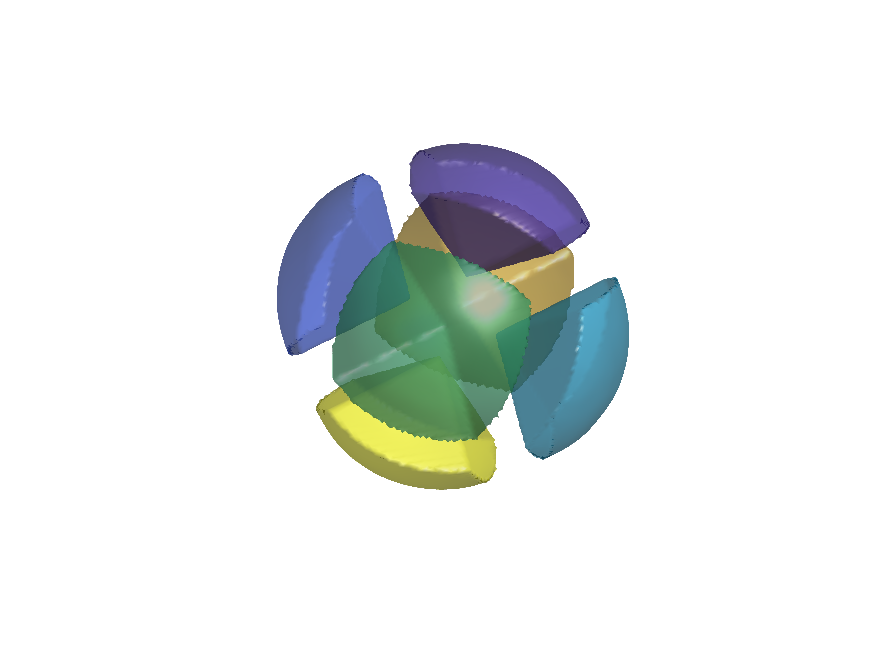}}&
		\raisebox{-.5\height}{\includegraphics[width=0.12\textwidth, clip, trim = 4cm 2.5cm 3cm 2cm]{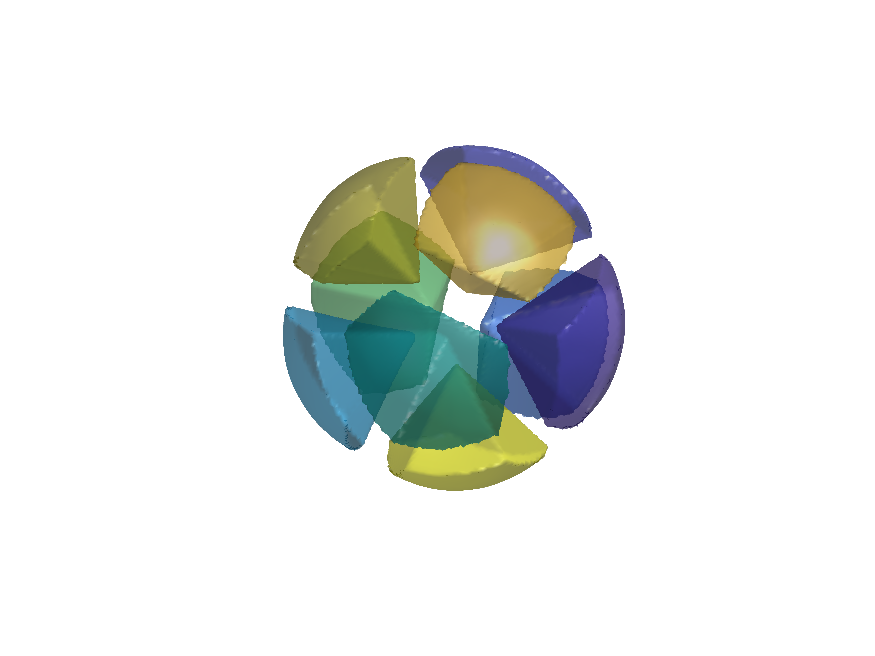}}&
		\raisebox{-.5\height}{\includegraphics[width=0.12\textwidth, clip, trim = 4cm 2.5cm 3cm 2cm]{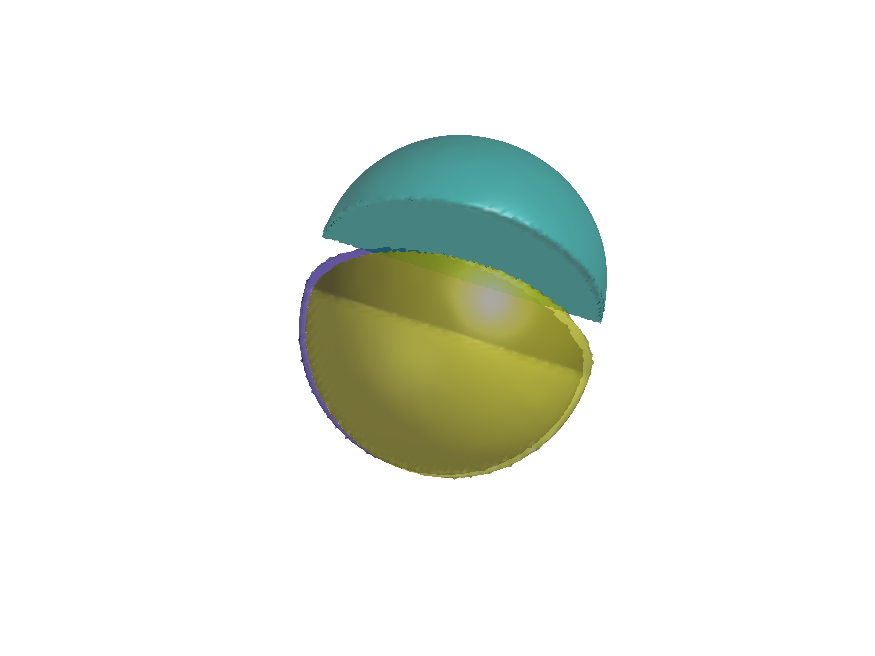}}&
		\raisebox{-.5\height}{\includegraphics[width=0.12\textwidth, clip, trim = 4cm 2.5cm 3cm 2cm]{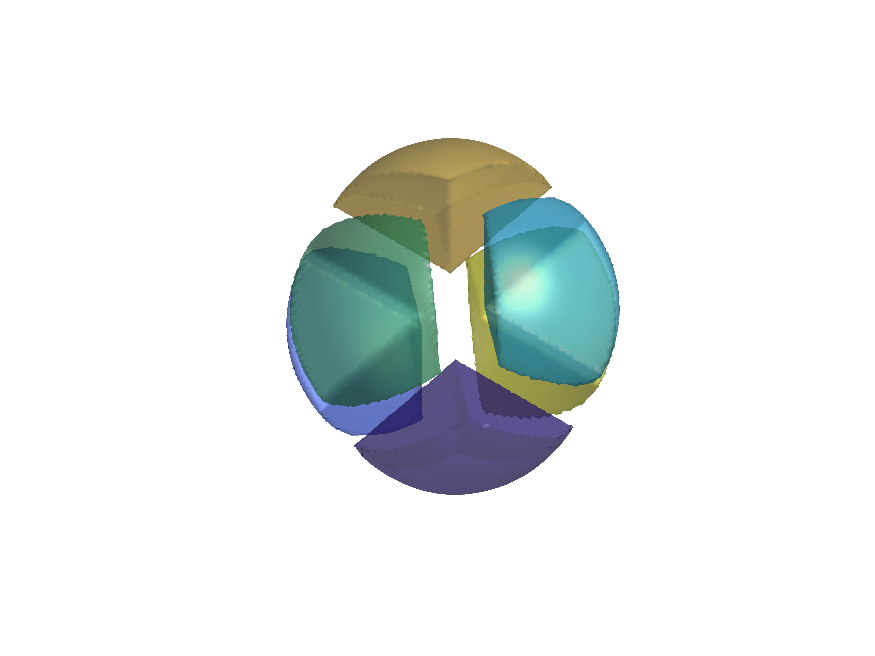}}&
        \raisebox{-.5\height}{\includegraphics[width=0.12\textwidth, clip, trim = 4cm 2.5cm 3cm 2cm]{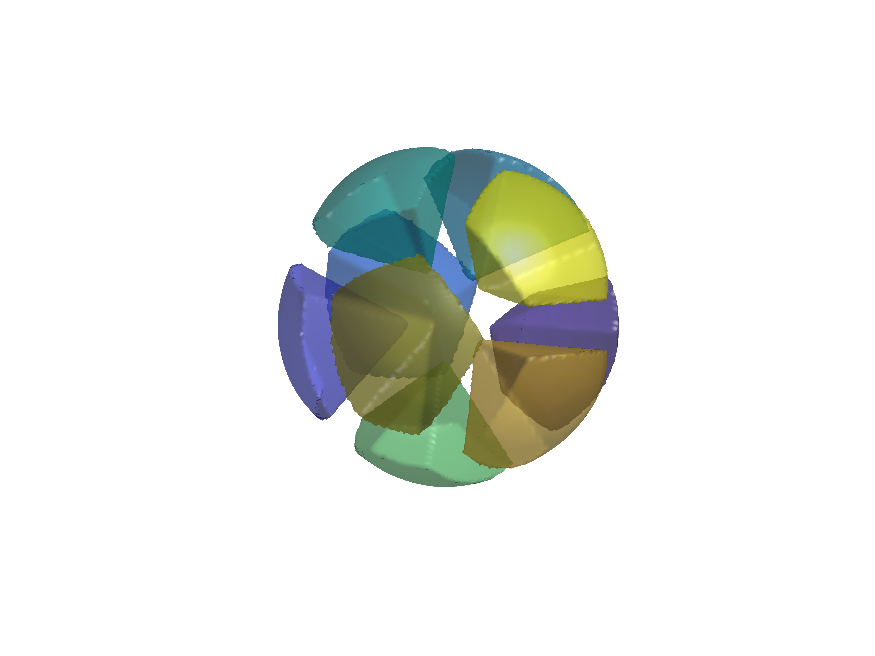}}
         \\ 
        \hline
	\end{tabular}
	\caption{Uniform partitions in three dimensions. The initial condition is set to be a cube. The total regions and their shortest partitions of both methods in the final iteration (together with an expanded view and a translucent expanded view) are plotted for three-partition case, six-partition case and nine-partition case.  }\label{fig: multi-partition 3d}
\end{figure}

In Figure \ref{fig: multi-partition 3d}, the solutions of the total region after convergence by the first method is consistent with the solutions of the second method. They both give a ball as the optimal solution for different numbers of partitions. This numerically supports the conjecture that the ball is the maximizer of the longest minimal length partitions problem in three dimensions. 

The evolution of approximate objective functionals for each experiment are shown in Figure \ref{fig:energy multi-partition 3d}. As the number of partitions increases, the approximate objective functional of final total region also increases. In each case, the objective functional by the first method is consistent to the objective functional by the second method. 

% Besides, protrusions in three-dimensional cases are less than those in two-dimensional cases, which implies with random initialization auction dynamics is more likely to give a global shortest partition instead of a local minimizer in three-dimensional cases than in two-dimensional cases for some unknown reasons. \hao{maybe remove this discussion (and perhaps the following remark) on protrusions, since we cannot explain it?}

The CPU time of both methods with different numbers of partitions are given in Table \ref{tab:cpu time 3d}. Again, the second method uses more iterations but less GPU time compared to the first method.

\begin{figure}[t!]
    \centering
    \subfigure[First method with three partitions]{\includegraphics[width=0.32\textwidth, clip, trim = 1.1cm 0.8cm 0.5cm 0.5cm]{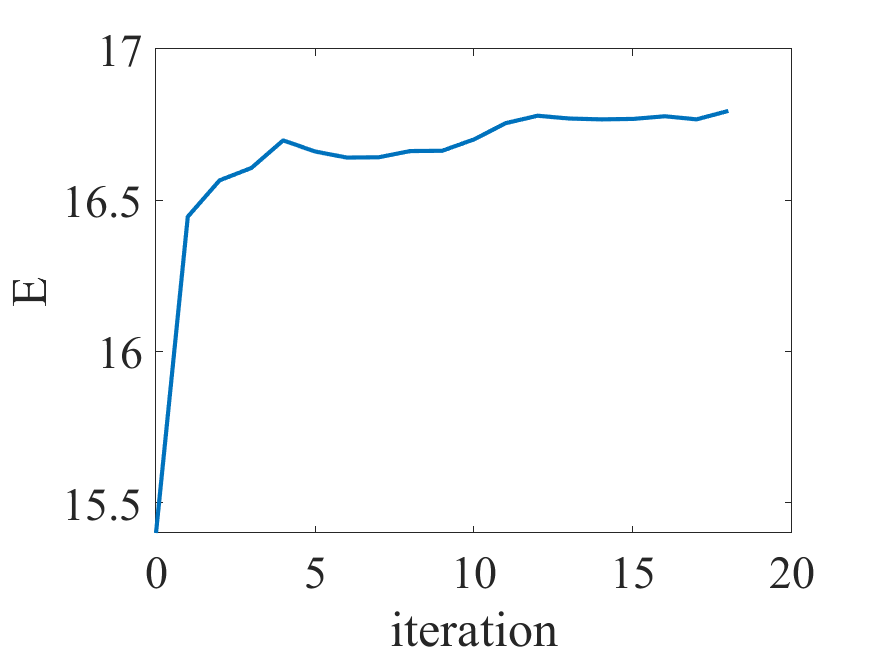}}
    \subfigure[First method with six partitions]{\includegraphics[width=0.32\textwidth, clip, trim = 1.1cm 0.8cm 0.5cm 0.5cm]{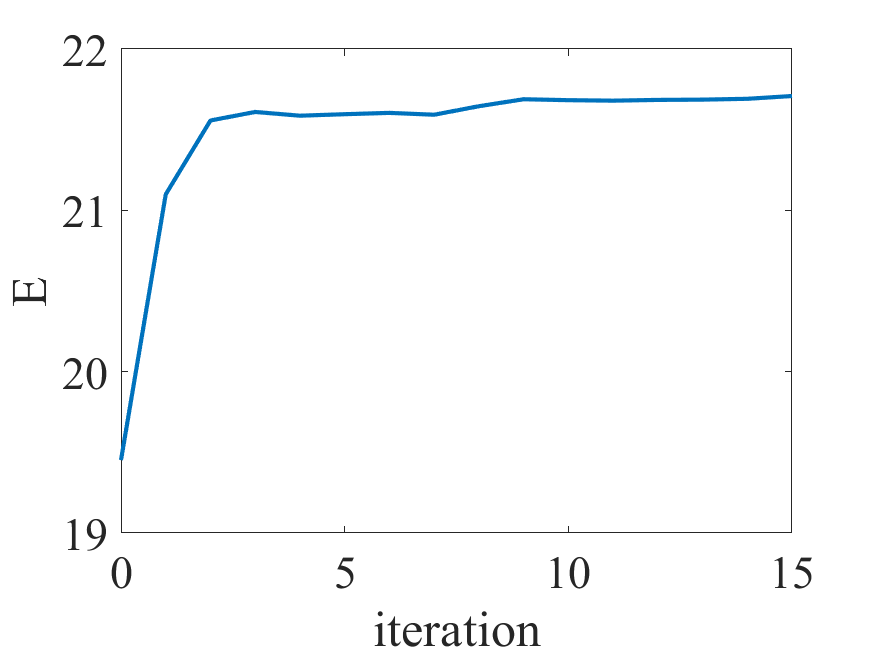}}
    \subfigure[First method with nine partitions]{\includegraphics[width=0.32\textwidth, clip, trim = 1.1cm 0.8cm 0.5cm 0.5cm]{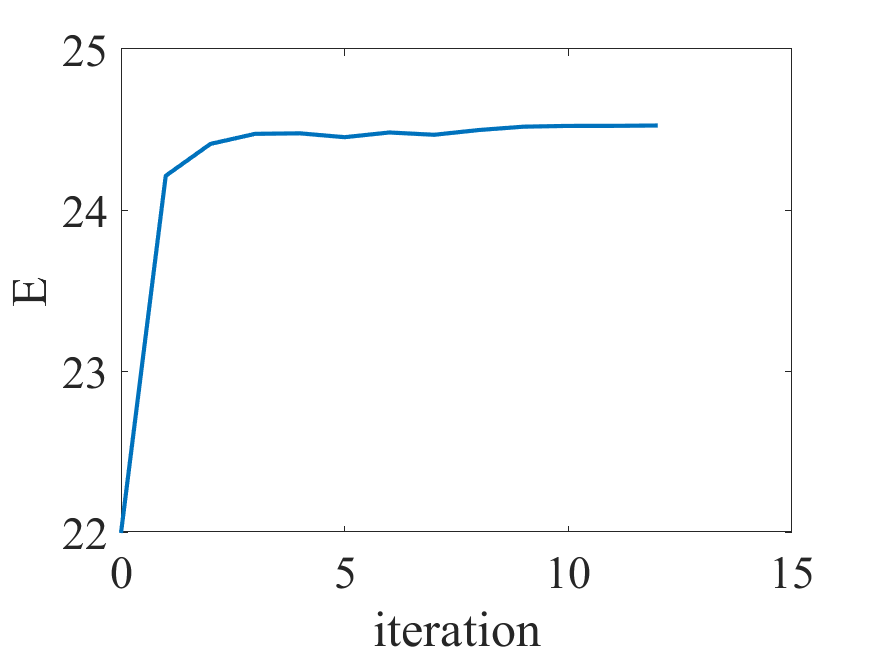}}
    \subfigure[Second method with three partitions]{\includegraphics[width=0.32\textwidth, clip, trim = 1.1cm 0.8cm 0.5cm 0.5cm]{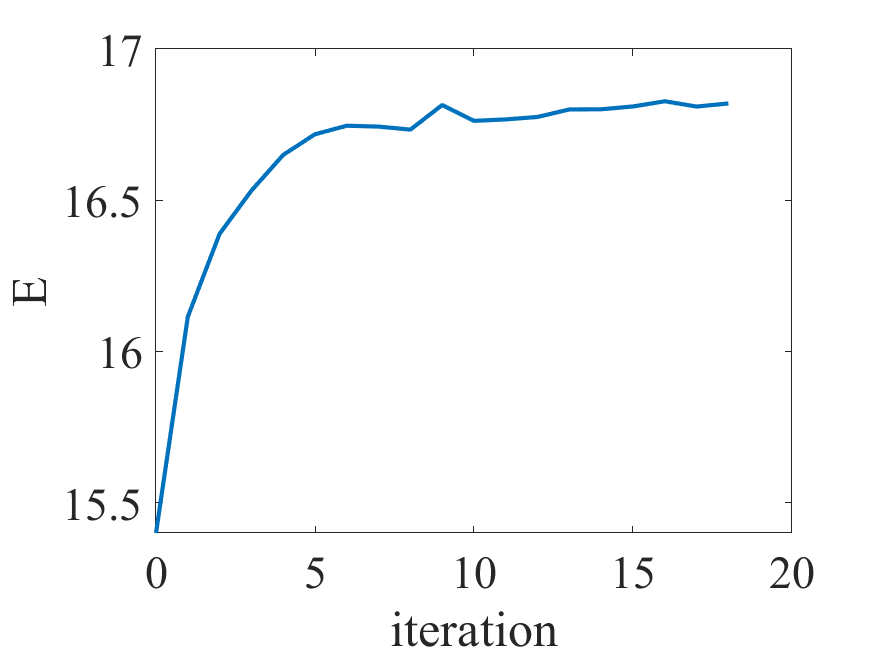}}
    \subfigure[Second method with six partitions]{\includegraphics[width=0.32\textwidth, clip, trim = 1.1cm 0.8cm 0.5cm 0.5cm]{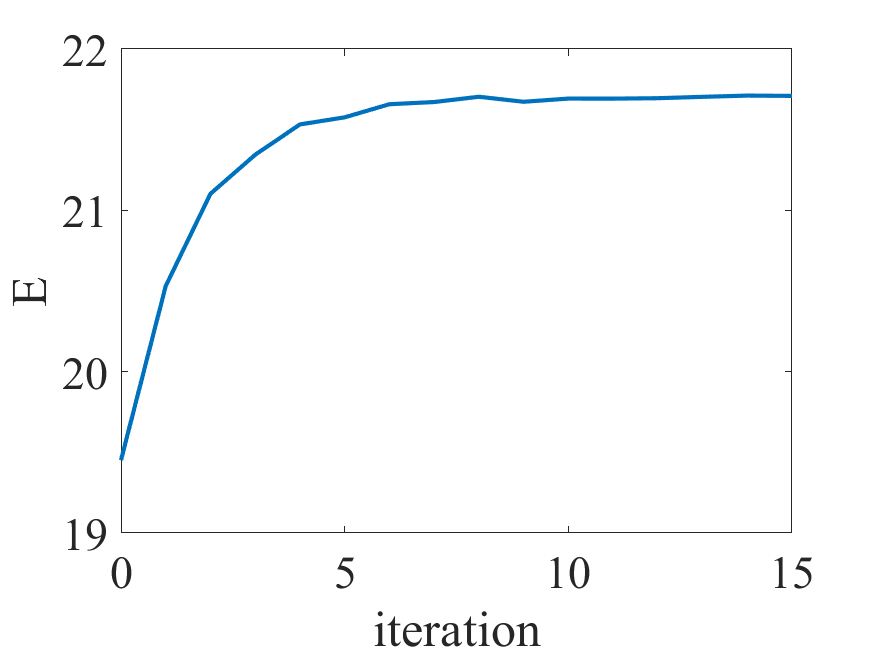}}
    \subfigure[Second method with nine partitions]{\includegraphics[width=0.32\textwidth, clip, trim = 1.1cm 0.8cm 0.5cm 0.5cm]{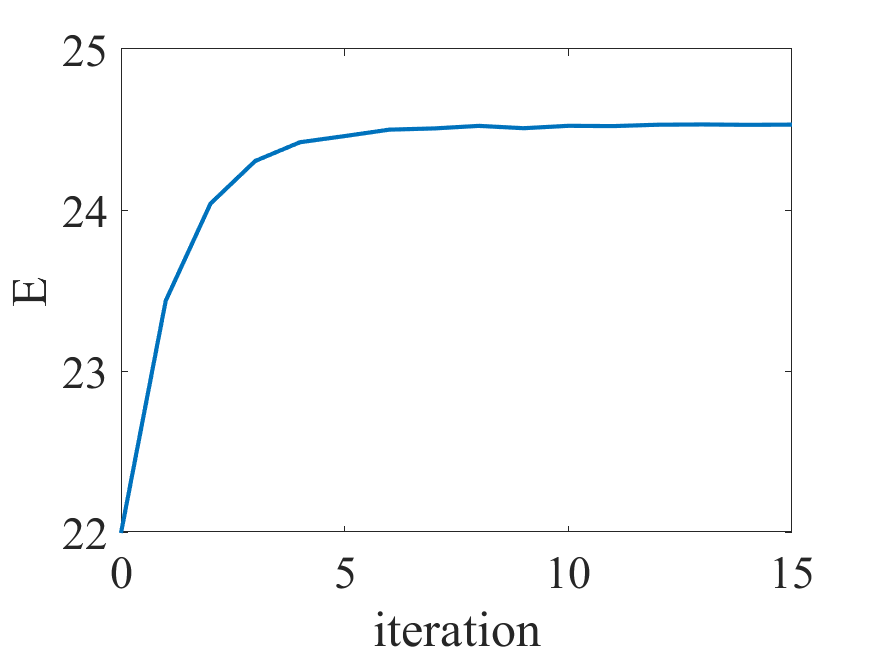}}
    \caption{Uniform partitions in three dimensions. The approximate objective functionals $\Tilde{E}_{\tau}(u_{\Omega}^{k},(u_i^k)_{i=1}^{n})$ of the two methods versus iteration times. The horizontal axis represents the number of iterations. The vertical axis represents the functional value. }
    \label{fig:energy multi-partition 3d}
\end{figure}

% \begin{rem}
%     In Figure \ref{fig:energy multi-partition} and Figure \ref{fig:energy multi-partition 3d}, the energies of both methods in three-dimensional cases seem to be more stable than those in two-dimensional cases. A possible reason for this is that ADM performs better in three dimensions than in two dimensions to find the shortest partitions. This may be interesting to study in the future works.
% \end{rem}

\begin{table}[t!]
    \centering
    \begin{tabular}{|c|c|c|c|c|}
    \hline
    \multirow{2}{*}{No. of partitions}& \multicolumn{2}{c|}{No. of Iter.} & \multicolumn{2}{c|}{CPU time(s)}\\
    \cline{2-5}
    & First method & Second method& First method & Second method\\
    \hline
      3 &18&18&  10678.05 & 2768.09 \\
      6 &15&15& 11053.22 & 4260.25 \\
      9 &12&15& 15129.92 & 6282.31 \\
      \hline
    \end{tabular}
    \caption{Uniform partitions in three dimensions. Number of iterations and CPU times of both methods.}
    \label{tab:cpu time 3d}
\end{table}

\section{An objective-functional-monotone method}\label{sec:monotone}

Based on the first method, we further develop a method that promises the approximate objective functional is increasing during iterations. The idea is to check whether $\Tilde{E}_{\tau}(u_{\Omega}^{k+1},(u_i^{k+1})_{i=1}^{n})$ is greater than or equal to $\Tilde{E}_{\tau}(u_{\Omega}^{k},(u_i^k)_{i=1}^{n})$ in each step. If not, then $(u_i^k)_{i=1}^{n}$ is possible to be a local minimizer instead of global shortest partition, or $u_{\Omega}^{k}$ is near the optimal solution. In this case, we just reduce $\beta$ from $1$, get new $u_{\Omega}^{k+1},(u_i^{k+1})_{i=1}^{n}$, and check if the objective functional is increasing. If the objective functional is increasing, then the iterations keep going. Otherwise, $\beta$ is further reduced, and new $u_{\Omega}^{k+1},(u_i^{k+1})_{i=1}^{n}$ are computed to compare the objective functional. If $u_{\Omega}^{k+1},(u_i^{k+1})_{i=1}^{n}$ are not found to increase the objective functional, and $\beta$ is reduced to $0$, then $(u_i^k)_{i=1}^{n}$ is not a global shortest partition, or $u_{\Omega}^{k}$ is likely to be the optimal total region. To check which case it is, auction dynamics is repeated multiple times to find a partition of $u_{\Omega}^{k}$ with less objective functional than $(u_i^k)_{i=1}^{n}$. If such partition is found, then it is set to be the new $(u_i^k)_{i=1}^{n}$, and we take a step back to check whether $\Tilde{E}_{\tau}(u_{\Omega}^{k},(u_i^{k})_{i=1}^{n})$ is greater than or equal to $\Tilde{E}_{\tau}(u_{\Omega}^{k-1},(u_i^{k-1})_{i=1}^{n})$. If such partition is not found, then we accept $u_{\Omega}^{k}$ to be the result, and $(u_i^k)_{i=1}^{n}$ is its shortest partition.

This method is tested for two-dimensional examples with two and three partitions and equal proportions. The total regions and their shortest partitions of both cases in some iterations are shown in Figure \ref{fig: energy monotone result}. The approximate objective functionals are shown in Figure \ref{fig:energy monotone}. For two-partition case, this method uses 12 iterations and 3147.61 seconds. For three-partition case, this method uses 19 iterations and 3896.23 seconds. Although the method promises the objective functional is increasing, it spends ten times more CPU time than the first method for the two-partition case. This is because the new method uses a lot of time (many repetition of auction dynamics) for checking whether $u^k_{\Omega}$ and $(u^k_{i})_{i=1}^n$ are global optimizers.

\begin{figure}[t!]
	\centering
	\begin{tabular}{|c|c|c|c|c|c|c|}
		\hline 
		 Iterations& 2  & 4 & 6  & 8  & 10  & 12 \\
        \hline
        \makecell{Two\\-partition\\case}  &
		\raisebox{-.5\height}{\includegraphics[width=0.1\textwidth, clip, trim = 4cm 2.5cm 3cm 1.5cm]{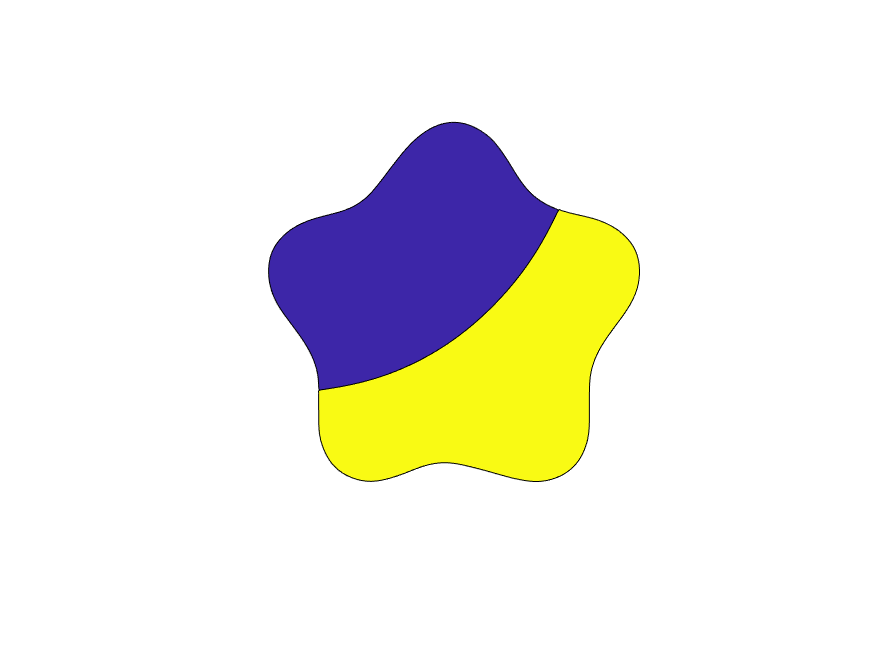}}&  
		\raisebox{-.5\height}{\includegraphics[width=0.1\textwidth, clip, trim = 4cm 2.5cm 3cm 1.5cm]{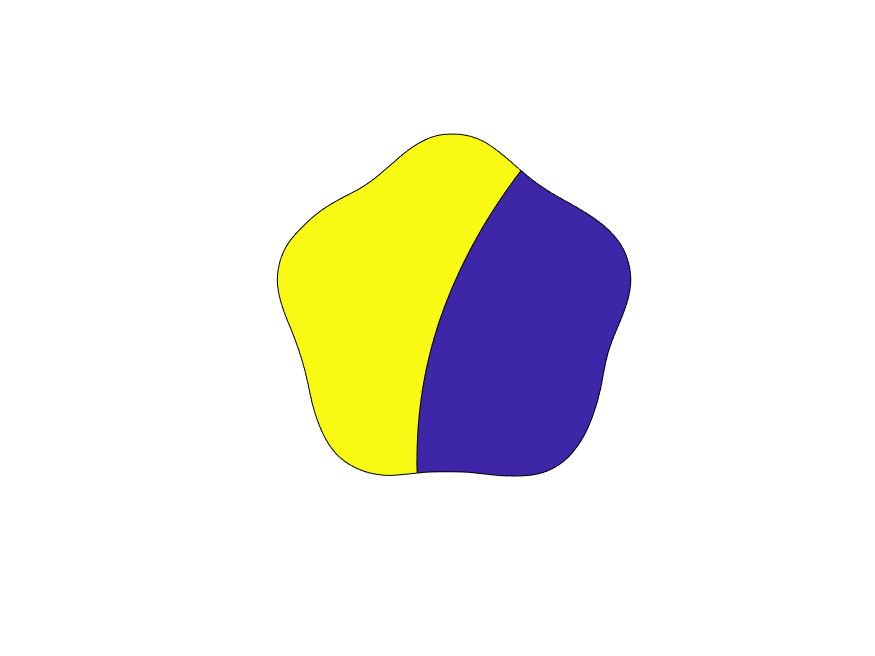}}&
		\raisebox{-.5\height}{\includegraphics[width=0.1\textwidth, clip, trim = 4cm 2.5cm 3cm 1.5cm]{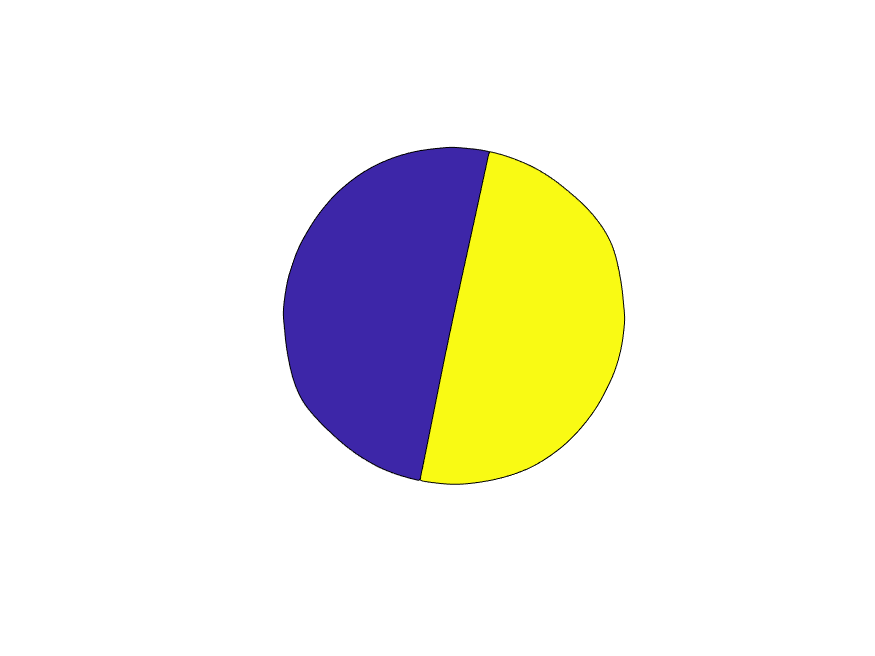}}&
		\raisebox{-.5\height}{\includegraphics[width=0.1\textwidth, clip, trim = 4cm 2.5cm 3cm 1.5cm]{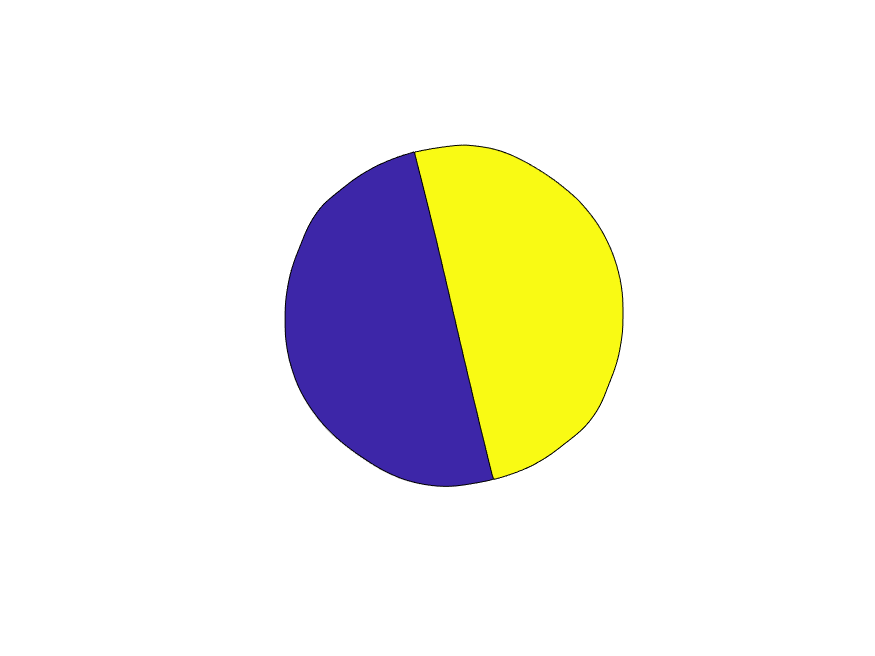}}&
  	\raisebox{-.5\height}{\includegraphics[width=0.1\textwidth, clip, trim = 4cm 2.5cm 3cm 1.5cm]{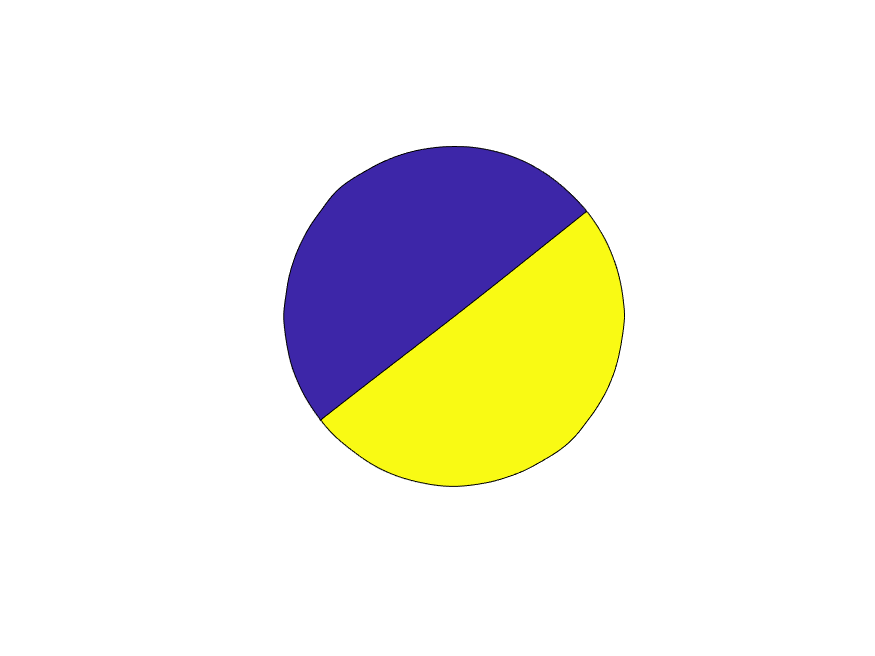}}&
   		\raisebox{-.5\height}{\includegraphics[width=0.1\textwidth, clip, trim = 4cm 2.5cm 3cm 1.5cm]{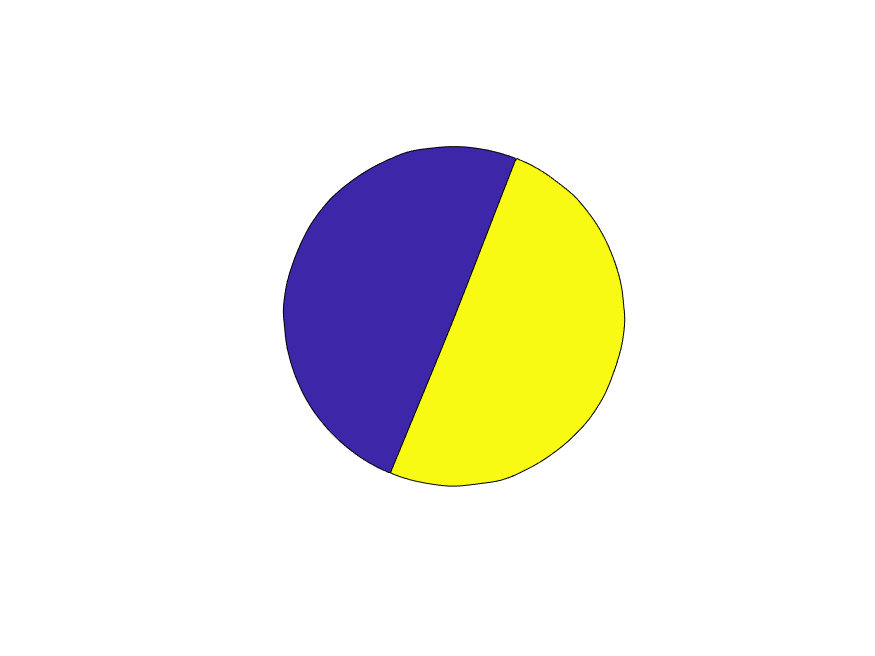}}\\
		\hline 
		 Iterations& 2 & 4 & 8 & 12 & 16 & 19\\
     \hline
        \makecell{Three-\\partition\\case}  &
		\raisebox{-.5\height}{\includegraphics[width=0.1\textwidth, clip, trim = 4cm 2.5cm 3cm 1.5cm]{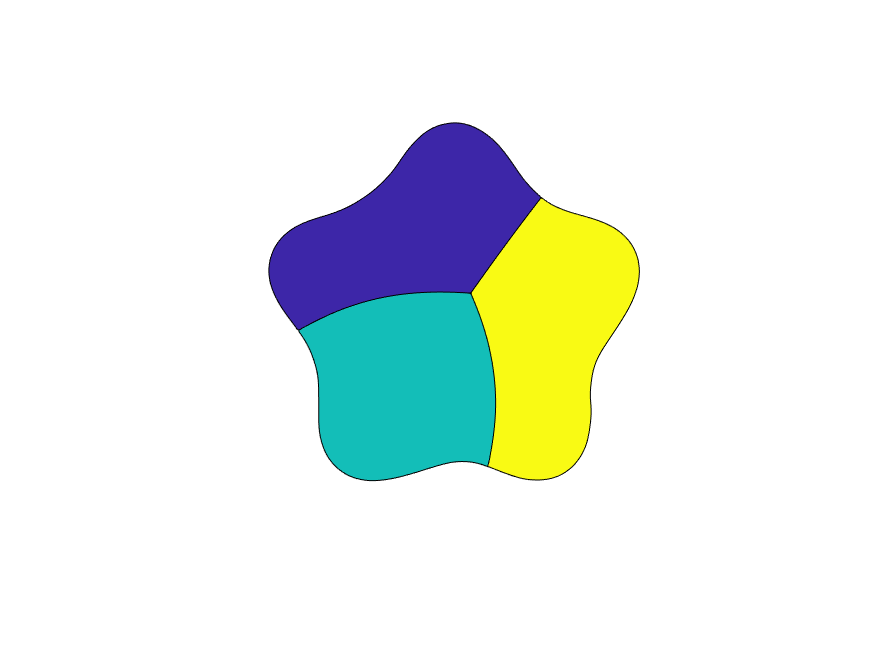}}&  
		\raisebox{-.5\height}{\includegraphics[width=0.1\textwidth, clip, trim = 4cm 2.5cm 3cm 1.5cm]{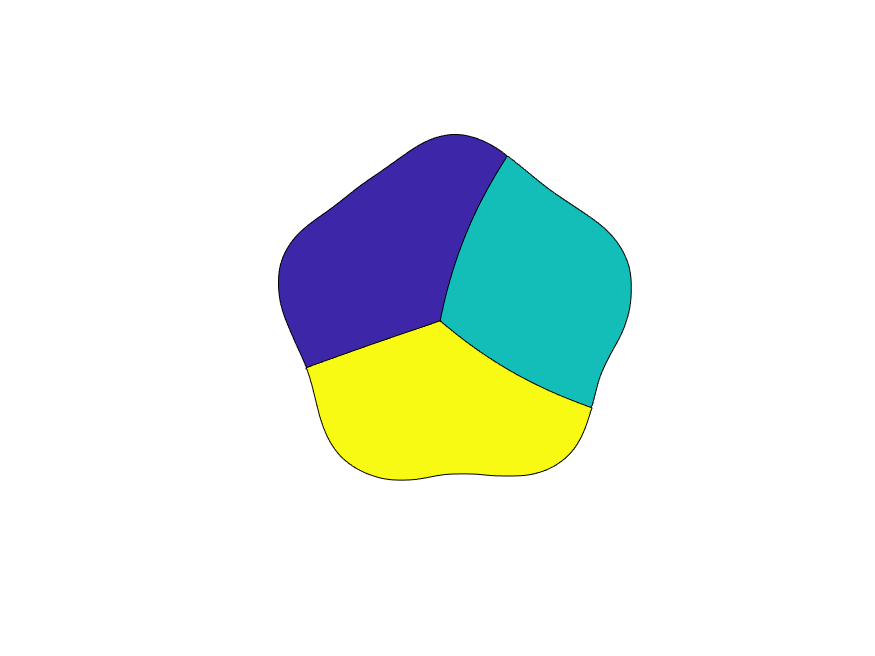}}&
		\raisebox{-.5\height}{\includegraphics[width=0.1\textwidth, clip, trim = 4cm 2.5cm 3cm 1.5cm]{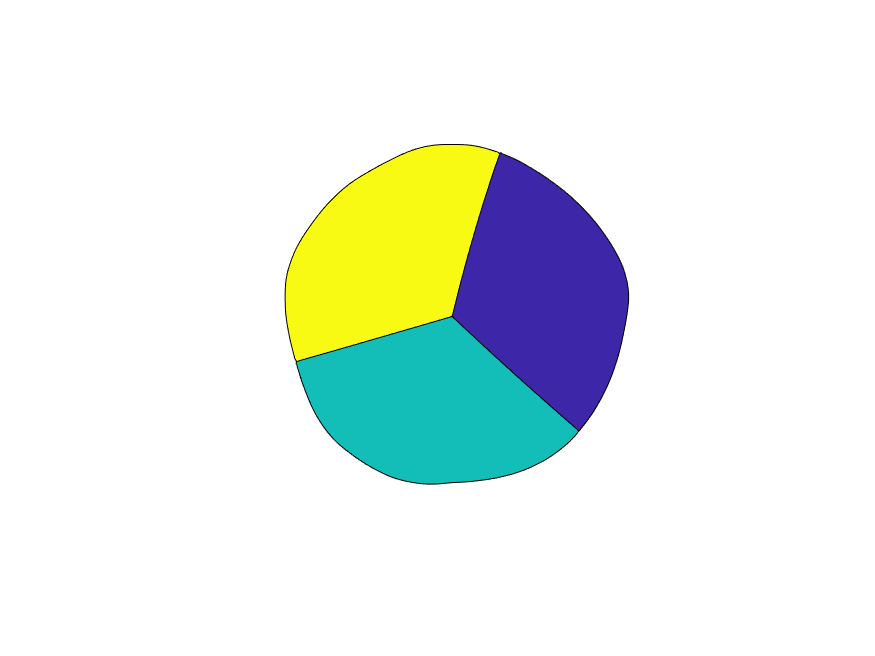}}&
		\raisebox{-.5\height}{\includegraphics[width=0.1\textwidth, clip, trim = 4cm 2.5cm 3cm 1.5cm]{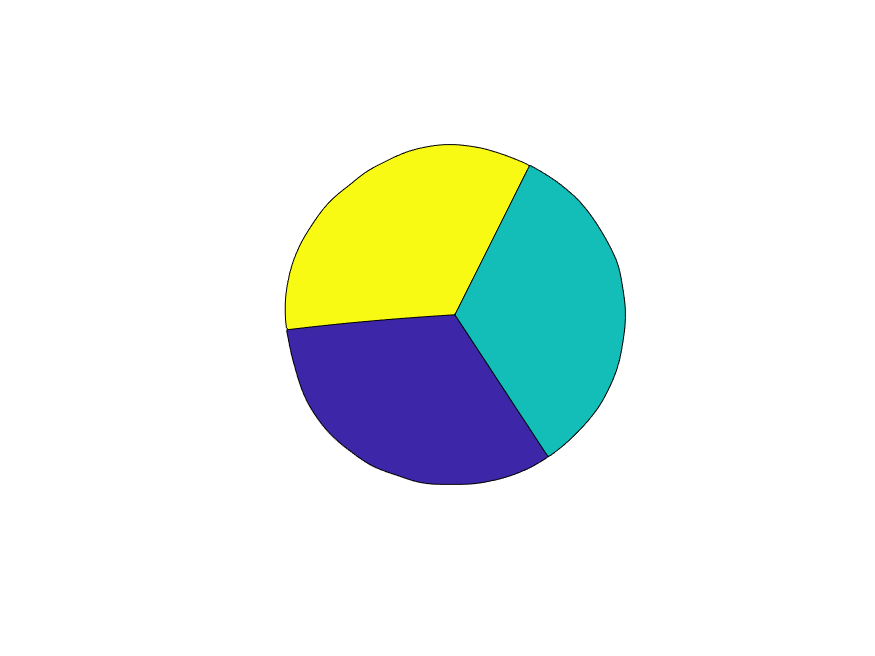}}&
		\raisebox{-.5\height}{\includegraphics[width=0.1\textwidth, clip, trim = 4cm 2.5cm 3cm 1.5cm]{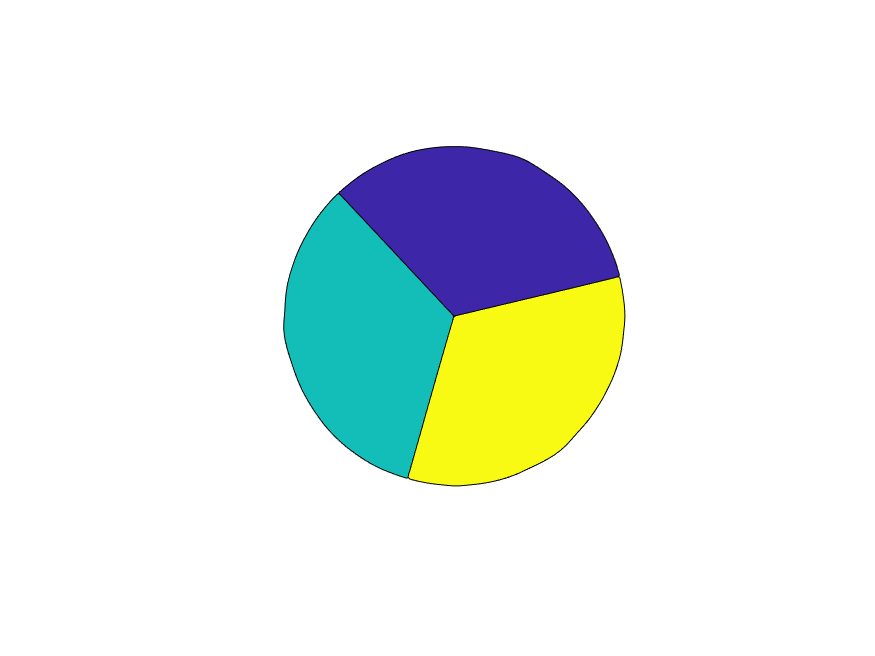}}&
		\raisebox{-.5\height}{\includegraphics[width=0.1\textwidth, clip, trim = 4cm 2.5cm 3cm 1.5cm]{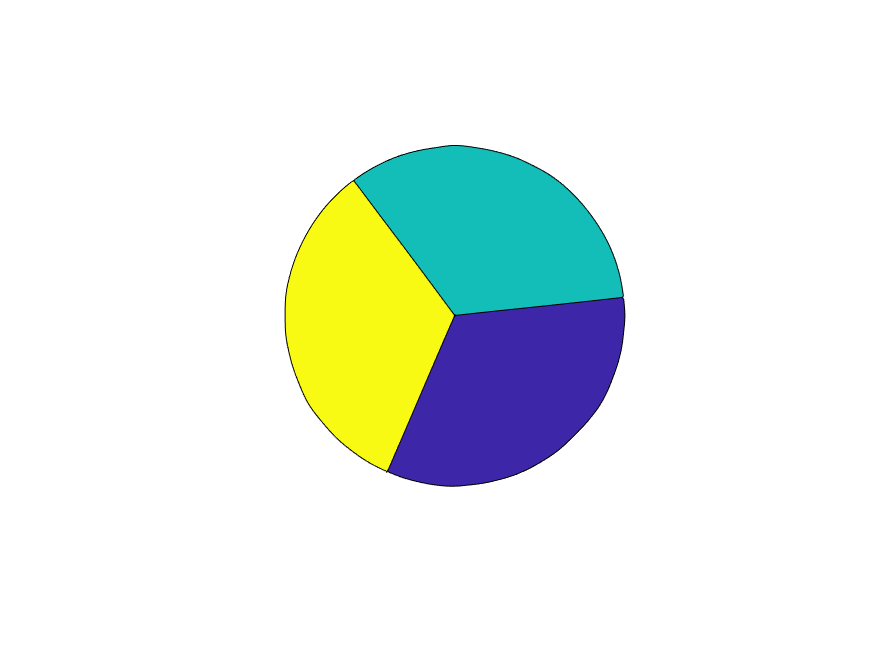}}\\
		\hline
	\end{tabular}
	\caption{Objective-functional-monotone method. The initial condition is set to be a "five-petal flower". The total regions and their shortest partitions of the objective functional monotone method in some iterations are plotted for two-partition and three-partition cases with equal proportions.} \label{fig: energy monotone result}
\end{figure}

\begin{figure}[t!]
    \centering
    \subfigure[Two-partition case]{ \includegraphics[width=0.49\textwidth, clip, trim = 1.1cm 0.5cm 0.5cm 0.5cm]{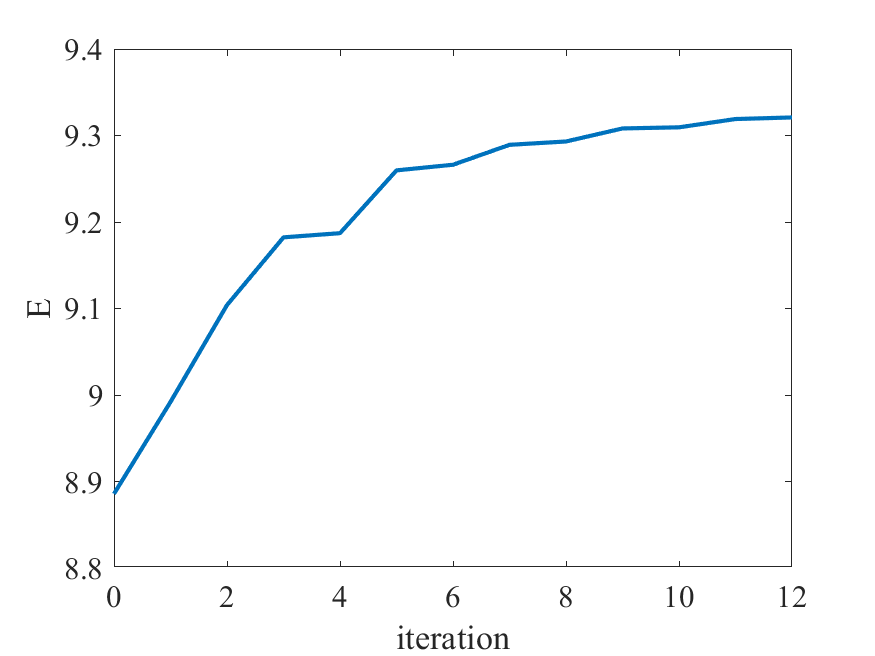}}
    \subfigure[Three-partition case]{\includegraphics[width=0.49\textwidth, clip, trim = 0.8cm 0.5cm 0.5cm 0.5cm]{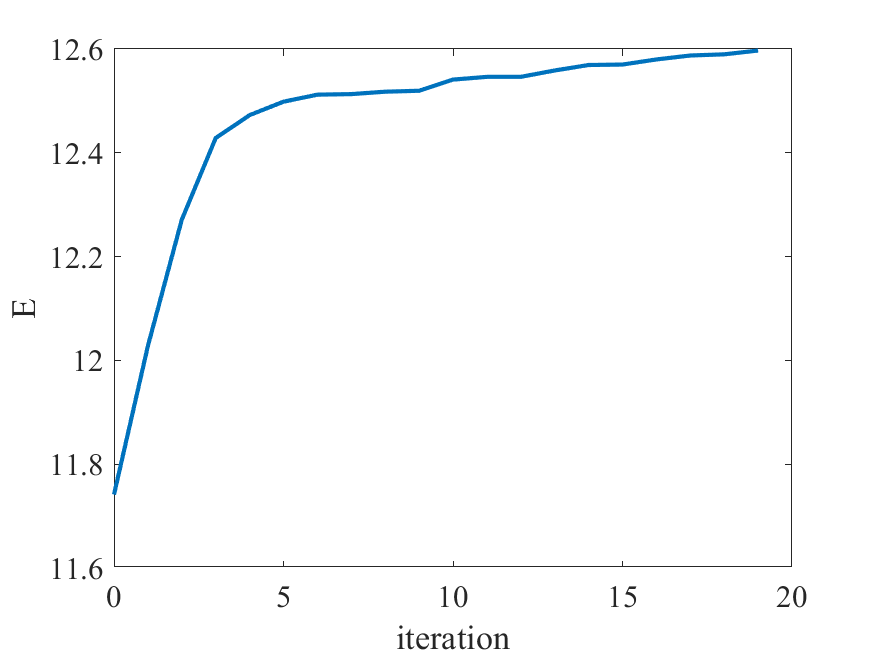}}
    \caption{Objective-functional-monotone method for uniform partitions in two dimensions. The approximate objective functionals $\Tilde{E}_{\tau}(u_{\Omega}^{k},(u_i^k)_{i=1}^{n})$ of both cases versus iteration times are plotted. The horizontal axis represents the number of iterations. The vertical axis represents the functional value.}
    \label{fig:energy monotone}
\end{figure}

\section{Conclusions}
\label{sec:conclusion}

In this paper, two numerical methods are derived for the simulation of longest minimal length partitions problem. The methods are based on approximations of the objective functional by a short time heat flow, threshold dynamics, and splitting methods to solve the relative constrained max-min optimization problem. In both methods, auction dynamics is used to find the shortest partitions, and the update of total region is based on threshold dynamics. The shortest partitions and total region are updated in an alternative manner. To reduce the influence of local minimal partitions, the first method repeats auction dynamics several times in each iteration, and the second method introduces a proximal regularization term, which is shown to be equivalent to a gradient ascent flow. Both methods update the total region by introducing an update step-length, which decreases gradually according to some criterion during iterations.
The effectiveness of both methods are demonstrated by comprehensive numerical experiments, which validate the conjecture that the disc is the minimizer for two-dimensional problems and the ball is the minimizer for three-dimensional problems. 

This paper focuses on the simulation to an approximate problem to provide the numerical evidence to the conjecture for the longest minimal length partitions problem. It would be interesting to consider the proof of the conjecture in the sense of the approximate objective functional and further consider the limit as $\tau$ goes to $0$.
It would be also interesting to use the level-set-based framework \cite{osher1988fronts,merriman1994motion,chan2004level,lie2005piecewise} to approximate such problems.

%Through comparing the energy figures and CPU times, it is obvious that the first method is more stable than the second method, but the second method is faster than the first method. 

%Since both methods are not unconditionally stable, further study may focus on how to design a unconditionally stable scheme for the problem. Why the energies in 3d cases are more stable than in 2d cases is also an interesting topic to study.



\section*{Acknowledgements}
H. Liu was  partially supported by National Natural Science Foundation of China  12201530, HKRGC ECS 22302123, HKBU 179356. D. Wang was partially supported by National Natural Science Foundation of China (Grant No. 12101524), Guangdong Basic and Applied Basic Research Foundation (Grant No. 2023A1515012199) and Shenzhen Science and Technology Innovation Program (Grant No. JCYJ20220530143803007, RCYX20221008092843046), Hetao Shenzhen-Hong Kong Science and Technology Innovation Cooperation Zone Project (No.HZQSWS-KCCYB-2024016) and the Guangdong Key Lab of Mathematical Foundations for Artificial Intelligence (2023B1212010001).

 \bibliographystyle{elsarticle-num} 
 \bibliography{refs}





% \end{thebibliography}
\end{document}